\pdfoutput=1

\documentclass[12pt]{article}
\usepackage[margin=1in]{geometry}
\usepackage{graphicx,amsmath,amssymb,amsthm,tikz,tikz-cd, mathrsfs} % Required for inserting images
\usepackage{stmaryrd} %for \sslash in use of GIT quotient
\usepackage{biblatex}
\usepackage{hyperref}
\usetikzlibrary{arrows.meta,positioning}
\usepackage{comment}
\bibliography{./ref}

\newcommand{\ns}{\mathrm{ns} \,}
\newcommand{\ver}{V}
\newcommand{\fundwt}{\varpi}
\newcommand{\qv}{\mathcal{M}}
\newcommand{\dv}{\mathsf{v}}
\newcommand{\dw}{\mathsf{w}}
\newcommand{\tbw}{\mathcal{W}}
\newcommand{\bA}{\mathsf{A}}
\newcommand{\bT}{\mathsf{T}}
\newcommand{\lb}{\mathscr{L}}
\newcommand{\tb}{\mathcal{V}}
\newcommand{\qmtb}{\mathscr{V}}

\newcommand{\qmtbw}{\mathscr{W}}
\newcommand{\qmb}{\mathscr{N}}
\newcommand{\qmtvir}{\mathscr{T}_{\mathrm{vir}}}
\newcommand{\qmpol}{\mathscr{T}^{1/2}}
\newcommand{\qm}{\mathsf{QM}}
\newcommand{\Cs}{\mathbb{C}^{\times}}
\newcommand{\C}{\mathbb{C}}
\newcommand{\exppref}{\mathbf{e}}
\newcommand{\Phipref}{\Phi}
\DeclareMathOperator{\im}{im}
\DeclareMathOperator{\vir}{vir}

\DeclareMathOperator{\ev}{ev}
\DeclareMathOperator{\eff}{Eff}
\newcommand{\vrs}{\hat{\mathcal{O}}_{\vir}}

\newcommand{\hunter}[1]{{\color{blue}Hunter: #1}}

\newcommand{\vasya}[1]{{\color{purple}Vasya: #1}}

\newcommand\iso{\,\vphantom{j^{X^2}}\smash{\overset{\sim}{\vphantom{\rule{0pt}{0.20em}}\smash{\longrightarrow}}}\,}

\theoremstyle{plain}
\newtheorem{fact}{Fact}[section]
\newtheorem{lemma}[fact]{Lemma}
\newtheorem{theorem}[fact]{Theorem}
\newtheorem{proposition}[fact]{Proposition}
\newtheorem{corollary}[fact]{Corollary}
\newtheorem{conjecture}[fact]{Conjecture}

\theoremstyle{definition}
\newtheorem{example}[fact]{Example}
\newtheorem{remark}[fact]{Remark}
\newtheorem{definition}[fact]{Definition}

\newtheorem{assumption}[fact]{Assumption}

\newcommand{\msver}{\mathsf{V}}
\newcommand{\sspt}{\star}
\newcommand{\slantsum}{\#}
\newcommand{\ssmap}{\Psi}
\DeclareMathOperator{\rank}{rank}

\DeclareMathOperator{\Hom}{Hom}
\DeclareMathOperator{\rep}{Rep}
\newcommand{\chamb}{\mathfrak{C}}
\newcommand{\extpow}{\bigwedge}

\title{Slant sums of quiver gauge theories}
\author{Hunter Dinkins, Vasily Krylov, and Reese Lance}
\date{
    \today
}

\begin{document}

\maketitle

\begin{abstract}
    We define the slant sum of quiver gauge theories, a gluing on the underlying quivers that identifies a gauge vertex with a framing vertex. Under some mild assumptions, we relate torus fixed points on the corresponding Higgs branches, which are Nakajima quiver varieties. Then we prove a ``branching rule" relating the quasimap vertex functions before and after a slant sum and deduce a number of ``factorization" corollaries.
    
    Our construction is motivated by a factorization conjecture for the vertex functions of zero-dimensional quiver varieties, which can be approached inductively using the branching rule. In special cases, it also shows that vertex functions can be written as sums over reverse plane partitions, even outside ADE type. 
    
    % When passed through the quantum Hikita conjecture, such expressions provide conjectural formulas for graded traces of Verma modules on the 3d mirror dual side.

    We make some conjectures for Coulomb branches reflecting what can be seen on the Higgs side and prove them in ADE type. In particular, we obtain refined character formulas for the so-called ``extremal'' irreducible modules over shifted Yangians.  We also study slant sums of Coulomb branches and their quantizations. We observe that for one-dimensional framing, the slant sum of Coulomb branches is the same as the product.

    %This allows to compute Gelfand-Tsetlin-characters of certain modules over quantized Coulomb branhces beyond the ADE type quivers. 
    %For general $\mathsf{w}$, we relate Gelfand-Tsetlin modules over quantized Coulomb branches before and after a slant sum. This should reflect our vertex function computation on the Higgs side. Our approach allows to compute Gelfand-Tsetlin-characters of certain modules over quantized Coulomb branhces beyond the ADE type quivers. 

\end{abstract}

\tableofcontents

\section{Introduction}

\subsection{Slant sums of heaps}

The \emph{slant sum} of two posets was defined by Proctor in \cite{proctor} to break certain posets into simpler, \emph{slant irreducible} pieces. Some particular posets to which this applies are those that appear as \emph{heaps} $H(w)$ of fully commutative elements $w$ of a Weyl group $W$, see \cite{stembridge}. Such posets are defined in a straightforward combinatorial way from a choice of reduced word for $w$.

It was noticed in \cite{heaps} that the heap of a minuscule element gives rise to a module over the preprojective algebra of the corresponding Dynkin quiver. Equivalently, it gives rise to a point in a \emph{Nakajima quiver variety}. When $w$ is \emph{dominant} minuscule, the corresponding quiver variety is a single point, and the preprojective algebra module studied in \cite{heaps} is a choice of representative of this point. 

Suppose $w$, $w^{(1)}$, and $w^{(2)}$ are dominant minuscule elements of some Weyl groups and let $\qv$, $\qv^{(1)}$, and $\qv^{(2)}$ be the respective zero-dimensional Nakajima quiver varieties. If the heap of $w$ decomposes as a slant sum of the other two heaps, written $H(w)=H(w^{(1)})\slantsum H(w^{(2)})$, then we call $\qv$ the slant sum of the two other quiver varieties. 

The aim of this paper is to revisit slant sums from the perspective of quiver varieties, or more precisely, quiver gauge theories. If the numerical data used to construct $\qv^{(1)}$ and $\qv^{(2)}$, which now need not be zero-dimensional, is \emph{compatible}, we define a new quiver variety $\qv:=\qv^{(1)}\slantsum \qv^{(2)}$ which we call the slant sum of the two constituents. Unlike in \cite{heaps}, our procedure quickly leaves the world of Dynkin quivers, which we view as a feature.

\subsection{Slant sums of quiver gauge theories}
Recall that a quiver gauge theory is specified by a choice of quiver $Q=(Q_{0},Q_{1})$ and $\dv,\dw \in \mathbb{N}^{Q_{0}}$. It is convenient to think of $\dv_{i}$ as associated to a gauge vertex $i \in Q_{0}$ and $\dw_{i}$ as associated to a framing vertex, as in the following picture:

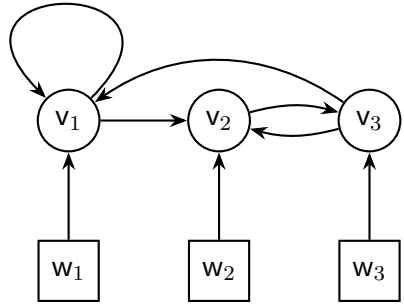
\[
\begin{tikzpicture}[->,>=Stealth,thick,node distance=2cm]

  % Gauge nodes (circles)
  \node[circle,draw,minimum size=0.8cm] (V1) {$\dv_1$};
  \node[circle,draw,minimum size=0.8cm,right of=V1] (V2) {$\dv_{2}$};
\node[circle,draw,minimum size=0.8cm,right of=V2] (V3) {$\dv_{3}$};

  % Framing nodes (squares)
  \node[rectangle,draw,minimum size=0.8cm,below of=V1] (W1) {$\dw_{1}$};
  \node[rectangle,draw,minimum size=0.8cm,below of=V2] (W2) {$\dw_{2}$};
  \node[rectangle,draw,minimum size=0.8cm,below of=V3] (W3) {$\dw_{3}$};

  % Arrows between gauge nodes
  \draw[->] (V1) to (V2);
  \draw[->,bend left=15] (V2) to (V3);
    \draw[->,bend left=15] (V3) to (V2);
      \draw[->,bend right=35] (V3) to (V1);
    \draw[->] (V1) to[out=45,in=135,looseness=10] (V1);

  % Framing arrows
  \draw[->] (W1) -- (V1);
  \draw[->] (W2) -- (V2);
\draw[->] (W3) -- (V3);
 
\end{tikzpicture}
\]

Let $Q^{(r)}$, $\dv^{(r)}$, and $\dw^{(r)}$ for $r\in \{1,2\}$ be two such collections. Suppose there exist vertices $\sspt_{1} \in Q^{(1)}_{0}$ and $\sspt_{2} \in Q^{(2)}_{0}$ such that $\dv^{(1)}_{\sspt_{1}}=\dw^{(2)}_{\sspt_{2}}$. We say that the two quiver gauge theories are slant summable, and we define their slant sum to be the quiver gauge theory defined by identifying the gauge vertex at $\sspt_{1}$ with the framing vertex at $\sspt_{2}$. We will use the symbol ${}_{\sspt_{1}}\slantsum_{\sspt_{2}}$, or just $\slantsum$, throughout to denote the effect of this operation on various objects. For example, the quiver for the slant sum of the two quiver gauge theories is $Q=Q^{(1)} {}_{\sspt_{1}}\slantsum_{\sspt_{2}} Q^{(2)}$. 

The local picture is
\[
\begin{tikzpicture}[->,>=Stealth,thick,node distance=1.8cm]

  \node[circle,draw,minimum size=0.8cm] (V1) {$\dv_{\sspt_{1}}$};
\node[right =0.2cm of V1] (ss){${}_{\sspt_{1}}\slantsum_{\sspt_{2}}$}; 
    \node[rectangle,draw,minimum size=0.8cm,right=1.5 of V1] (W2) {$\dw_{\sspt_{2}}$};
     \node[rectangle,draw,minimum size=0.8cm,below of=V1] (W1) {$\dw_{\sspt_{1}}$};
  \node[circle,draw,minimum size=0.8cm,right of=W2] (V2) {$\dv_{\sspt_{2}}$};

  \node[above left of=V1] (A){};
   \node[above right of=V2] (B){};
    \node[below right of=V2] (C){};
     \node[below left of=V1] (D){};

  % Framing arrows
  \draw[->] (W2) -- (V2);
  \draw[->] (W1) -- (V1);

% Right half
\node[right=0.8 of V2] (EE) {$=$};
\node[circle,draw,minimum size=0.8cm, right=0.7 of EE] (V3){$\dv_{\sspt_{1}}$};
\node[circle,draw,minimum size=0.8cm, right of=V3] (V4) {$\dv_{\sspt_{2}}$};
\node[rectangle,draw,minimum size=0.8cm,below=1 of V3] (W3) {$\dw_{\sspt_{1}}$};

 \node[above left of=V3] (E){};
   \node[above right of=V4] (F){};
    \node[below right of=V4] (G){};
     \node[below left of=V3] (H){};

 \draw[->] (W3) -- (V3);
  \draw[->] (V3) -- (V4);
 
\end{tikzpicture}
\]

Let $\qv$, $\qv^{(1)}$, and $\qv^{(2)}$ be the associated resolved Higgs branches (equivalently, the Nakajima quiver varieties). From the definitions, 
\[
\dim \qv= \dim \qv^{(1)}+\dim \qv^{(2)}
\]
and it is tempting to think of $\slantsum$ as a sort of product which preserves connectedness of the quivers. Actually, one interpretation of Corollary \ref{cor: CY vertex slant sum factorization} is that $\slantsum$ is closer to being 3d mirror dual to the product, see also Proposition \ref{prop_slant_sum_of_coulomb}.

In general, $\qv$, $\qv^{(1)}$, and $\qv^{(2)}$ are related by a certain diagram, see \eqref{eq: slant sum diagram}. But under some additional assumptions, we are able to define a map from certain torus fixed points on $\qv^{(1)} \times \qv^{(2)}$ to those on $\qv$. Let $\bT$, $\bT^{(1)}$, and $\bT^{(2)}$ be the tori acting on the quiver varieties. Let $\tb_{i}$ be the tautological bundle on $\qv$ for the vertex $i$. Similarly, let $\tb^{(1)}_{j}$ and $\tb^{(2)}_{k}$ be the tautological bundles on $\qv^{(1)}$ and $\qv^{(2)}$.

\begin{theorem}[Theorem \ref{thm: slant sum and fixed points}]\label{thm: slant sum and fixed points intro}
    Suppose that $p^{(1)}\in \left(\qv^{(1)}\right)^{\bT^{(1)}}$ is split over $\sspt_{1}$ in the sense of Definition \ref{def: split fixed points}. There is an embedding\footnote{Strictly speaking, this embedding depends mildly on a choice of ordering of some equivariant parameters, see Proposition \ref{prop: chamber dependence}.}
    \[
\ssmap_{p^{(1)}}: \left( \qv^{(2)}\right)^{\bT^{(2)}} \to \qv^{\bT}.
    \]
    Furthermore, there is an inclusion $\iota:\bT \hookrightarrow \bT^{(1)} \times \bT^{(2)}$ such that $\tb_{i}|_{p}=\iota^{*} \tb^{(r)}_{i}|_{p^{(r)}}$ for $i \in Q^{(r)}$, $r\in\{1,2\}$.
\end{theorem}

In the setting of the previous proposition, we will also denote $p^{(1)} \slantsum p^{(2)}:= \ssmap_{p^{(1)}}(p^{(2)})$. Theorem \ref{thm: slant sum and fixed points intro} generalizes the classical slant sum of heaps studied in \cite{proctor}, see also \cite{heaps}.

% When $\qv$, $\qv^{(1)}$, and $\qv^{(2)}$ are 0-dimensional quiver varieties corresponding to dominant minuscule Weyl group elements and $\dv^{(1)}_{\sspt_{1}}=\dw^{(2)}_{\sspt_{2}}=1$, Theorem \ref{thm: slant sum and fixed points intro} recovers the classical notion of slant sums of heaps studied in \cite{proctor}, see also \cite{heaps}. 

%{\vasya{the condition $\dw^{(2)}_{\sspt_{2}}=1$ will be automatic if we assume that the support of the second heap is connected, see Example \ref{Ex:minuscule} below?}}

\subsection{Slant sums and vertex functions}

Our constructions also relate to certain quasimap counts. In the $K$-theoretic enumerative geometry of quiver varieties, the \emph{descendant vertex function} is a key object \cite{pcmilect}. It appears in Okounkov's enumerative 3d mirror symmetry \cite{AOElliptic,Okounkov_video}, see also \cite{bottadink,dinkms2,dinkms1}. It is defined as the generating function of equivariant counts of quasimaps from $\mathbb{P}^{1}$ to $\qv$ with a nonsingular condition at $\infty \in \mathbb{P}^{1}$. We denote it by 
\[
\ver^{(\tau)}_{\qv}(z) \in K_{\bT \times \mathbb{C}^{\times}_{q}}(\qv)[[z]]
\]
Here, $[[z]]$ stands for power series in a certain cone in variables $z_{i}$ for $i \in Q_{0}$. The descendant $\tau$ is an element of $K_{ G_{\dv}}(\mathrm{pt})$ where $G_{\dv}$ is the gauge group. For a fixed point $p \in \qv^{\bT}$, we will denote $\ver^{(\tau)}_{p}:=\ver^{(\tau)}_{p}(z):=\ver^{(\tau)}_{\qv}(z)|_{p}$.

In the setting of Theorem \ref{thm: slant sum and fixed points intro}, there is a map between the spaces of descendants: 
\begin{equation}\label{eq: descendant map}
\iota^{*}:K_{\bT^{(1)} \times G_{\dv^{(1)}}}(\mathrm{pt})\otimes K_{\bT^{(2)} \times G_{\dv^{(2)}}}(\mathrm{pt}) \to K_{\bT \times G_{\dv}}(\mathrm{pt}),
\end{equation}
and it is natural to wonder if there is any relationship between $\ver^{(\tau)}_{p}$, $\ver^{(\tau_{1})}_{p^{(1)}}$, and $\ver^{(\tau_2)}_{p^{(2)}}$ when $\tau=\iota^{*}(\tau_1 \otimes \tau_2)$. This is the content of our next theorem.

To state it, let $\qm_{p^{(1)}}$ denote the moduli space of stable quasimaps from $\mathbb{P}^{1}$ to $\qv^{(1)}$ which send $\infty$ to $p^{(1)}$. There is an evaluation morphism $\ev_{0}:\qm_{p^{(1)}} \to K_{G_{\dv^{(1)}}}(\mathrm{pt})$. For a cocharacter $\sigma: \Cs \to \bA^{(2)} \subset \bT^{(2)}$ of the framing torus, let $\ver^{(\tau_2),\sigma}_{p^{(2)}}$ be the $\sigma$-twisted vertex function.

\begin{theorem}[Theorem \ref{thm: vertex and slant sum 1}]\label{thm: vertex and slant sums intro}
    Under the hypotheses of Theorem \ref{thm: slant sum and fixed points intro}, let $\tau=\iota^{*}(\tau_1 \otimes \tau_2) \in K_{\bT \times G_{\dv}}(\mathrm{pt})$. Then
    \[
\ver^{(\tau)}_{p}(z_1,z_2)=\sum_{F} \chi\left(F, \frac{(\ev_{0}^{*}(\tau_{1}) \otimes \vrs)|_{F}}{\extpow(N_{\vir}|_{F}^{\vee})}\right) z_1^{\deg F} \iota^{*} \ver^{(\tau_{2}),\sigma_{F}}_{p^{(2)}}(z_2) 
    \]
    where the sum runs over $\bT^{(1)} \times \Cs_{q}$-fixed components $F$ of $\qm_{p^{(1)}}$.
\end{theorem}

We prove this theorem by using equivariant localization, relating torus fixed quasimaps to $\qv$ with torus fixed quasimaps to $\qv^{(1)}$ and (twisted) quasimaps to $\qv^{(2)}$. Theorem \ref{thm: vertex and slant sums intro} can be thought of as a ``branching rule" for vertex functions, describing a vertex function in terms of vertex functions for smaller quivers.

We obtain several corollaries in Section \ref{sec: vertex and slant sums} in which Theorem \ref{thm: vertex and slant sum 1} can be made more explicit. 

\begin{corollary}[Corollary \ref{cor: vertex slant sum factorization}]
    If $\qv^{(2)}=p^{(2)}$ and $\tau_2=1$, or more generally under the assumptions of Corollary \ref{cor: vertex slant sum factorization}, we have
    \[
\ver^{(\tau_1\otimes \tau_{2})}_{p}(z_1,z_2)=\ver^{(\tau_{1})}_{p^{(1)}}(z_{1}') \ver^{(\tau_2)}_{p^{(2)}}(z_2)
    \]
    where $z_1'$ stands for a shift of $z_1$ by some powers of $z_2$.
\end{corollary}
% In the special case when $\qv^{(2)}$ is zero-dimensional, $\ver^{(1)}_{\qv^{(2)}}$ will not depend on the so-called framing parameters. It follows that the twisted vertex is equal to the ordinary vertex up to a monomial in $z_2$, and we obtain the factorization
% \[
% \ver^{(\tau_1\otimes 1)}_{p}(z_1,z_2)=\ver^{(\tau_{1})}_{p^{(1)}}(z_1') \ver^{(1)}_{p^{(2)}}(z_2)
% \]
% where $z_1'$ stands for a shift of $z_1$ by some powers of $z_2$, see Corollary \ref{cor: vertex slant sum factorization}.

Let $\hbar$ be the weight of the symplectic form on $\qv$ and let $q$ be the equivariant parameters for the action of $\Cs_{q}$ on $\mathbb{P}^{1}$. Under the \emph{Calabi-Yau} specialization $\hbar=q$, Theorem \ref{thm: vertex and slant sums intro} significantly simplifies.

\begin{corollary}[Corollary \ref{cor: CY vertex slant sum factorization}]\label{cor: CY vertex slant sum factorization intro}
    Under the assumptions of Corollary \ref{cor: CY vertex slant sum factorization}
      \[
\ver^{(\tau_1 \otimes \tau_{2})}_{p}(z_1,z_2)\big|_{q=\hbar}=\ver^{(\tau_{1})}_{p^{(1)}}(z_1')\big|_{q=\hbar} \ver^{(\tau_{2})}_{p^{(2)}}(z_2)\big|_{q=\hbar}
    \]
\end{corollary}

Theorem \ref{thm: vertex and slant sums intro} also subsumes an additional formula present in the literature, the branching rule for the ``nonstationary Ruijsenaars function", studied, for example, in \cite{BFS, LMS,NSmac,tamagninonabelian}. We explain the details in Section \ref{sec: ruijsenaars branching}.

\subsection{Factorization for limits of vertex functions}

Theorem \ref{thm: vertex and slant sums intro} was motivated by 3d mirror symmetry, a proposed duality between Higgs and Coulomb branches \cite{HW, IntSei,kamnitzersurvey, WebsterYoo}. Let $\bA \subset \bT$ be the subtorus preserving the symplectic form and let $\nu: \Cs \to \bA$. The cocharacter $\nu$ defines a partial resolution $\qv^{!}$ of the corresponding Coulomb branch \cite{BFNII}. The GIT stability parameter $\theta$ used in the definition of $\qv$ gives a cocharacter of a torus acting on $\qv^{!}$. 3d mirror symmetry predicts that isolated points in $\qv^{\nu(\Cs)}$ are in bijection with nonsingular points in $(\qv^{!})^{\theta(\Cs)}$.

A conjecture of Smirnov and the first author \cite{dinksmir} claims that $\lim_{a \to 0} \nu^{*}(\ver_{p})$, where $a$ is the coordinate on $\Cs$, is a product of $q$-binomial series, one for each repelling weight of $T_{p^{!}} \qv^{!}$. This implies a similar statement for the more extreme Calabi-Yau specialization $\ver_{p}|_{q=\hbar}$, in which $q$-binomials are replaced by ordinary binomials. Both of these are degenerations of a stronger statement, proven for type $A$ bow varieties in \cite{bottadink}, relating vertex functions of $\qv$ with those of $\qv^{!}$. Conveniently, the degenerations can be checked even when there is no proposed definition of the vertex function of $\qv^{!}$, which is currently the case outside in type $A$.

%\vasya{do we mention $a \rightarrow 0$ limit somewhere else in the paper?} \hunter{No. I just put this here for motivation, to explain why Conjecture \ref{conj: 0 dimensional vertex} really knows about quiver varieties which are not single points.}\vasya{ok, got it}

This conjecture reduces to the case where $\qv$ is a single point as follows. An isolated fixed point $p \in \qv^{\nu(\Cs)}$ can be viewed as the Higgs branch of its own quiver gauge theory, in which case $\lim_{a \to 0} \nu^{*}(\ver_{p})$ is equal to the vertex function of this zero-dimensional Higgs branch. Assuming a conjecture of Hilburn \cite{KPW}, Corollary \ref{cor: iso of tangent spaces} explains the similar reduction for Coulomb branches, which we verify for finite-type Dynkin quivers at the level of tangent space characters in \eqref{eq: tangent resolved slice} (generalizing \cite[Lemma 4.4]{KamnitzerTingleyWebsterWeeksYacobi} and \cite[Equation (5.8)]{KP}).

\subsection{Factorization for 0-dimensional vertex functions}

 % In this case, we make two conjectures refining the conjecture of \cite{dinksmir}, providing an explicit formula for the factorization and, equivalently, for the character of the tangent space of $\qv^{!}$.

Now fix a quiver $Q$ without edge loops and a zero-dimensional quiver variety $\qv:=\qv_{Q,\theta}(\dv,\dw)$ where the stabilty parameter is $\theta=(1,1,\ldots,1)$. We define coweights
\begin{equation}\label{eq: lambda and mu general}
\lambda:=\sum_{i \in Q_{0}} \dw_{i} \fundwt_{i}^{\vee}, \quad \mu:=\lambda-\sum_{i \in Q_{0}} \dv_{i} \alpha_{i}^{\vee}
\end{equation}
of the Kac-Moody Lie algebra $\mathfrak{g}_Q$ associated to the quiver $Q$. Here, $\fundwt_{i}^{\vee}$ and $\alpha_{i}^{\vee}$ denote the fundamental coweights and simple coroots, respectively. Let $\Phi^{\pm, \text{re}}$ be the set of real positive or negative roots. Let $\Phi^{-,\text{re}}_{\mu}=\{\alpha \in \Phi^{-,\text{re}} \, \mid \, \langle \alpha,\mu\rangle>0 \}$. It is a standard fact, see Lemma \ref{lem: nakajima quiver point ADE}, that $\dim \qv =0 \iff \mu \in W \lambda$.

The vertex function depends on a choice of polarization, and we choose the canonical polarization\footnote{That is determined by the choice of orientation of the edges in the underlying graph of the quiver. Equivalently, we choose the polarization induced by the zero section of the cotangent bundle appearing in the symplectic reduction defining $\qv$.}. Then we define 
\begin{equation}\label{eq: msver}
\msver_{\qv}(z)=\ver_{\qv}(z)_{z_i=z_{i}(-\hbar^{-1/2})^{a_{i}}}, \quad \text{where} \quad  a_{i}=\sum_{\substack{e \\  h(e)=i}} \dv_{t(e)} -\sum_{\substack{e \\ t(e)=i}} \dv_{h(e)} + \dw_{i}
\end{equation}

Let $A_{Q}=(a_{i,j})_{i,j \in Q_{0}}$ be the adjacency matrix of Q, i.e., $a_{i,j}$ is the number of arrows from vertex $i$ to vertex $j$. We identify the K\"ahler parameters with formal exponentials of simple roots via $e^{\alpha_i}=(\hbar/q)^{(\dv-A_{Q} \cdot \dv )_{i}}z_{i}$. More explicitly,
\[
e^{\alpha_{i}}=z_{i} (\hbar/q)^{b_{i}}, \quad b_{i}=\dv_{i}-\sum_{\substack{e \\ t(e)=i}} \dv_{h(e)}
\]
Let $\Phi(x):=\prod_{i=0}^{\infty}(1-x q^{i})$. As discussed in \cite{AOElliptic} and \cite{bottadink}, the normalization of the vertex $\msver_{\qv}(z^{-1})$ is the most natural for 3d mirror symmetry.

\begin{conjecture}\label{conj: 0 dimensional vertex}
The vertex function of $\qv$ factorizes
\begin{equation}\label{eq: conj zero dim vertex}
\msver_{\qv}(z^{-1})=\prod_{\substack{\alpha \in \Phi^{-,\text{re}}_{\mu}\\}}\prod_{i=1}^{\langle\alpha,\mu\rangle} \frac{\Phi\left(\hbar \left(\frac{\hbar}{q}\right)^{i-1} e^{\alpha}\right)}{\Phi\left(\left(\frac{\hbar}{q}\right)^{i-1} e^{\alpha}\right)}
\end{equation}
   
\end{conjecture}

In type $A$, this conjecture follows from \cite{bottadink}, though more direct proofs are available in special cases, see \cite{dinkms1} and \cite{dinksmir2}. It was also proven by Jang and the first author for type $D$ quiver varieties with a single minuscule framing \cite{dinkjang}.

In \S\ref{sec: 0 dim varieties}, and specifically in Proposition \ref{prop: inductive factorization}, we demonstrate how Theorem \ref{thm: vertex and slant sums intro} can be used to inductively prove instances of Conjecture \ref{conj: 0 dimensional vertex}.

\begin{remark}
    While this paper was in the final stages of preparation, we learned about a work in-progress by H. Nakajima, A. Okounkov, and Z. Zhou which contains a proof of Conjecture \ref{conj: 0 dimensional vertex} for ADE quiver varieties \cite{ONZ}. The slant sum of two Dynkin quivers is almost never a Dynkin quiver; so the main interest of the techniques developed here lie outside ADE types.
\end{remark}

\subsection{Duals to zero-dimensional quiver varieties}

Recall that $\qv$ is a single point if and only if $\mu \in W \lambda$. Combining \eqref{eq: conj zero dim vertex} with Conjecture 1 from \cite{dinksmir}, and recalling that mirror symmetry identifies $\hbar^{!}=q/\hbar$, we obtain the following:

\begin{conjecture}\label{conj: dual tangent space}
If $\mu \in W \lambda$, then the unique torus fixed point $p^{!} \in \qv^{!}$ is nonsingular, and the character of the tangent space of $\qv^{!}$ at $p^{!}$ is
   \begin{equation}\label{eq: intro tangent formula}
\sum_{\alpha \in \Phi^{-,\mathrm{re}}_\mu} \sum_{i=1}^{\langle \alpha,\mu\rangle} (\hbar^!)^{1-i}e^{\alpha}+(\hbar^!)^ie^{-\alpha}.
\end{equation}
\end{conjecture}

We provide some numerical evidence for Conjecture \ref{conj: dual tangent space} in Lemma \ref{lem: conj 1.4 at the level of dimensions}.

By \cite{NW}, the Coulomb branch is well-defined for quivers encoding {\emph{symmetrizable}} Cartan matrices. We thus see that Conjecture~\ref{conj: dual tangent space} makes sense more generally for any quiver $Q$ such that $\mathfrak{g}_Q$ is symmetrizable. Here it becomes crucial that we treat $\mu$ as a coweight (instead of a weight as is often the case) and the $\alpha$ as roots in (\ref{eq: intro tangent formula}), equipped with the natural pairing between them.

Using the identification of Coulomb branches for finite Dynkin quivers with affine Grassmannian slices, we will prove in Section~\ref{subsec: proof for ADE} the following.

\begin{proposition}[Proposition \ref{prop: tangent for ade}]\label{prop_fin_Dyn_tangent}
    Conjecture \ref{conj: dual tangent space} holds for finite type Dynkin quivers.
\end{proposition}
Our argument for the proof of Proposition \ref{prop_fin_Dyn_tangent} is obtained by combining the argument similar to the one in \cite[Lemma 4.4]{KamnitzerTingleyWebsterWeeksYacobi} with \cite[Theorem 3.1(1)]{Kr}.

\subsection{Modules over quantized Coulomb branches}
Assume that Conjecture~\ref{conj: dual tangent space} holds.
In Section~\ref{section: character of simple}, we observe that the category $\mathcal{O}$ for a quantization of $\mathcal{M}^!$ contains a \emph{unique} irreducible object, and we show that its character can be read off from the explicit formula for the tangent space to the Coulomb branch at a fixed point (see Proposition~\ref{prop: char irreducible}). Our main technical tool is the $D$-module of graded traces introduced in \cite{KMP} and further studied in \cite{DKK}.

Restricting to finite-type Dynkin quivers, and using that Conjecture~\ref{conj: dual tangent space} holds in this case (see Proposition~\ref{prop_fin_Dyn_tangent} above), we obtain explicit formulas for the characters of so-called  ``extremal'' (terminology is due to \cite{min_chamber_mod}) irreducible modules $L$ over shifted Yangians $Y_\mu$. Namely, we get (see Proposition \ref{prop: char irreducible}):
\begin{equation}\label{eq:norm_char_form_intro}
\widetilde{\operatorname{ch}}(L) = \prod_{\alpha \in \Phi_\mu^{-}}\frac{1}{(1-e^\alpha)^{\langle \alpha,\mu\rangle}},
\end{equation}
where $\widetilde{\operatorname{ch}}(L)$ is the {\emph{normalized}} character of $L$.

Using \cite[Corollary 1.2.1]{VV}, these formulas may be translated to the case of shifted quantum groups. A closely related result was obtained by Negut by completely different methods (see \cite[Theorem 1.3]{Negut_char}). The authors of the forthcoming paper \cite{min_chamber_mod} also obtain the same formula for the extremal $Y_\mu$-modules; their argument combines \cite[Theorem 3.1(1)]{Kr} with the theory developed in \cite{catO_coulomb}.

Finally, the character formula (\ref{eq:norm_char_form_intro}) we obtain admits a natural one-parameter deformation coming from the parameter $\hbar^!$ in (\ref{eq: intro tangent formula}) above. This refined character agrees with that in \cite[Theorem 1.3]{Negut_char}. For us, the parameter $\hbar^!$ corresponds to the $\mathbb{C}^\times$-action on $\mathcal{M}^!$ induced by the homological grading on $\mathbb{C}[\mathcal{M}^!]$.

%{\vasya{ask Joel -- we reprove Negut's formulas from 2025 paper by completely different methods?}}

%{\vasya{maybe  delete or shorten this}} More generally, we believe that Conjecture \ref{conj: dual tangent space} should provide a formula for normalized characters of a natural family of modules over Coulomb branches labeled by nonsingular fixed points of their partial resolutions, see Remark \ref{eq: prefund building blocks} for a speculation about this. An additional 

\subsection{Connection with quantum Hikita conjecture}

% As another special case, we consider the $q=\hbar$ specialization of Theorem \ref{thm: vertex and slant sums intro}, where $\hbar$ is the weight of the symplectic form of $\qv$ and $q$ is the equivariant parameter for the action of $\Cs$ on the source $\mathbb{P}^{1}$ of quasimaps. For certain descendants, this again forces each twisted vertex in Theorem \ref{thm: vertex and slant sums intro} to be equal to the untwisted vertex up to a monomial in $z_2$. If we are lucky, this monomial can be absorbed into $z_1$, leading to another factorization
% \begin{equation}\label{eq: CY vertex intro}
% \ver^{(\tau_1\otimes \tau_{2})}_{p}(z_1,z_2)\big|_{q=\hbar}=\ver^{(\tau_{1})}_{p^{(1)}}(z_1')\big|_{q=\hbar} \ver^{(\tau_{2})}_{p^{(2)}}(z_2)\big|_{q=\hbar},
% \end{equation}
% see Corollary \ref{cor: CY vertex slant sum factorization} for the precise statement. 

Corollary \ref{cor: CY vertex slant sum factorization intro} is related to another manifestation of 3d mirror symmetry: the quantum Hikita conjecture proposed in \cite{KMP}. This conjecture predicts an isomorphism between the $q=\hbar$ specialized quantum $D$-module for the Higgs branch and the $D$-module of graded traces for the Coulomb branch.  

A proof of this conjecture was given for ADE quiver varieties with minuscule framings in \cite{DKK}. The technical heart of the proof is to identify the $q=\hbar$ specialized descendant vertex functions with the graded traces of elements of the quantized Coulomb branch--in this case a \emph{truncated shifted Yangian}--over a Verma module. Thus Corollary \ref{cor: CY vertex slant sum factorization intro} provides many cases of the Higgs side of the computation needed to prove the quantum Hikita for more general quiver varieties. In many examples, see Section \ref{sec: 0 dim varieties}, vertex functions can be computed in terms of reverse plane partitions over some poset. Theorem \ref{thm: vertex and slant sums intro} shows that this is true for $\ver^{(\tau)}_{p}$ if it is true for $\ver^{(\tau_1)}_{p^{(1)}}$ and $\ver^{(\tau_2)}_{p^{(2)}}$. %suggesting the explicit formula for the Gelfand-Tseitlin-character of the simple \vasya{will finish this} 

Furthermore, passing Corollary \ref{cor: CY vertex slant sum factorization intro}, or more generally the $q=\hbar$ specialization of Theorem \ref{thm: vertex and slant sums intro}, through 3d mirror symmetry gives a conjectural ``branching rule" for graded traces. To the best of the authors' knowledge, this formula, and its expression in terms of reverse plane partitions in special cases, is a new expectation.

%\vasya{not to forget -- say that generalize prefundamental module}

%\vasya{$\Theta$-module vs prefundamental mention simple}

\subsection{Slant sums of Coulomb branches}

In Section \ref{subsec:slant via Coulomb}, we initiate the study of the slant sum construction from the Coulomb branch perspective. We prove in Proposition \ref{prop_slant_sum_of_coulomb} that for $\mathsf{w}^{(2)}$ such that $\dw^{(2)}_{\star_2}=1$ and all other $\dw_i^{(2)}=0$, the slant sum on the Coulomb branch side corresponds to taking products. The Higgs branch counterpart of this statement was observed in 
\cite[Lemma 8.4]{sing_CM_varieties}
 (after applying the Crawley-Boevey trick). We note that the identification of a Coulomb branch $\mathcal{M}_{C}$ for the quiver $Q=Q^{(1)} {}_{\sspt_{1}}\slantsum_{\sspt_{2}} Q^{(2)}$ with the product $\mathcal{M}_C^{(1)} \times \mathcal{M}_C^{(2)}$ of Coulomb branches for $Q^{(1)}, Q^{(2)}$ is compatible with  natural structures (integrable system, torus action) after a certain twist.

\subsection{Outline}
In Section \ref{sec: quiver varieties}, we establish notation and review the facts we need about quiver varieties. We also clearly state the exact hypotheses needed in Theorem \ref{thm: slant sum and fixed points intro} and prove this theorem. 

In Section \ref{sec: quasimaps}, we review vertex functions, prove Theorem \ref{thm: vertex and slant sums intro}, and discuss a number of factorization corollaries. 

In Section \ref{sec: applications}, we consider some examples which show how Theorem \ref{thm: vertex and slant sums intro} relates to Conjecture \ref{conj: 0 dimensional vertex}. We also explicitly spell out the connection between our results the branching of the non-stationary Ruijsenaars function of \cite{NSmac}. 

In Section \ref{subsec: Higgs and Coulomb branches}, we recall some facts about Higgs and Coulomb branches and their mirror symmetry. We make some conjectures needed for Conjecture \ref{conj: dual tangent space} to make sense. We also make some conjectures which reduce the character of tangent spaces at nonsingular fixed points for an arbitrary partial resolution of a Coulomb branch to the case of Conjecture \ref{conj: dual tangent space}.

In Section \ref{section: character of simple} we prove a general result describing the normalized character of irreducible modules over certain Coulomb branches in presence of Conjecture \ref{conj: dual tangent space} above. In particular, we obtain a refined
character formulas for so-called ``extremal'' irreducible modules over shifted Yangians. 

In Section \ref{subsec: proof for ADE}, we prove all of the conjectures above for finite type Dynkin quivers. 

Finally, in Section~\ref{subsec:slant via Coulomb}, we investigate slant sums from the perspective of Coulomb branches.

%In Section \ref{subsec:slant via Coulomb}, we investigate slant sums from the perspective of Coulomb branches.

%{\vasya{maybe will delete some of the conjecture just not to overwhelm the reader}}

\subsection{Acknowledgements}

We are grateful to Jiwu Jang, Shrawan Kumar and Zijun Zhou for helpful conversations regarding this project. We are grateful to Ivan Losev and Dmytro Matvieievskyi for helpful discussions on the material of Section \ref{section: character of simple}.
We are grateful to Joel Kamnitzer,  Alexis Leroux-Lapierre, and Hiraku Nakajima for valuable
comments on an earlier version of this text.
H. Dinkins was supported by NSF grant DMS-2303286 at MIT and the NSF RTG grant Algebraic Geometry and Representation Theory at Northeastern University DMS–1645877. R. Lance was supported by NSF grant DMS-2200867. Vasily Krylov was supported by the Simons
Foundation Award 888988 as part of the Simons Collaboration on Global Categorical Symmetries.

\section{Slant sums of Higgs branches}\label{sec: quiver varieties}
    
    We briefly recall definitions, theorems, and fix conventions regarding quiver varieties. A more detailed introduction can be found in \cite{Nak1}. 
    
    A \textit{quiver}, $Q$, is a directed graph. It comes equipped with a finite vertex set, $Q_0$ and a set of directed edges $Q_1$. There are maps $h,t:Q_{1} \to Q_{0}$ sending an arrow to its head and tail. A \textit{representation} of the quiver $Q=(Q_0,Q_1)$ of dimension $\dv \in \mathbb{N}^{Q_{0}}$ is a collection of complex vector spaces $V_i$ with $\dim V_i= \dv_i$ for each $i \in Q_{0}$ and linear maps $x_e: V_{t(e)} \to V_{h(e)}$ for each arrow $e \in Q_{1}$.

Choose $\dv,\dw \in \mathbb{N}^{Q_{0}}$ and fix a pair of $Q_0$-graded vector spaces $V=\bigoplus_{i\in Q_0} V_i, W=\bigoplus_{i\in Q_0}W_i$ such that $\dv_i=\dim V_i$ and $\dw_i=\dim W_i$. Let
\begin{align*} {\bf{M}}:={\bf{M}}(\dv,\dw):=\Bigg(\bigoplus_{e \in Q_{1}}\Hom(V_{t(e)},V_{h(e)})\Bigg)&\oplus\Bigg(\bigoplus_{e \in Q_{1}}\Hom(V_{h(e)},V_{t(e)})\Bigg)\\
\oplus\Bigg(\bigoplus_{i\in Q_0}\Hom(W_i,V_i)\Bigg)&\oplus\Bigg(\bigoplus_{i\in Q_0}\Hom(V_i,W_i)\Bigg)
\end{align*}
which, by the trace pairing, may be identified with $T^*{\bf{N}}$, where
\begin{equation}\label{eq: def of N}
{\bf{N}}:={\bf{N}}(\dv,\dw) := \bigoplus_{e \in Q_{1}}\Hom(V_{t(e)},V_{h(e)})\oplus\bigoplus_{i\in Q_0}\Hom(V_i,W_i)
\end{equation}

We denote an element of ${\bf{M}}$ as $(X,Y,I,J)$, where 
\begin{align*}
X = \bigoplus_{e\in Q_1} X_e, \quad
Y = \bigoplus_{e\in Q_1} Y_e, \quad 
I = \bigoplus_{i \in Q_0}I_i, \quad 
J = \bigoplus_{i \in Q_0}J_i
\end{align*}

% Elements of ${\bf{M}}$ may be viewed as representations of the framed and doubled quiver. Locally,

% \[
% \begin{tikzcd}[sep=large]
% \bullet \arrow[r, bend right] \arrow[d, bend left] & \bullet \arrow[l, bend right] & \rightsquigarrow & V_i \arrow[r, "X", bend left] \arrow[d, "J", bend left] & V_j \arrow[l, "Y", bend left] \\
% \square \arrow[u, bend left]                       &                               &                  & W_i \arrow[u, "I", bend left]                           &                              
% \end{tikzcd}
% \]
As the cotangent bundle to a smooth manifold, ${\bf{M}}$ comes with a canonical symplectic form.

${\bf{M}}$ has an action of $G_{\dv}:=\prod_{i\in Q_0}GL(V_i)$, sometimes referred to as the gauge group, induced via basis change at each vertex: 
\[
g.(X,Y,I,J) = (gXg^{-1}, gYg^{-1},gI,Jg^{-1})
\]
We denote $\text{Lie}(G_{\dv}):=\mathfrak{g}_{\dv}$.

The action of $G_{\dv}$ on ${\bf{M}}$ is Hamiltonian, and the moment map 
\[
\mu: {\bf{M}} \to \mathfrak{g}_{\dv}
\]
is given by 
\[
(X,Y,I,J) \mapsto \bigoplus_{i\in Q_0}\left(\sum_{\substack{e\in Q_1\\h(e)=i}}X_eY_e-\sum_{\substack{e\in Q_1\\t(e)=i}} Y_eX_e +I_iJ_i\right)
\]
where we have identified $\mathfrak{g}_{\dv}\cong\mathfrak{g}_{\dv}^{*}$ via the trace pairing.

A choice of stability parameter $\theta \in \mathbb Z^{Q_0}$ induces a character $\chi_\theta$ of $G_{\dv}$ by $\chi_\theta(g)=\prod_{i \in Q_{0}}(\det g_i)^{\theta_i}$.

\begin{definition}\label{qv}
The \emph{Nakajima quiver variety} associated to the data $Q$, $V_{i}$, $W_{i}$, and $\theta$ is the algebraic symplectic reduction
\[
\qv_{Q,\theta}(\dv,\dw)={\bf{M}} /\!\!/\!\!/_{\chi_{\theta}} G_{\dv}=\mu^{-1}(0)/\!\!/_{\chi_\theta} G_{\dv}
\]
\end{definition}

We may omit $Q,\theta$ when they are understood. By the general theory of GIT, $\qv_{Q,\theta}(\dv, \dw)$ is a quasiprojective variety and admits a projective morphism 
\[
\pi: \qv_{Q,\theta}(\dv,\dw)\to \qv_{Q,0}(\dv,\dw)
\]
to the affine categorical quotient $\qv_{Q,0}(\dv,\dw)=\mu^{-1}(0)/\!\!/ G_{\dv}$.

\begin{proposition}[\cite{NakALE}]\label{prop: generic theta}
Fix $Q$, $\dv$, and $\dw$. There is a finite set of hyperplanes in $\mathbb{Z}^{Q_{0}} \otimes_{\mathbb{Z}} \mathbb{R}$ such that if $\theta$ is not on a hyperplane, then $\qv_{Q,\theta}(\dv,\dw)$ is nonsingular, symplectic, and connected.
\end{proposition}

We will call stability parameters as in the previous proposition as ``generic". For this paper it suffices to know that $\theta=\pm(1,\dots,1)$ is always generic. 

Furthermore, GIT provides a notion of $\theta$ (semi)stable points. It is known that for generic $\theta$, stability and semistability are equivalent. Thus the closed points of $\qv_{Q,\theta}(\dv,\dw)$ are in bijection with orbits $\mu^{-1}(0)^{\theta-s}/G_{\dv}$. The superscript here refers to the intersection of $\mu^{-1}(0)$ with the locus of $\theta$-stable points. The following proposition gives a criteria for stability.

\begin{proposition}[\cite{Nak1}]\label{prop: rep-theory-stability}
Fix $\theta\in \mathbb{Z}^{Q_0}$. A representation $V=(X,Y,I,J)\in {\bf{M}}$ is \emph{$\theta$-semistable} if for any proper nonzero subrepresentation, $V'$, we have
\begin{align*}
 (V'\subset \ker J) &\Rightarrow \theta\cdot \dim V'\geq0\\
 (V'\supset \im I) &\Rightarrow \theta\cdot \dim V' \geq \theta \cdot \dim V
\end{align*}
and \emph{$\theta$-stable} if both implications hold with strict inequalities. 
\end{proposition}

It follows from Proposition \ref{prop: rep-theory-stability} that all stability conditions $\theta$ such that $\theta_{i}>0$ for all $i\in Q_{0}$ are equivalent to each other. We will denote such a stability condition by $\theta>0$. Similarly, we have stability conditions such that $\theta<0$. 

Under some assumptions, $\pi:\qv_{Q,\theta}(\dv,\dw)\to \qv_{Q,0}(\dv,\dw)$ is a symplectic resolution of singularities of its image.

When the data $Q$, $\dv$, $\dw$, and $\theta$ is fixed, we will often save on notation and denote $\qv:=\qv_{Q,\theta}(\dv,\dw)$.

\subsection{Torus action and fixed points}\label{sec: torus and fixed points}

Fix $Q$, $\dv$, $\dw$, and $\theta$. The vector space ${\bf{M}}$ has actions of the \emph{framing torus} $\bA:=(\Cs)^{|\dw|}$. There is another action of a rank one torus, denoted $\Cs_{\hbar}$ by scaling the cotangent fibers with weight $1$. In formulas,
\[
(a,\hbar).(X,Y,I,J) = (X,\hbar^{-1} Y, Ia^{-1},\hbar^{-1} aJ), \quad (a,\hbar) \in \bA \times \Cs_{\hbar}
\]
Both actions descend to $\qv$, and we denote $\bT:=\bA\times \C^\times_{\hbar}$. We will be interested in the $\bT$-fixed locus of $\qv$. 

The vector spaces $V_{i}$ descend to tautological bundles $\tb_{i}$ over $\qv$, defined by 
\[
\tb_{i}=(\mu^{-1}(0)\times V_{i})/G_{\dv}
\]
where $g\in G_{\dv}$ acts on $(p,v)\in \mu^{-1}(0)\times V_{i}$ by $g \cdot (p,v)=(g^{-1} p, g_{i}^{-1} v)$. Similarly, there are topologically trivial tautological bundles $\tbw_i$ induced by the framing spaces. Both types of bundles are $\bT$-equivariant.

When $p\in \qv^\bT$, the vector space $\tb_i|_p$ is a $\bT$-module, and thus may be decomposed into $\bT$ weight spaces. 

Let $p \in \qv^\bT$ and let $[(X,Y,I,J)]$ be a representative of $p$. Since $p$ is $\bT$-fixed, there is a homomorphism $\rho: \bT \to G_{\dv}$ such that
\[
(a,\hbar) \cdot (X,Y,I,J)= \rho(a,\hbar) \cdot (X,Y,I,J)
\]
Written explicitly,
\[
(X,\hbar^{-1} Y,I a^{-1},\hbar^{-1} a J)= (g X g^{-1},g Y g^{-1},g I,J g^{-1}), \quad g=\rho(a,\hbar)
\]
So $\bT$ acts on each $V_{i}$ space via $\rho$. Let $a_{i,j}$ for $i \in Q_{0}$ and $1 \leq j \leq \dw_{i}$ be the $\bA$-weights of the framing spaces $W_{i}$. The decomposition of $V_{i}$ into $\bT$-weight spaces takes the form 
\[
V_{i}=\bigoplus_{j \in Q_{0}} \bigoplus_{k=1}^{\dw_{j}} \bigoplus_{l \in \mathbb{Z}} V_i(a_{j,k}^{\pm} \hbar^{l})
\]
where $V_i(a_{j,k} \hbar^{l})$ denotes the (possibly empty) space of weight $a_{j,k} \hbar^{l}$ of $V_i$.

% , then in a representative $V$ of $p$, $V$ and $W$ have decompositions into $\bT$-weight spaces
% \[
% V_i=\bigoplus_{j=1}^n\bigoplus_{\ell\in\mathbb Z}V(a_{ij}\hbar^\ell)
% \]
% and
% \[
% W_i = \bigoplus_{j=1}^{\dw_i}W(a_{ij})
% \]
% The $W$-weight spaces are 1-dimensional, and do not have $\hbar$-weights. Because $p\in \qv^{\bT}$, it must be that
% \[
% (a,\hbar).\big[(X,Y,I,J)\big]=\big[(X,Y,I,J)\big]
% \]
% Here square brackets indicate a $GL(V)$-equivalence class. Then there exists some $g \in GL(V)$ such that 
% \begin{align*}
% (X,\hbar Y, Ia^{-1}, \hbar a J) &= g.(X,Y,I,J)\\
% &=\big(gXg^{-1},gYg^{-1},gI,Jg^{-1}\big)
% \end{align*}
From the equation $gI=Ia^{-1}$, it follows that $I(W_i(a_{i,j}))\subset V(a_{ij}^{-1})$. Similarly, the equation $Jg^{-1}=\hbar^{-1} a J$ implies $J(V_{i}(w))\subset W_i(\hbar w^{-1})$. Furthermore, for any $e \in Q_{1}$, we have $X(V_{t(e)}(w)) \subset V_{h(e)}(w)$ and $Y(V_{h(e)}(w)) \subset V_{t(e)}(\hbar^{-1} w)$.

Note also that the $\bT$-weights of $\tb_{i}|_{p}$ are the inverse of the weights of $V_{i}$ under the $\bT$-action by $\rho$ for the fixed point $p$.

We will need the following lemmas characterizing representatives of equivalence classes of $\bT$-fixed stable representations. They utilize the computation and notation from above. 

\begin{lemma}\label{lem: wght-spces}
Consider a quiver variety $\qv:=\qv_{Q,\theta}(\dv,\dw)$, and let $p\in\qv^{\bT}$. Let $V$ be a representative of $p$. Then
\begin{enumerate}
\item If $\theta<0$, then $V_i(a_{j,k}^{b} \hbar^{c})=0$ unless $b=-1$ and $c\geq 1$ for all $i \in Q_{0}$.

\item If $\theta>0$, then $V_i(a_{j,k}^{b}\hbar^{c})=0$ unless $b=-1$ and $c \leq 0$ for all $i \in Q_{0}$.. 
\end{enumerate}
\end{lemma}
\begin{proof}

By Proposition \ref{prop: rep-theory-stability}, if $\theta>0$, then for a stable point $(X,Y,I,J)$, the image of the $I$ generates each $V_{i}$ under the action of $X$ and $Y$. Part 2 follows by the above discussion of how the quiver maps affect weight spaces. Part 1 is similar.

\end{proof}

\begin{lemma}\label{fixed-pt-J-I=0}
Consider a quiver variety $\qv:=\qv_{Q,\theta}(\dv,\dw)$, and let $p\in\qv^{\bT}$. Let $(X,Y,I,J)$ be a representative of $p$. 

\begin{enumerate}
    \item $J_iI_i=0$ for all $i \in Q_0$.
    \item If $\theta<0$, then $I_i=0$ for all $i \in Q_0$.
    \item If $\theta>0$, then $J_{i}=0$ for all $i \in Q_{0}$.
\end{enumerate}
\end{lemma}
\begin{proof}
Since $p$ is $\bT$-fixed, $JI(W_{i}(a_{i,j})) \subset W_i(a_{i,j} \hbar)$. Since framing spaces have trivial $\Cs_{\hbar}$-weight, $JI=0$.

For statements 2 and 3, we utilize Lemma \ref{lem: wght-spces}. 

Since $p$ is $\bT$-fixed, $I_i(W_{i}(a_{i,j})) \subset V_{i}(a_{i,j}^{-1})$. If $\theta<0$, the latter weight space is $0$ by Lemma \ref{lem: wght-spces}. This proves statement 2.

Similarly, $J_i(V_{i}(a_{j,k}^{-1} \hbar^{c})) \subset W_{i}(a_{j,k} \hbar^{1-c})$. If $\theta>0$, the former weight space is $0$ unless $c\leq 0$, in which case the latter is $0$. This proves statement 3.
\end{proof}

\subsection{Slant sum of quivers}\label{sec: slant sum quivers}

Let $Q^{(1)}$ and $Q^{(2)}$ be quivers and let $\sspt_{r} \in Q^{(r)}_{0}$. Let $Q$ be the quiver $Q=(Q_{0},Q_{1})$ such that $Q_{0}=Q^{(1)}_{0} \sqcup Q^{(2)}_{0}$ and $Q_{1}=Q^{(1)}_{1} \sqcup Q^{(2)}_{1} \sqcup \{\star_{1} \to \star_{2}\}$. We call $Q$ the \emph{slant sum} of $Q^{(1)}$ and $Q^{(2)}$ over the vertices $\sspt_{1}$ and $\sspt_{2}$ and will denote this by
\[
Q=Q^{(1)} {}_{\sspt_1}\slantsum_{\sspt_2}Q^{(2)}
\]
or sometimes just $Q=Q^{(1)}\slantsum Q^{(2)}$ if the choices of $\sspt_{1}$ and $\sspt_{2}$ are clear from context. In words, $Q$ is given by taking the disjoint union of $Q^{(1)}$ and $Q^{(2)}$ and adding an arrow from $\sspt_{1}$ to $\sspt_{2}$.

Let $\dv^{(r)},\dw^{(r)} \in \mathbb{Z}_{\geq 0}^{Q^{(r)}_{0}}$. We call $\sspt_{1}$ and $\sspt_{2}$ \emph{compatible} if $\dv^{(1)}_{\sspt_{1}}=\dw^{(2)}_{\sspt_{2}}$.

Let $\dv, \dw \in \mathbb{Z}_{\geq}^{Q_{0}}$ be defined by
\[
\dv_{i}=\begin{cases}
    \dv^{(1)}_{i} & \text{if $i \in Q^{(1)}_{0}$} \\
    \dv^{(2)}_{i} & \text{if $i \in Q^{(2)}_{0}$}
\end{cases}
\]
and
\[
\dw_{i}= \begin{cases}
    \dw^{(r)}_{i} & \text{if $i \in Q^{(r)}_{0} \setminus \{\sspt_{2}\}$, $r \in \{1,2\}$} \\
    0 & \text{if $i=\sspt_{2}$}
\end{cases}
\]
We will denote this by
\[
\dv=\dv^{(1)} {}_{\sspt_1}\slantsum_{\sspt_2} \dv^{(2)}, \quad \dw=\dw^{(1)} {}_{\sspt_1}\slantsum_{\sspt_2} \dw^{(2)}
\]
or sometimes just $\dv^{(1)} \slantsum \dv^{(2)}$ (similarly for $\dw$).

Intuitively, we have identified the gauge vertex at $\sspt_{1}$ with the framed vertex at $\sspt_{2}$ and chosen to view the result as a gauge vertex. All together, we call $Q, \dv, \dw$ the \emph{slant sum} of the quiver gauge theories associated to $Q^{(1)},\dv^{(1)},\dw^{(1)}$ and $Q^{(2)},\dv^{(2)},\dw^{(2)}$.

Finally, given stability vectors $\theta^{(r)} \in \mathbb{Z}^{Q^{(r)}_{0}}$, we define $\theta \in \mathbb{Z}^{Q_{0}}$ by
\[
\theta_{i}=\theta^{(r)}_{i}, \quad \text{ where $i \in Q^{(r)}_{0}$}
\]
and denote it as $\theta^{(1)}\slantsum \theta^{(2)}$.

\begin{definition}
    Fix the data of $Q^{(r)}$, $\dv^{(r)}$, $\dw^{(r)}$, and $\theta^{(r)}$ as above for $r\in \{1,2\}$. Let $\sspt_{1}$ and $\sspt_{2}$ be compatible vertices. Define $Q$, $\dv$, $\dw$, and $\theta$ as above. The \emph{slant sum} of $\qv_{Q^{(1)},\theta^{(1)}}(\dv^{(1)},\dw^{(1)})$ and $\qv_{Q^{(2)},\theta^{(2)}}(\dv^{(1)},\dw^{(2)})$ over $\sspt_{1}$ and $\sspt_{2}$ is the quiver variety $\qv_{Q,\theta}(\dv,\dw)$.
\end{definition}

We will at times omit pieces of the notation and denote this by $\qv^{(1)} \slantsum \qv^{(2)}$.

\subsection{Relationship between slant sums}\label{relat-slant-sum}

Fix $\qv^{(r)}:=\qv_{Q^{(r)},\theta^{(r)}}(\dv^{(r)},\dw^{(r)})$ for $r\in \{1,2\}$ as above. Let $\sspt_{1}$ and $\sspt_{2}$ be compatible vertices. Let $\qv:=\qv^{(1)} {}_{\sspt_{1}}\slantsum_{\sspt_{2}} \qv^{(2)}$. There is a diagram relating these quiver varieties.

We fix the vector spaces $V^{(r)}_{i}$ and $W^{(r)}_{i}$ used in the construction of $\qv^{(r)}$. Then we use
\[
V_{i}:= V^{(r)}_{i} \text{ where $i \in Q^{(r)}_{0}$}
\]
and
\[
W_{i}:=\begin{cases}
   W^{(r)}_{i} & \text{if $i \in Q^{(r)}_{0} \setminus \{\sspt_2\}$, $r \in \{1,2\}$}  \\
    0 & \text{if $i=\sspt_{2}$}
\end{cases}
\]
as the vector spaces in the construction of $\qv$.

Consider the affine variety
\[
{\bf{M}}(\dv^{(1)},\dw^{(1)}) \times \Hom(V^{(1)}_{\sspt_{1}},W^{(2)}_{\sspt_{2}})_{\text{iso}} \times {\bf{M}}(\dv^{(2)},\dw^{(2)}) 
\]
where the subscript $\text{iso}$ denotes the subvariety consisting of linear isomorphisms. Let $\bf{Y}$ be the subvariety satisfying the moment map equations and stability for both $Q^{(1)}$ and $Q^{(2)}$, i.e.
\[
{\bf{Y}}=\mu^{-1}_{1}(0)^{s} \times \Hom(V^{(1)}_{\sspt_{1}},W^{(2)}_{\sspt_{2}})_{\text{iso}} \times \mu^{-1}_{2}(0)^{s}
\]
The projection
\[
{\bf{Y}} \to \mu^{-1}_{1}(0)^{s} \times \mu^{-1}_{2}(0)^{s}
\]
is a $G=G_{\dv^{(1)}}\times G_{\dv^{(2)}}$-equivariant map which thus descends to the quotients
\[
{\bf{Y}}/G \to \qv^{(1)} \times \qv^{(2)}
\]

We denote points in ${\bf{Y}}$ by $(p^{(1)},\phi,p^{(2)})$, where $p^{(r)}=(X^{(r)},Y^{(r)},I^{(r)},J^{(r)})$ for $r \in \{1,2\}$ denotes the quiver representation as usual. We define a morphism

    \[
F: {\bf{Y}} \to {\bf{M}}(\dv,\dw)
    \]
    by 
    \[
    F(p^{(1)},\phi,p^{(2)})=(X,Y,I,J)
    \]
    where the components of the linear maps $(X,Y,I,J)$ are defined as follows. Set
    \[
X_{e}=X^{(r)}_{e} \quad \text{where $e \in Q^{(r)}_{1}$}
    \]
    and similarly for $Y_{e}$. Set
    \[
I_{i} = \begin{cases}
    I^{(r)}_{i} & \text{if $i \in Q^{(r)}_{0} \setminus \{\sspt_{2}\}$} \\
    0 & \text{if $i=\sspt_{2}$}
\end{cases}
    \]
    and similarly for $J_{i}$. Finally, for the new edge $e_{0}=\sspt_{1} \to \sspt_{2}$, we set $X_{e_{0}}= I^{(2)}_{\sspt_{2}} \circ \phi$ and $Y_{e_{0}}=\phi^{-1} \circ J^{(2)}_{\sspt_{2}}$.

\begin{lemma}
    The morphism $F$ is $G$-equivariant.
\end{lemma}
\begin{proof}
    The only things to check are that $g_{\sspt_{2}} X_{e_{0}} g_{\sspt_{1}}^{-1}=g_{\sspt_{2}} I^{(2)}_{\sspt_{2}} \phi'$ and $g_{\sspt_{1}} Y_{e_{0}} g_{\sspt_{2}}^{-1}= \phi'^{-1} J^{(2)}_{\sspt_{2}} g_{\sspt_{2}}^{-1}$ where $\phi'=\phi g_{\sspt_{1}}^{-1}$. These are immediate.
\end{proof}
    
    Let ${\bf{Z}}=F^{-1}(\mu^{-1}(0)^{s}) \subset {\bf{Y}}$. 

By the definitions, we obtain a diagram 
\begin{equation}\label{eq: slant sum diagram}
 \begin{tikzcd}
      \mu^{-1}_{1}(0)^{s}\times \mu^{-1}_{2}(0)^{s} \arrow[d] 
      & {\bf{Y}} \arrow[d]  \arrow{l}& {\bf{Z}}    \arrow[hookrightarrow,l] \arrow[d] \arrow[r] & \mu^{-1}(0)^{s} \arrow[d] \\
      \qv^{(1)} \times \qv^{(2)} & {\bf{Y}}/G \arrow[l]  & {\bf{Z}}/G \arrow[hookrightarrow,l] \arrow[r]  & \qv
       \end{tikzcd}
\end{equation}

Later we will need the following lemma.
\begin{lemma}\label{lem: preserve stability}
The following holds:

\begin{enumerate}
    \item The closed subset of ${\bf{Y}}$ defined by $J^{(2)}_{\sspt_{2}} I^{(2)}_{\sspt_{2}}=0$ is equal to $F^{-1}(\mu^{-1}(0))$.
    \item If $\theta^{(2)}<0$, then $\{I^{(2)}_{\sspt_{2}}=0\} \subset {\bf{Z}}$.
     \item If $\theta^{(2)}>0$, then $\{J^{(2)}_{\sspt_{2}} =0\} \subset {\bf{Z}}$.
\end{enumerate}
 
\end{lemma}

\begin{proof}
Consider the first statement. The moment map equations for $\qv$ differ from those of $\qv^{(1)}$ or $\qv^{(2)}$ only at the two vertices $\sspt_{1}$ and $\sspt_{2}$. The $\sspt_{1}$-component of the moment map for the $G$ action on ${\bf{M}}(\dv,\dw)$ applied to $F(p^{(1)},\phi,p^{(2)})$ is
   \begin{align*}
    & \sum_{\substack{e\in Q_{1} \\ h(e)=\sspt_{1}}}X_{e} Y_{e}- \sum_{\substack{e \in Q_{1}\\ t(e)=\sspt_{1}}} Y_{e} X_{e} +I_{\sspt_{1}} J_{\sspt_{1}} 
     \\
   &=\sum_{\substack{e\in Q^{(1)}_{1} \\ h(e)=\sspt_{1}}}X^{(1)}_{e} Y^{(1)}_{e}- \sum_{\substack{e \in Q^{(1)}_{1}\\ t(e)=\sspt_{1}}} Y^{(1)}_{e} X^{(1)}_{e} +I^{(1)}_{\sspt_{1}} J^{(1)}_{\sspt_{1}}-Y_{e_{0}} X_{e_{0}} \\
   &= -Y_{e_{0}} X_{e_{0}} \\
   &= \phi^{-1} J^{(2)}_{\sspt_{2}} I^{(2)}_{\sspt_{2}} \phi
    \end{align*}
  Similarly, the $\sspt_{2}$-component of the moment map is
    \begin{align*}
         & \sum_{\substack{e\in Q_{1} \\ h(e)=\sspt_{2}}}X_{e} Y_{e}- \sum_{\substack{e \in Q_{1}\\ t(e)=\sspt_{2}}} Y_{e} X_{e} \\
         &= \sum_{\substack{e\in Q^{(2)}_{1} \\ h(e)=\sspt_{2}}}X^{(2)}_{e} Y^{(2)}_{e}- \sum_{\substack{e \in Q^{(2)}_{1}\\ t(e)=\sspt_{2}}} Y^{(2)}_{e} X^{(2)}_{e} +X_{e_{0}}Y_{e_{0}} \\
         &= \sum_{\substack{e\in Q^{(2)}_{1} \\ h(e)=\sspt_{2}}}X^{(2)}_{e} Y^{(2)}_{e}- \sum_{\substack{e \in Q^{(2)}_{1}\\ t(e)=\sspt_{2}}} Y^{(2)}_{e} X^{(2)}_{e} +I^{(2)}_{\sspt_{2}} J^{(2)}_{\sspt_{2}}
    \end{align*}
    which is exactly the $\sspt_{2}$-component of the moment map equation for the $G_{\dv^{(2)}}$ action on ${\bf{M}}(\dv^{(2)},\dw^{(2)})$ and hence is $0$. This proves the first statement.

For the second statement, recall the representation-theoretic characterization of stability from Proposition \ref{prop: rep-theory-stability}. 

Let $S=(S_i)_{i\in Q_0}$ be a subrepresentation of 
    \[
    f(p^{(1)},\phi,p^{(2)})= (X,Y,I,J)
    \]
such that $S_i\subseteq\ker J_i$ for all $i\in Q_0$.

There is a decomposition
\[
(S_i)_{i\in Q_0} = (S^{(1)}_i)_{i\in Q_0^{(1)}} \oplus (S_i^{(2)})_{i\in Q_0^{(2)}}.
\]
Denote $S^{(r)}=(S_i^{(r)})_{i\in Q_0^{(r)}}$ for $r=1,2$. It is automatic that $S^{(1)}$ is a subrepresentation of $(X^{(1)},Y^{(1)},I^{(1)},J^{(1)})$ for which $S^{(1)}_i \subseteq \ker J^{(1)}_i$ for all $i\in Q_0^{(1)}$. So stability for $p^{(1)}$ implies $\theta^{(1)} \cdot \dim S^{(1)} < 0$. If $\theta^{(2)}=-(1,1,\ldots,1)$, then we also have $\theta^{(2)} \cdot \dim S^{(2)} \leq 0$. If $\theta^{(2)}=(1,1,\ldots,1)$ and $J^{(2)}_{\sspt_{2}}=0$, then $S^{(2)}$ is also a subrepresentation of $p^{(2)}$ for which $S^{(2)}_{i} \subseteq \ker J^{(2)}_{i}$ for all $i \in Q_0^{(2)}$; so stability for $p^{(2)}$ implies $\theta^{(2)} \cdot \dim S^{(2)}\leq 0$. In either case, we have
\[
\theta \cdot \dim S = \theta^{(1)} \cdot \dim S^{(1)} + \theta^{(2)} \cdot \dim S^{(2)}\leq 0
\]

Now let $S=(S_{i})_{i \in Q_{0}}$ be a subrepresentation of $(X,Y,I,J)$ such that $\im I_{i} \subseteq S_{i}$ for all $i \in Q_{0}$. As above, we obtain $S^{(1)}$ and $S^{(2)}$. It is automatic that $\im I^{(1)}_{i} \subseteq S^{(1)}_{i}$ for all $i \in Q^{(1)}_{i}$. So stability for $p^{(1)}$ implies that $\theta^{(1)} \cdot \dim S^{(1)} \leq \theta \cdot \dim V^{(1)}$. If $\theta^{(2)}=(1,1,\ldots,1)$, then $\theta^{(2)} \cdot \dim S^{(2)} \leq \theta^{(2)} \dim V^{(2)}$. If $\theta^{(2)}=-(1,1,\ldots,1)$ and $I_{\sspt_{2}}^{(2)}=0$, then $S^{(2)}$ is also a subrepresentation of $p^{(2)}$ such that $\im I^{(2}_{i} \subseteq S^{(2)}_{i}$ for all $i \in Q^{(2)}_{0}$; so stability for $p^{(2)}$ implies $\theta^{(2)} \cdot \dim S^{(2)} \leq \theta^{(2)} \cdot \dim V^{(2)}$. Either way

\[
\theta \cdot \dim S = \theta^{(1)} \cdot \dim S^{(1)} + \theta^{(2)} \cdot \dim S^{(2)} \leq \theta \cdot \dim V
\]
This concludes the proofs of parts 2 and 3.

\end{proof}

\subsection{Split fixed points}\label{sec: split fixed points}

We next define the notion of split fixed points. This is a technical assumption that allows us to choose a (almost) canonical basis of a gauge vector space, allowing us in practice to view it as a framing space. 

\begin{definition}\label{def: split fixed points}
   Let $\qv$ be a quiver variety. Let $p \in \qv^{\bT}$ be a fixed point. We say that $p$ is \emph{split over $i \in Q_{0}$} if all the $\bT$-weight spaces of $\tb_{i}|_{p}$ have dimension at most $1$. 
\end{definition}

Equivalently, by the discussion of Section \ref{sec: torus and fixed points}, choosing a representative $\tilde{p}$ of the fixed point $p$, there is a homomorphism $\rho: \bT \to G_{\dv}$ such that $t \cdot \tilde{p}=\rho(t) \cdot \tilde{p}$ and the joint eigenspaces of the action of $\rho(\bT)$ on the vector space $V_{i}$ are all one dimensional.

\begin{definition}
    Let $p \in \qv^{\bT}$ be a split fixed point over $i \in Q_{0}$. A \emph{$(p,i)$-chamber of $\bT$}, denoted $\chamb$, is a choice of total ordering of the $\bT$-weights of $\tb_{i}|_{p}$.
\end{definition}

Choosing a weight basis with respect to the ordering given by a $(p,i)$-chamber $\chamb$, we obtain an isomorphism $V_{i} \cong \mathbb{C}^{\dv_{i}}$. Different choices of basis give isomorphisms that differ by a diagonal matrix in $GL(\C^{\dv_{i}})$.
    
\subsection{Fixed points and slant sums}

Fix $\qv^{(r)}:=\qv_{Q^{(r)},\theta^{(r)}}(\dv^{(r)},\dw^{(r)})$ for $r=1,2$ as above, equipped with the actions for $\bT^{(r)}$ for $r \in \{1,2\}$. Let $\sspt_{1}$ and $\sspt_{2}$ be compatible vertices. Let $\qv:=\qv^{(1)} {}_{\sspt_{1}}\slantsum_{\sspt_{2}} \qv^{(2)}$, which is equipped with the action of a torus $\bT$ where $\rank \bT=\rank \bT^{(1)}+\rank \bT^{(2)}-\dw^{(2)}_{\sspt_{2}}-1$.

\begin{theorem}\label{thm: slant sum and fixed points}
   Let $p^{(1)}$ be a fixed point on $\qv^{(1)}$ split over $\sspt_{1}$. Choose a $(p^{(1)},\sspt_{1})$-chamber $\chamb$ of $\bT^{(1)}$. Assume that $\theta^{(2)}> 0$ or $\theta^{(2)}<0$. Then there is a closed embedding
    \[
\ssmap:=\ssmap_{p^{(1)},\chamb}: (\qv^{(2)})^{\bT^{(2)}} \to \qv^{\bT}
    \]
\end{theorem}
\begin{proof}
    Since $p^{(1)}$ is $\bT^{(1)}$-fixed, there is a homomorphism $\rho: \bT^{(1)} \to G_{\dv^{(1)}}$ as in Section \ref{sec: torus and fixed points} which provides a decomposition of $V_{\sspt_{1}}$ into weight spaces of $\bT^{(1)}$. Since $p^{(1)}$ is split over $\sspt_{1}$, the nonzero weight spaces are one-dimensional. The chamber $\chamb$ totally orders the weight spaces. Making a choice of weight basis for $V_{\sspt_{1}}$, we obtain an isomorphism $\phi: V_{\sspt_{1}} \xrightarrow{\sim} W_{\sspt_{2}}$ which depends both on $\chamb$ and the choice of basis.

    Let $p^{(2)} \in (\qv^{(2)})^{\bT^{(2)}}$. Let $\pi_{r}: \mu_{r}^{-1}(0)^{s} \to \qv^{(r)}$ be the quotient map for $r\in \{1,2\}$. Let $\pi_1^{\mathrm{inv}}$ be the restriction of $\pi_1$ to the locus of quiver representations which are invariant with respect to the action of $\bT^{(1)}$ induced by the grading on $V_{i}$ for $i \in Q^{(1)}_{0}$ and the action on edges of the quiver. We will define a morphism
    \[
\tilde{\ssmap}: (\pi_{1}^{\mathrm{inv}})^{-1}(p^{(1)}) \times \pi_{2}^{-1}(p^{(2)}) \to \qv
    \]
    which will descend to the map in the statement of the theorem. Note that the desired domain of $\tilde{\ssmap}$ lies in the top left of \eqref{eq: slant sum diagram}, and we will use the notation ${\bf{Y}}$ and ${\bf{Z}}$ defined there. Our choice of isomorphism $\phi: V_{\sspt_{1}} \xrightarrow{\sim} W_{\sspt_{2}}$ discussed above provides a section $ \pi_{1}^{-1}(p^{(1)}) \times \pi_{2}^{-1}(p^{(2)}) \to {\bf{Y}}$. Since $p^{(2)}$ is $\bT^{(2)}$ fixed, Lemma \ref{fixed-pt-J-I=0} implies that the hypotheses of Lemma \ref{lem: preserve stability} hold, implying that this section factors through ${\bf{Z}}$. Composing with the top right horizontal arrow followed by the rightmost vertical arrow in \eqref{eq: slant sum diagram} provides the desired $\tilde{\ssmap}$.

    The map $\tilde{\ssmap}$ is easily seen to be invariant under the action of $G_{\dv^{(2)}}$ and factors of the centralizer $Z_{G_{\dv^{(1)}}}(\rho(\bT^{(1)}))$ corresponding to vertices besides $\sspt_{1}$. Since $p^{(1)}$ is split over $\sspt_1$, the factor of this centralizer in $GL(V_{\sspt_1})$ consists of matrices that are diagonal with respect to a choice of weight basis of $V_{\sspt_{1}}$. Using the isomorphism $\phi: V_{\sspt_{1}} \to W_{\sspt_2}$ and the fact that $p^{(2)}$ is $\bT^{(2)}$-fixed, we can compensate for the action of these diagonal matrices by an element of $G_{\dv^{(2)}}$. Thus $\tilde{\ssmap}$ is also $Z_{G_{\dv^{(1)}}}(\rho(\bT^{(1)}))$ invariant and thus descends to a map $p^{(1)} \times p^{(2)} \to \qv$. Varying $p^{(2)}$ provides the map $\ssmap$ in the statement of the theorem.

    We now show that $\ssmap$ is independent of the choice of $\phi$, i.e. the choice of weight basis for $V_{\sspt_{1}}$. Let $\phi_{1},\phi_{2}: V_{\sspt_{1}} \to W_{\sspt_{2}}$ be two isomorphisms as in the first paragraph of this proof, leading to $\tilde{\ssmap}_{1}$ and $\tilde{\ssmap}_{2}$. Then $\tilde{\ssmap}_{1}(x,y)= \tilde{\ssmap}_{2}(x,t \cdot y)$, where $t:=\phi_{1} \circ \phi_{2}^{-1} \in \text{Aut}(W_{\sspt_{2}}) \cap \bT^{(2)}$. Since $p^{(2)}$ is $\bT^{(2)}$-fixed, for any $y\in \pi_{2}^{-1}(p^{(2)})$, $\exists g \in G_{\dv^{(2)}}$ such that $g \cdot y = t \cdot y$. So $\tilde{\ssmap}_{2}(x,t \cdot y)=\tilde{\ssmap}(x,g \cdot y)=(1, g)\cdot \tilde{\ssmap}_{2}(x,y)$. In other words, $\tilde{\ssmap}_{1}$ and $\tilde{\ssmap}_{2}$ differ by an element of $G$ and thus descend to the same maps.

    Finally we must justify why $\Psi$ lands in the $\bT$-fixed locus. To show this, we define a inclusion 
    \begin{equation}\label{eq: torus inclusion}
    \iota: \bT =\bA \times \Cs_{\hbar} \hookrightarrow \bT^{(1)} \times \bT^{(2)}=\bA^{(1)} \times \Cs_{\hbar^{(1)}} \times \bA^{(2)} \times \Cs_{\hbar^{(2)}}
    \end{equation}
    We have
    \begin{itemize}
        \item a diagonal inclusion $\Cs_{\hbar} \hookrightarrow \Cs_{\hbar^{(1)}} \times \Cs_{\hbar^{(2)}}$,
        \item an inclusion
    \[
   \bA \times \Cs_{\hbar}= \bA^{(1)} \times \Cs_{\hbar} \times \left(\prod_{\substack{i \in Q^{(2)}_{0} \\ i \neq \sspt_{2}}} (\Cs)^{\dw^{(2)}_{i}} \right) \subset \bA^{(1)} \times \bA^{(2)}
    \]
    whose component inside of $\text{Aut}(W_{j})$ for $j \neq \sspt_{2}$ is induced by the identity maps and whose component inside $\text{Aut}(W_{\sspt_{2}})$ is given by combining the $\bT^{(1)}=\bA^{(1)} \times \Cs_{\hbar}$ action on $V_{\sspt_{1}}$ with the isomorphism $\phi$
    \end{itemize}
    Putting these two maps together gives $\iota$, along which $\tilde{\ssmap}$ is equivariant. So $\ssmap$ is $\bT$-fixed.  
\end{proof}

The maps $\ssmap_{p^{(1),\chamb}}$ depend crucially on $p^{(1)}$. The dependence on $\chamb$, however, is very mild.

\begin{proposition}\label{prop: chamber dependence}
    Let $\chamb$ and $\chamb'$ be $(p^{(1)},\sspt_{1})$-chambers. Let $F_{i}$, $i \in S$ be the fixed components of $\qv^{(2)}$. There exist a permutation $s: S \to S$ and isomorphisms $f_{i}: F_{i} \to F_{s(i)}$ such that
    \[
    \ssmap_{p^{(1)},\chamb}\big|_{F_{i}}=\ssmap_{p^{(1)},\chamb'}\big|_{F_{s(i)}} \circ f_{i}
    \]
\end{proposition}
\begin{proof}
    The maps $\ssmap_{p^{(1)},\chamb}$ and $\ssmap_{p^{(1)},\chamb'}$ depend on $\chamb$ and $\chamb'$ only through the two isomorphisms $V_{\sspt_{1}} \xrightarrow{\sim} W_{\sspt_{2}}$. Composing one with the inverse of the other induces the desired $s$ and $f_{i}$.
\end{proof}
In other words, the maps $\ssmap_{p^{(1)},\chamb}$ depends on $\chamb$ only up to permutation of the fixed components.

\section{Quasimaps and slant sums}\label{sec: quasimaps}

\subsection{Review of quasimaps}\label{sec: qm defs}

As discussed in \S\ref{sec: quiver varieties}, Nakajima quiver varieties are defined as GIT quotients. so we may utilize the theory of quasimaps to a GIT quotient developed in \cite{qm}. In this paper, the domain of quasimaps will always be a parameterized $\mathbb{P}^{1}$.

Let $\qv:=\qv_{Q,\theta}(\dv,\dw)$ be a quiver variety, equipped with the action of $\bT=\bA \times \Cs_{\hbar}$. 

By definition a \emph{quasimap} from $\mathbb{P}^{1}$ to $\qv$ is a map $\mathbb P^{1}\to [\mu^{-1}(0)/G_{\dv}]$. It is \emph{stable} if it generically lands in $\qv$, which is contained in the stack quotient as an open subset. The points for which the map does not land in $\qv$ are called \emph{singularities} of the quasimap.

So a quasimap consists of the data $(f,(\qmtb_{i} )_{i \in Q_{0}})$ where $\qmtb_{i}$ is a vector bundle of rank $\dv_{i}$ on $\mathbb{P}^{1}$ and $f \in H^{0}(\mathbb{P}^{1}, \qmb \oplus \hbar^{-1} \otimes \qmb^{*})$ where
\[
\qmb= \bigoplus_{e \in Q_{1}}\Hom(\qmtb_{t(e)},\qmtb_{h(e)}) \oplus \bigoplus_{i\in Q_0}\Hom(\qmtbw_i,\qmtb_i)
\]
Here $\qmtbw_{i}$ denotes the trivial bundle $W_{i} \times \mathbb{P}^{1}$ on $\mathbb{P}^{1}$, which has a natural $\bA$-equivariant structure, and $\hbar^{-1}$ denotes the trivial line bundle acted on by $\bT$ with weight $\hbar^{-1}$.

The \emph{degree of a quasimap} is defined to be $(\deg \qmtb_{i})_{i \in Q_{0}} \in \mathbb{Z}^{Q_{0}}$.

Let $\qm$ be the moduli space of stable quasimaps from $\mathbb{P}^{1}$ to $\qv$. By \cite{qm}, it is a Deligne-Mumford stack of finite type with a perfect deformation-obstruction theory. Thus it is equipped with a canonical virtual structure sheaf, which we will denote by $\mathcal{O}_{\mathrm{vir}}$. The canonical polarization $T^{1/2}$ of $\qv$ pulls back under the universal evaluation morphism $\qm \times \mathbb{P}^{1} \to [\mu^{-1}(0)/G_{\dv}]$, to a class $\qmpol$, satisfying
\[
\qmpol|_{(f,(\qmtb_{i}))}=\bigoplus_{e \in Q_{1}} \Hom(\qmtb_{t(e)},\qmtb_{h(e)})\oplus \bigoplus_{i \in Q_{0}} \Hom(\qmtbw_{i},\qmtb_{i}) \ominus \bigoplus_{i \in Q_{0}} \Hom(\qmtb_{i},\qmtb_{i})
\]
The virtual tangent space at a quasimap $(f,(\qmtb_{i}))$ is 
\begin{equation}\label{eq: qm tvir}
\qmtvir|_{(f,(\qmtb_i))}= H^{*}\left(\mathbb{P}^{1},\qmpol \oplus \hbar^{-1} \qmpol\right)
\end{equation}

For $x \in \mathbb{P}^{1}$, let $\qm_{\ns x}$ be the moduli space of stable quasimaps which are nonsingular at $x$. The substack of degree $d$ quasimaps and the virtual sheaf on it will be denoted by $\qm^{d}_{\ns x}$ and $\mathcal{O}_{\mathrm{vir}}^{d}$ respectively.

% As in \cite{qm,pcmilect}, $\qm^{d}_{\text{ns }p_0}$ is the stack parameterizing a collection of vector bundles $\qmtb_i$ and trivial vector bundles $\qmtbw_i$, a pair for each $i\in Q_0$, with $\rank \qmtb_i=\dv_i$, degree $\qmtb_i=d_i,\ \rank \qmtbw_i=\dw_i$, and a section $f\in \Gamma(\qmb\oplus\qmb^*\otimes\hbar)$, where $\hbar$ is a trivial line bundle with $\bT$-weight $\hbar$, and
% \[
% \qmb= \bigoplus_{i\to j}\Hom(\qmtb_i,\qmtb_j) \oplus \bigoplus_{i\in Q_0}\Hom(\qmtbw_i,\qmtb_i)
% \]
% such that
% \begin{itemize}
% \item 1. The section $f$ is \textit{stable}, meaning for all but finitely many $p \in \mathbb{P}^1$, $f$ is stable at $p$. In this context stable means we regard $f(p)$ as an element of $T^*\rep_Q(\dv,\dw)$, and say $f$ is stable at $p$ if it is $\theta$-stable, as in \ref{prop: rep-theory-stability}. A stable section of a quasimap may be not-stable at finitely many $p \in \mathbb{P}^1$, and these points are called \textit{singularities}. Otherwise we say $f$ is nonsingular at $p$. 
% \item 2. $f(p)$ satisfies the moment map equations, again interpreted as an element of $T^*\rep_Q(\dv,\dw)$. 
% \item 3. $f$ is nonsingular at $p_0\in\mathbb{P}^1$. 
% \end{itemize}

We sometimes refer to an element of $\qm^d_{\ns x}$ only by its section $f$, though it also carries the information of the bundles $\qmtb_i$. As stated, $\qm^{d}_{\ns x}$ depends explicitly on the presentation of $\qv$ as a GIT quotient. In 3d mirror symmetry, this is viewed as a feature.

There is a natural action of $\Cs$ on $\mathbb{P}^{1}$ such that $(\mathbb{P}^{1})^{\Cs}=\{0,\infty\}$. We denote by $q$ the weight of $T_{0} \mathbb{P}^{1}$ and denote this rank one torus by $\Cs_{q}$. The action of $\Cs_{q}$ on $\mathbb{P}^{1}$ induces an action on $\qm^{d}_{\ns \infty}$, as does the action of $\bT$ on $\qv$. There is a natural $\bT \times \Cs_{q}$ equivariant structure on $\mathcal{O}_{\mathrm{vir}}^{d}$ and we will work $\bT\times \Cs_q$-equivariantly. 

% If a stable quasimap is $\bT\times \C^\times_q$-fixed, then the section $f$ cannot have singularities away from the $\Cs_q$-fixed locus of $\mathbb{P}^1$, so the nonsingular condition only has information if $p_0\in\{0,\infty\}\subset\mathbb{P}^1$. For this paper, we set $p_0=\infty$.

There are morphisms
\[
\begin{tikzcd}
    & \qm^{d}_{\ns \infty} \arrow{ld}{\ev_{0}} \arrow{rd}{\ev_{\infty}}  &  \\
    \left[\mu^{-1}(0)/G_{\dv}\right]  &  & \qv
\end{tikzcd}
\]
given by evaluating quasimaps at the respective points.

If a point $p\in \qv$ is chosen, we define $\qm^{d}_{p}$ by the pullback diagram
\[
\begin{tikzcd}
    \qm^{d}_{p}\arrow[d] \arrow[hookrightarrow]{r} & \qm^{d}_{\ns \infty} \arrow{d}{\ev_{\infty}}\\
    \{p\} \arrow[hookrightarrow]{r} & \qv
\end{tikzcd}
\]
of quasimaps such that $f(\infty)=p$. The space $\qm^{d}_{p}$ is sometimes referred to as the moduli space of quasimaps which are \textit{based at $p$.} If $p \in \qv^{\bT}$, then $\qm^{d}_{p}$ is preserved by the $\bT \times \Cs_{q}$ action.

\subsection{Twisted quasimaps}

Choose a co character $\sigma: \Cs \to \bA$. We will also need to consider \emph{$\sigma$-twisted quasimaps to $\qv$}. The definition is identical to that of \S\ref{sec: qm defs}, except instead of the trivial bundles $\qmtbw_{i}$, we use the nontrivial bundles
\[
\qmtbw^{\sigma}_{i}:=(\mathbb{C}^{2}\setminus \{0\}) \times W_i)/\Cs
\]
where $\Cs$ acts by scaling on the first factor and on the second by $\sigma$. 

% Equivalently, a quasimap from $\mathbb{P}^{1}$ to $\qv$ is a section of the trivial $[\mu^{-1}(0)/G_{\dv}]$ bundle on $\mathbb{P}^{1}$; a $\sigma$-twisted quasimap is a section of a nontrivial $[\mu^{-1}(0)/G_{\dv}]$ bundle on $\mathbb{P}^{1}$ whose topology is determined by $\sigma$.

Denote the moduli space of $\sigma$-twisted quasimaps by $\qm^{\sigma}$. All the discussion of \S\ref{sec: qm defs} applies to $\sigma$-twisted quasimaps. In particular, we have $\qm^{\sigma,d}$, $\qm^{\sigma,d}_{\ns \infty}$, $\qm^{\sigma,d}_{p}$. The virtual tangent space is given by \eqref{eq: qm tvir}.

% The definition of a $\sigma$-twisted quasimap is identical to that of a quasimap, except the trivial bundles in the untwisted version, $\qmtbw_i$, are not forced to be trivial anymore: their topological information is determined by the cocharacter. Since $\tbw_i$ carry an $\bA$-action, by Grothendieck splitting theorem the $\qmtbw_i$ are uniquely determined by a choice of cocharacter as bundles over $\mathbb{P}^1$ associated to $\mathcal O(1)$.  

% \begin{example}
% If $\qv$ has only one framing with $|w|=n$, then for $\sigma(t)=(t^{\sigma_1},\dots,t^{\sigma_n})$, the rank $n$ bundle appearing in the definition of a twisted quasimap is
% \[
% \qmtbw_i\cong \mathcal O(\sigma_1)\oplus\dots\oplus \mathcal O(\sigma_n)
% \]
% thus a component of the resulting section is a section of 
% \[\Hom(\qmtbw_i,\qmtb_i) \cong \bigoplus_{b=1}^n\bigoplus_{a=1}^{v_i}\mathcal O(d_a-\sigma_b)\]
% where $\qmtb_i\cong \mathcal O(d_1)\oplus\dots \mathcal O(d_{v_i})$. 
% \end{example}

% \textcolor{red}{Example/Definition above subject to change}

\subsection{Vertex functions}
The restriction of the evaluation map $\ev_{\infty}: \qm^{d}_{\ns \infty} \to \qv$ to the $\C^\times_q$-fixed locus is proper, see \cite[\S 7.3]{pcmilect}. So we obtain a pushforward in localized equivariant $K$-theory,
\[
(\ev_{\infty})_*: K_{\bT\times\Cs_q}(\qm^{d}_{\ns \infty})\to K_{\bT\times\Cs_q}(\qv)_{loc}:=K_{\bT\times\Cs_q}(\qv) \otimes \operatorname{Frac} K_{\bT \times \Cs_{q}}(pt)
\]

% Moduli spaces of quasimaps are examples of Deligne-Mumford stacks with perfect obstruction theory, so we may consider the corresponding virtual structure sheaf, $\mathcal{O}^d_{\text{vir}}\in K_{\bT\times\C^\times_q}(\qm^d_{\text{ns }\infty})$ and virtual tangent spaces $\vT$.

The \emph{symmetrized virtual structure sheaf} on $\qm_{\ns \infty}^{d}$ is defined by

\begin{equation}\label{eq: symvss}
\vrs^{d}= \mathcal{O}_{\mathrm{vir}}^{d} \otimes \left( \mathscr{K}_{\mathrm{vir}}  \frac{ \det \qmpol|_{\infty} }{\det \qmpol|_{0}}\right)^{1/2}
\end{equation}
where $\mathscr{K}_{\mathrm{vir}}=\ev_{\infty}^{*}(\det T \qv)\otimes \left(\det \qmtvir\right)^{-1}$ is the (normalized) virtual canonical bundle. The existence of the square root is discussed in \cite[\S 6.1.8]{pcmilect}.

% After choosing a polarization, one can define the symmetrized virtual structure sheaf $\vrs^{d}$ which is preferred for $K$-theoretic quasimap counting. Using $\ev_0$, we may pull back classes from the equivariant $K$-theory of the quotient stack, $\tau\in K_{\bT\times\Cs_q}([\mu^{-1}(0)/G_{\dv}])$, and twist the symmetrized virtual structure sheaf
% \[
% \ev_0^*(\tau)\otimes\vrs \in K_{\bT\times\Cs_q}(\qm^d_{\text{ns }\infty})
% \]
Enumerative counts of quasimaps are encoded in the generating function called the \emph{descendant vertex function}. For a \emph{descendant} $\tau \in K_{\bT}([\mu^{-1}(0)/G_{\dv}])$ this is defined as
\[
\ver^{(\tau)}(z)=\sum_{d \in \eff(\qv)} (\ev_\infty)_*\left(\qm^d_{\ns \infty}, \ev_0^*(\tau)\otimes\vrs^d\right) z^d
\]
Here the sum runs over the cone of \emph{effective quasimap classes}, which is defined as the set of $d \in \mathbb{Z}^{Q_{0}}$ for which there exists a stable quasimap of degree $d$ and $z^d$ stands for the multidegree $\prod_{i \in Q_{0}} z_i^{d_i}$. Vertex functions are elements of
\[
K_{\bT\times\Cs_q}(\qv)_{loc}[[z]]:=\left\{\sum_{d} a_{d} z^{d} \, \mid \, d \in \eff(\qv), a_{d} \in K_{\bT\times\Cs_q}(\qv)_{loc} \right\}
\]

The case of $\tau=1$ is sometimes referred to as the \textit{bare vertex function} and denoted $\ver(z):=\ver^{(1)}(z)$. For $p \in \qv^{\bT}$, we denote the restriction by $\ver^{(\tau)}_p(z):=\ver^{(\tau)}(z)$.

Using twisted quasimaps, we can also define $\sigma$-twisted vertex functions, which we will denote by $\ver^{(\tau),\sigma}(z)$.

\subsection{Branching rule for vertex functions}\label{sec: vertex and slant sums}

Now we return to slant sums and study vertex functions in the setting of Theorem \ref{thm: slant sum and fixed points}, which we briefly recall.

Let $\qv=\qv^{(1)} \slantsum \qv^{(2)}$ be a slant sum over compatible vertices $\sspt_{1}$ and $\sspt_{2}$. We write the three quivers as $Q$, $Q^{(1)}$, and $Q^{(2)}$. Let $p^{(r)} \in (\qv^{(r)})^{\bT^{(r)}}$ and assume that $p^{(1)}$ is split over $\sspt_{1}$. Assume $\theta^{(2)}=\pm(1,1,\ldots,1)$. Choose a $(p^{(1)},\sspt_{1})$-chamber $\chamb$ and let $p=\ssmap_{p^{(1)},\chamb}(p^{(2)})$ as in Theorem \ref{thm: slant sum and fixed points}.

As in \eqref{eq: torus inclusion}, we have an inclusion of tori $\iota: \bT \hookrightarrow \bT^{(1)} \times \bT^{(2)}$ such that  $\tb_{i}|_{p}=\iota^{*} \tb^{(r)}_{i}|_{p^{(r)}}$ as $\bT$-representations for $i \in Q^{(r)}$, $r\in\{1,2\}$. Furthermore, the composition $\bT \xrightarrow{\iota} \bT^{(1)} \times \bT^{(2)} \xrightarrow{pr_2} \bT^{(2)}$ is surjective.

There are now several quiver varieties present and we will consider based quasimaps to each of them. The notation $\qm_{p}$ and $\qm_{p^{(1)}}$ now refer to based quasimaps to the different varieties $\qv$ and $\qv^{(1)}$, respectively, rather than quasimaps to a fixed variety based at different points.

Any ordered tuple of integers $d=(d_1,d_2,\ldots,d_n)$ where $n=\dim W_{\sspt_{2}}$ gives a cocharacter $\sigma_{d}$ of the framing torus $\bA^{(2)}$, whose component acting on $W_{\sspt_{2}}$ is $(t) \mapsto \text{diag}(t^{d_1},t^{d_2},\ldots,t^{d_n})$ and is the identity on $W^{(2)}_{i}$ for $i \neq \sspt_{2}$.

We can obtain such a tuple from a $\bT\times \Cs_{q}$-fixed quasimap $(f, (\qmtb_{i}))$ in $\qm_{p}$ in the following way. Being fixed by $\bT \times \Cs_{q}$ implies that each bundle $\qmtb_{i}$ is endowed with a grading by $\bT$-weights of $\tb_{i}|_{p}$. The $\bT$-weights of $\tb_{i}|_{p}$ are exactly the same as the $\bT^{(1)}$-weights of $\tb^{(1)}_{i}|_{p^{(1)}}$ for $i \in Q^{(1)}$. In particular, $\qmtb_{\sspt_{1}}$ is graded by distinct and ordered (by $\chamb$) weights $w_1, w_2, \ldots, w_n$ of $\bT^{(1)}$. Then we have $\qmtb_{\sspt_{1}} \cong \bigoplus_{i=1}^{n} w_i \mathcal{O}(d_{i})$, giving rise to a tuple $(d_1,d_2,\ldots, d_n)$.

By continuity, the cocharacter $\sigma_{f}$ depends only on the fixed component $F$ containing $(f,(\qmtb_{i}))$. So we will also denote it by $\sigma_{F}$.

We are now ready to relate based quasimaps to the three quiver varieties.

\begin{theorem}\label{thm: qm and slant sums}
   There is an isomorphism
   \[
\left(\qm^{d^{(1)},d^{(2)}}_{p}\right)^{\bT\times \mathbb{C}^{\times}_{q}} \cong \bigsqcup_{F} F \times \left(\qm^{\sigma_{F},d^{(2)}}_{p^{(2)}}\right)^{\bT^{(2)} \times \Cs_{q}}
   \]
   where $F$ runs over all connected components of $\left(\qm^{d^{(1)}}_{p^{(1)}}\right)^{\bT \times \Cs_{q}}$. Furthermore, this isomorphism identifies the virtual tangent spaces as $\bT$ representations, where the virtual tangent spaces on the right hand side are pulled back via $\iota$.
\end{theorem}

\begin{proof}
By definition, a quasimap to $\qv$ is a pair $(f,(\qmtb)_{i \in Q_{0}})$, where $\qmtb_i$ are vector bundles on $\mathbb{P}^{1}$ and
\[
f\in H^{0}\left(\mathbb{P}^{1},\qmb\oplus\qmb^*\right)
\]
satisfying aforementioned conditions, where 
\[
\qmb=\bigoplus_{e \in Q_{1}}\Hom(\qmtb_{t(e)},\qmtb_{h(e)})\oplus\bigoplus_{i\in Q_0}\Hom(\qmtbw_i,\qmtb_i)
\]
and $\qmtbw_{i}$ is the trivial bundle of rank $\dw_{i}$. This data is considered up to isomorphism of quasimaps, which are defined to be isomorphisms of the vector bundles (which are identities on $\qmtbw_{i}$). 

Assume that the quasimap $f$ is $\bT\times \Cs_{q}$-fixed and evaluates to $p$ at $\infty$. Being $\bT$-fixed implies that each bundle $\qmtb_{i}$ is graded by the $\bT$-weights of $\tb_{i}|_{p}$. In particular, $\qmtb_{\sspt_{1}}=\bigoplus_{j} w_{j} \qmtb_{\sspt_{1},j}$ where $w_{j}$ run over distinct $\bT$ weights, which are canonically ordered due to the choice of $\chamb$. Let $\qmtbw_{\sspt_{2}}=\bigoplus_{j} a_{j} \mathcal{O}(d_{j})$ where $d_{j}=\deg \qmtb_{\sspt_{1},j}$. Choose an isomorphism $\phi: \qmtb_{\sspt_{1}} \to \qmtbw_{\sspt_{2}}$ equivariant with respect to $\iota$. The isomorphism $\phi$ determines and is determined by isomorphisms $\qmtb_{\sspt_{1},j} \to \mathcal{O}(d_j)$.

Being $\Cs_{q}$-fixed implies that the quasimap evaluates to $p$ everywhere except possibly at $0 \in \mathbb{P}^{1}$.

Let 
\[
\qmb^{(1)}=\bigoplus_{e \in Q_{1}^{(1)}}\Hom(\qmtb_{t(e)},\qmtb_{h(e)})\oplus\bigoplus_{i\in Q_0^{(1)}}\Hom(\qmtbw_i,\qmtb_i)
\]
and 
\[
\qmb^{(2)}=\bigoplus_{e \in Q_{1}^{(2)}}\Hom(\qmtb_{t(e)},\qmtb_{h(e)})\oplus\bigoplus_{i\in Q_0^{(2)}}\Hom(\qmtbw_i,\qmtb_i) \oplus \Hom(\qmtbw_{\sspt_{2 
}}, \qmtb_{\sspt_{2}})
\]

Let $f^{(1)}$ be the global section of $\qmb^{(1)}\oplus (\qmb^{(1)})^{*}$ determined by taking the respective components of $f$. The data $(f^{(1)},(\qmtb_{i})_{i \in Q^{(1)}_{0}})$ defines a quasimap to $\qv^{(1)}$. The quasimap $f^{(1)}$ is $\bT^{(1)}\times \Cs_{q}$-fixed and maps $\infty$ to $p^{(1)}$. Furthermore, an isomorphism $(f,(\qmtb_i)_{i \in Q_{0}}) \cong (f',(\qmtb_i)_{i \in Q_{0}})$ induces an isomorphism of $f^{(1)}$ and $f'^{(1)}$. Thus we have a well defined map
\[
\left(\qm_{p}\right)^{\bT \times \Cs_{q}} \to \left(\qm_{p^{(1)}}\right)^{\bT^{(1)} \times \Cs_{q}}
\]
It is clear from the construction that if $\deg f= (d^{(1)},d^{(2)})$, then $\deg f^{(1)}=d^{(1)}$.

Using $\phi$, we similarly define a global section $f^{(2)}$ of $\qmb^{(2)}\oplus (\qmb^{(2)})^{*}$ by taking the relevant components of $f$.

We claim that the isomorphism class of $f^{(2)}$ does not depend on $\phi$. Let $\phi$ and $\phi'$ be two such isomorphisms, leading to $f^{(2)}$ and $f'^{(2)}$. Note that $\phi' \circ \phi^{-1} \in \text{Aut}(\qmtbw_{\sspt_{2}})$ preserves the $\bT^{(2)}$-grading and is thus the same as the action of an element $t \in \bT^{(2)}$. By definition, the quasimaps $f^{(2)}$ and $f'^{(2)}$ are related by $f^{(2)}=t \cdot f'^{(2)}$. Choose a lift $\tilde{t} \in \bT$ of $t$ under the composition $\bT \xrightarrow{\iota} \bT^{(1)} \times \bT^{(2)} \to \bT^{(2)}$. Since $f$ is $\bT$-fixed, there is an isomorphism $\tilde{t} \cdot (f,(\qmtb_{i})_{i \in Q_{0}}) \cong (f,(\qmtb_{i})_{i \in Q_{0}})$ which induces an isomorphism $(f^{(2)},(\qmtb_{i})_{i \in Q^{(2)}_{0}}) \cong (f'^{(2)},(\qmtb_{i})_{i \in Q^{(2)}_{0}})$.

A similar argument shows that $(f,(\qmtb_{i})_{i \in Q_{0}}) \cong (f',(\qmtb'_{i})_{i \in Q_{0}})$ implies that $(f^{(2)},(\qmtb_{i})_{i \in Q_{0}^{(2)}}) \cong (f'^{(2)},(\qmtb'_{i})_{i \in Q_{0}^{(2)}})$. 

Overall, we obtain a degree preserving map
\[
\left(\qm_{p}\right)^{\bT \times \Cs_{q}} \to \bigsqcup_{F} F \times \left(\qm_{p^{(2)}}^{\sigma_{F}}\right)^{\bT^{(2)} \times \Cs_{q}}
\]

The inverse is easy to construct.

It also follows immediately from \eqref{eq: qm tvir} that the isomorphism of the theorem identifies the virtual tangent spaces as $\bT\times \Cs_{q}$ representations, where we view the virtual tangent spaces on the right hand side as a $\bT$ representation via the map $\iota: \bT \hookrightarrow \bT^{(1)} \times \bT^{(2)}$.

\end{proof}

Since it identifies the virtual tangent spaces, the isomorphism above respects the virtual structure sheaves. Since we assume the canonical polarization is used in \eqref{eq: symvss} for all the symmetrized virtual sheaves, it follows that the isomorphism respects the symmetrized virtual structure sheaves as well. Recall the map $\iota$ from \eqref{eq: descendant map} and let $\tau=\iota^{*}(\tau_1\otimes \tau_2)$. We write the K\"ahler parameters similarly as $z=(z_1,z_2)$.

\begin{theorem}\label{thm: vertex and slant sum 1}
We have
   \[
\ver^{(\tau)}_{p}(z_1,z_2)= 
\sum_{F} \chi\left(F, \frac{(\ev_{0}^{*}(\tau_{1}) \otimes \vrs)|_{F}}{\extpow(N_{\vir}^{\vee}|_{F})}\right) z_1^{\deg F} \iota^{*} \ver^{(\tau_{2}),\sigma_{F}}_{p^{(2)}}(z_2)
 \]
    where the sum runs over $\bT^{(1)} \times \Cs_{q}$-fixed components $F$ of $\qm_{p^{(1)}}$.
\end{theorem}
\begin{proof}
    This follows from Theorem \ref{thm: qm and slant sums} and the virtual localization theorem.
\end{proof}

For any quiver variety $\qv$ and for any cocharacter $\sigma: \Cs \to \bT$, there is a bijection between torus fixed quasimaps and torus fixed twisted quasimaps. This bijection shifts the degree:
\[
\left(\qm^{d,\sigma }_{p}\right)^{\bT \times \Cs_{q}} \cong \left( \qm^{\tilde{d}}_{p} \right)^{\bT \times \Cs_{q}}
\]
where $\tilde{d}_{i}=d_{i}+\langle \lb_{i}|_{p},\sigma\rangle$ and $\langle \cdot, \cdot \rangle$ is the natural pairing between characters and cocharacters.

And the virtual tangent spaces at a torus fixed quasimap $f$ are related by
\begin{equation}\label{eq: tvir comparison}
\qmtvir^{\sigma}|_{f}- \qmtvir|_{f}=\frac{T^{\sigma}_p \qv- T_p \qv}{1-q}
\end{equation}
where $T^{\sigma}_p \qv$ is the tangent space of $\qv$ at $p$, viewed as a $\bT \times \Cs_{q}$ representation where the second factor acts by $\sigma$. As in \cite[\S 8.2]{pcmilect}, this leads to a relationship between the twisted and untwisted vertex functions.

To state it, it is convenient to define a few transcendental functions. Let $\Phi$ be the function defined on torus weights by $\Phi(x):=\prod_{i=0}^{\infty} (1-x q^{i})$ and extended to sums and differences of weights by multiplicativity. For a torus fixed point $p$ in a quiver variety, let $\Phi_{p}=\Phi((q-\hbar) T^{1/2}|_{p})$ and $\exppref_p=\exp\left(\ln(q)^{-1}\sum_{i} \ln(\lb_{i}|_{p}) \ln(z_{i})\right)$. 

\begin{proposition}\label{prop: twisted vertex}
    Let $\qv$ be a quiver variety and let $p \in \qv^{\bT}$. Let $\sigma: \Cs \to \bA$ be a cocharacter. The twisted and untwisted vertex functions are related by
    \[
\ver^{(\tau),\sigma}_{p} =\Phipref_{p}^{-1} \exppref_{p}^{-1}  \left( \Phipref_{p} \exppref_{p} \ver^{(\tau)}_{p}\right)\big|_{a=a q^{\sigma}}
    \]
\end{proposition}

Combining Proposition \ref{prop: twisted vertex} and Theorem \ref{thm: vertex and slant sum 1}, we obtain the following.

\begin{theorem}\label{thm: vertex and slant sum 2}
With the notation as in Theorem \ref{thm: vertex and slant sum 1}, we have
    \begin{multline*}
\ver^{(\tau)}_{p}(z_1,z_2)= \\
\iota^{*}\left(\exppref_{p^{(2)}} \Phipref_{p^{(2)}}\right)^{-1} \sum_{F} \chi\left(F,\frac{(\ev_{0}^{*}(\tau_{1}) \otimes \vrs)|_{F}}{\extpow(N_{\vir}^{\vee}|_{F})} \right) z_1^{\deg F} \iota^{*}\left(\left(\exppref_{p^{(2)}} \Phipref_{p^{(2)}}\ver^{(\tau_{2})}_{p^{(2)}}\right)\big|_{a=a q^{\sigma_{F}}}\right) 
\end{multline*}
\end{theorem}

\subsection{From branching to factorization}

We explore here a few consequences of Theorem \ref{thm: vertex and slant sum 2} when the vertex function of $\qv^{(2)}$ is independent of the twist and thus can be factored out of the sum over $F$. We will freely use the notation of Theorem \ref{thm: vertex and slant sum 2} in this section.

\begin{corollary}\label{cor: vertex slant sum factorization}
    Assume that 
    \begin{itemize}
        \item $\tau_{2}|_{p^{(2)}}$ and $T\qv^{(2)}|_{p^{(2)}}$ do not depend on the framing parameters of $W_{\sspt_{2}}$ 
        \item $\tb_{i}|_{p^{(2)}}$ is symmetric in the framing parameters of $W_{\sspt_{2}}$ for all $i \in Q^{(2)}_{0}$.
    \end{itemize}
    Then
    \[
\ver^{(\tau_1\otimes \tau_{2})}_{p}(z_1,z_2)=\ver^{(\tau_{1})}_{p^{(1)}}(z_{1}') \ver^{(\tau_2)}_{p^{(2)}}(z_2)
    \]
    where $z_{1,\sspt_{1}}'=z_{1,\sspt_{1}} \prod_{i \in Q^{(2)}_{0}} z_{2,i}^{\deg_{\sspt_{2}} \lb_{i}|_{p^{(2)}}}$  and $z_{1,j}'=z_{1,j}$ otherwise.
\end{corollary}

See the proof for explanation of the notation $\deg_{\sspt_{2}} \lb_{i}$.
\begin{proof}
    The first assumption combined with \eqref{eq: tvir comparison} implies that $\ver^{(\tau_{2}),\sigma}_{p^{(2)}}$ is independent of the twist $\sigma$, up to a monomial in the K\"ahler parameters. This monomial is accounted for by $\exppref_{p^{(2)}}$. 

    The second assumption implies that the degree of $\lb_{i}|_{p^{(2)}}$ as a Laurent monomial in $a_{\sspt_{2},j}$ is independent of $j$. So we can denote it by $\deg_{\sspt_{2}} \lb_{i}|_{p^{(2)}}$.
    
    Let $F$ be a fixed component and let $d_{\sspt_{1}}\in \mathbb{Z}$ be the degree of the $\qmtb_{\sspt_{1}}$ bundle on this component. Then
    \[
    \exppref_{p^{(2)}}|_{a=a q^{\sigma_{F}}} =\exppref_{p^{(2)}} \prod_{i \in Q^{(2)}_{0}} z_{2,i}^{d_{\sspt_{1}} \deg_{\sspt_{2}} \lb_{i}|_{p^{(2)}} }
    \]
    Clearly, this monomial can be absorbed into a shift of $z_{1,\sspt_{1}}$ in the vertex function of $\qv^{(1)}$.
\end{proof}

\begin{remark}\label{remark: M2 is 0-dim corollary}
    One can show that the hypotheses of Corollary \ref{cor: vertex slant sum factorization} are satisfied if $\qv^{(2)}$ is zero dimensional and $\tau_{2}=1$, see also Proposition \ref{prop: inductive factorization} below.
\end{remark}

Due to form of the obstruction theory, the $q=\hbar$ specialization of vertex functions with descendant $1$ depends on $\hbar$ and not on any other equivariant parameters.

\begin{corollary}\label{cor: CY vertex slant sum factorization}
Assume that
\begin{itemize}
    \item  $\tb_{i}|_{p^{(2)}}$ is symmetric in the framing parameters of $W_{\sspt_{2}}$ for all $i\in Q^{(2)}_{0}$
    \item $\tau_{2}|_{p^{(2)}}$ is independent of the framing parameters of $W_{\sspt_{2}}$
\end{itemize}
Then
    \[
\ver^{(\tau_1 \otimes \tau_{2})}_{p}(z_1,z_2)\big|_{q=\hbar}=\ver^{(\tau_{1})}_{p^{(1)}}(z_1')\big|_{q=\hbar} \ver^{(\tau_{2})}_{p^{(2)}}(z_2)\big|_{q=\hbar}
    \]
    where $z_1'$ is as in Corollary \ref{cor: vertex slant sum factorization}.
\end{corollary}

Another interesting special case occurs when $\dw^{(2)}_{\sspt_2}=1$ and $\dw^{(2)}_{i}=0$ otherwise, which most closely resembles the original slant sum constructions of \cite{proctor}. 

\begin{corollary}
     Assume that $\dw^{(2)}_{\sspt_2}=1$ and $\dw^{(2)}_{i}=0$ otherwise. Then
    \[
\ver^{(\tau_1\otimes \tau_{2})}_{p}(z_1,z_2)=\ver^{(\tau_{1})}_{p^{(1)}}(z_{1}') \ver^{(\tau_2)}_{p^{(2)}}(z_2)
    \]
   where $z_{1,\sspt_{1}}'=z_{1,\sspt_{1}} q^{\deg_{\sspt_{2}} \tau_{2}|_{p^{(2)}}}\prod_{i \in Q^{(2)}_{0}} z_{2,i}^{\deg_{\sspt_{2}} \lb_{i}|_{p^{(2)}}}$ and $z_{1,j}'=z_{1,j}$ otherwise.
\end{corollary}
\begin{proof}
    The proof is similar to Corollary \ref{cor: vertex slant sum factorization}, the only difference being the possible dependence of $\tau_{2}|_{p^{(2)}}$ on the framing parameter of $W_{\sspt_{2}}$. This is accounted for in the shift for $z_1'$.
\end{proof}

\section{Consequences}\label{sec: applications}

\subsection{Factorization in zero-dimensional case}\label{sec: 0 dim varieties}

In this section, we will follow standard notation and denote simple (co)roots by $\alpha_i$ $(\alpha_i^\vee$) and fundamental (co)weights by $\fundwt_i$ ($\fundwt_i^\vee)$. We will also denote $\kappa=q/\hbar$.

Corollary \ref{cor: vertex slant sum factorization} provides a strategy to approach Conjecture \ref{conj: 0 dimensional vertex}. This conjecture is known only for type $A$ quiver varieties, proven in specific cases in \cite{dinkms1} and \cite{dinksmir2}, and later proven for all type $A$ in \cite{bottadink}, and type $D$ quiver varieties with framings only at minuscule vertices in \cite{dinkjang}. 

For each of these cases, the difficult part is proving that the vertex function factorizes to a certain explicit product of $q$-binomial functions. Then it is a relatively straightforward exercise in root combinatorics to identify the terms in the factorization with certain roots as in Conjecture \ref{conj: 0 dimensional vertex}. This is demonstrated by the following example.

\begin{example}\label{ex: (2,2) partition}
    Consider the $A_{3}$ quiver variety $\qv$ with $\dv=(1,2,1)$, $\dw=(0,1,0)$, and $\theta=(1,1,1)$. The vertex function (normalized as in \eqref{eq: msver}) can be computed by localization, which leads to the following formula:
\begin{align*}
\msver_{\qv}(z) &=\sum_{d_{i,j}}  \left(\kappa\right)^{N(d)} z_1^{d_{1,1}} z_{2}^{d_{2,1}+d_{2,2}} z_{3}^{d_{3,1}} \frac{\left(\hbar\right)_{d_{2,1}} \left(\hbar^{2}\right)_{d_{2,2}}}{\left(q\right)_{d_{2,1}} \left(q \hbar\right)_{d_{2,2}}} \cdot \\ 
&\frac{\left(\frac{q}{\hbar} \right)_{d_{2,1}-d_{2,2}} \left(q \hbar \right)_{d_{2,2}-d_{2,1}}}{\left(1 \right)_{d_{2,1}-d_{2,2}} \left( \hbar^{2} \right)_{d_{2,2}-d_{2,1}}} \cdot 
\frac{\left(\hbar \right)_{d_{2,2}-d_{1,1}} \left(1 \right)_{d_{2,1}-d_{1,1}} \left(\hbar \right)_{d_{3,1}-d_{2,1}} \left(1 \right)_{d_{3,1}-d_{2,2}}}{\left(q \right)_{d_{2,2}-d_{1,1}} \left(\frac{q}{\hbar} \right)_{d_{2,1}-d_{1,1}} \left(q \right)_{d_{3,1}-d_{2,1}} \left(\frac{q}{\hbar} \right)_{d_{3,1}-d_{2,2}}} 
\end{align*}
where the sum runs over the indices $d_{1,1}$, $d_{2,1}$, $d_{2,2}$, and $d_{3,1}$ that form a reverse plane partition over the Young diagram of $\lambda=(2,2)$ and $N(d)=-2 d_{1,1}+d_{2,1}+d_{2,2}+2 d_{3,1}$. Explicitly, the constraints are $0 \leq d_{2,1} \leq d_{1,1} \leq d_{2,2}$ and $d_{2,1} \leq d_{3,1} \leq d_{2,2}$.

Although it is not obvious from the formula, the vertex function factorizes as
\[
\msver_{\qv}(z) =\frac{\Phi\left(\hbar \kappa z_2 \right) \Phi(\hbar z_1 z_2) \Phi(\hbar \kappa^2 z_2 z_3  ) \Phi(\hbar \kappa z_1 z_2 z_3 )}{\Phi( \kappa z_2 ) \Phi( z_1 z_2) \Phi( \kappa^2 z_2 z_3 ) \Phi( \kappa z_1 z_2 z_3)}
\]

The coweights $\lambda$ and $\mu$ in this case are $\lambda=\fundwt_{2}^\vee$ and $\mu=-\fundwt_{2}^\vee$. It is easy to see that
\[
\langle\alpha,\mu\rangle =
\begin{cases}
1 & \text{if $\alpha \in \{-\alpha_{2},-\alpha_{1}-\alpha_{2},\alpha_{2}-\alpha_{3},\alpha_{1}-\alpha_{2}-\alpha_{3}\}$} \\
\geq 0 & \text{otherwise}
\end{cases}
\]
for any negative root $\alpha$. Recalling that $e^{-\alpha_1}=\kappa^{-1} z_1^{-1}$, $e^{-\alpha_2}=\kappa z_{2}^{-1}$, and $e^{-\alpha_3}=\kappa z_{3}^{-1}$, and that the vertex $\msver_\qv$ is evaluated at $z^{-1}$, this confirms Conjecture \ref{conj: 0 dimensional vertex} in this case.
\end{example}

Suppose that $\qv^{(1)}$ and $\qv^{(2)}$ are zero-dimensional quiver varieties whose vertex functions are known to factorize into a product of $q$-binomial functions. By the discussion above, they could be any type $A$ quiver varieties or type $D$ quiver varieties with framings only at minuscule vertices. Suppose that $\sspt_{1}$ and $\sspt_{2}$ are compatible vertices, that $\qv^{(1)}$ is split over $\sspt_{1}$, and that $\theta^{(2)}=\pm(1,1,\ldots,1)$. Consider $\qv:=\qv^{(1)}{}_{\sspt_{1}} \slantsum_{\sspt_{2}} \qv^{(2)}$.

With these assumptions, Corollary \ref{cor: vertex slant sum factorization} implies the following, see also Remark \ref{remark: M2 is 0-dim corollary}.
\begin{proposition}\label{prop: inductive factorization}
    The vertex function of $\qv$ factorizes into a product of $q$-binomial functions.
\end{proposition}

As remarked above, proving this factorization is difficult. Matching the terms with the roots in Conjecture \ref{conj: 0 dimensional vertex} for $\qv$ is now reduced to a combinatorial exercise.

Consider the following example.

\begin{example}
    Let $\qv^{(1)}$ be the type $A_3$ quiver variety with $\dv^{(1)}=(1,2,1)$ and $\dw^{(1)}=(0,1,0)$. Let $\qv^{(2)}$ be the type $A_1$ quiver variety with $\dv^{(2)}=(2)$ and $\dw^{(2)}=(2)$. Let $\sspt_{1}$ be the second vertex and let $\sspt_{2}$ be the first (and only) vertex. As a $\bT^{(1)}$-representation, $\tb_{\sspt_{1}}=1+\hbar$, so $\qv^{(1)}$ is split over $\sspt_{1}$.
    
    Let $\qv=\qv^{(1)}{}_{\sspt_{1}} \slantsum_{\sspt_{2}} \qv^{(2)}$. It is a $D_4$ quiver variety with $\dv=(1,2,1,2)$ and $\dw=(0,1,0,0)$. Note that the framing is not minuscule, so this $D_4$ quiver variety is not treated in \cite{dinkjang}. In agreement with our choice of labeling of vertices, we will use $z_1,z_2,z_3$ (resp. $z_4$) for the K\"ahler variables of $\qv^{(1)}$ (resp. $\qv^{(2)}$).

In pictures, we have
    
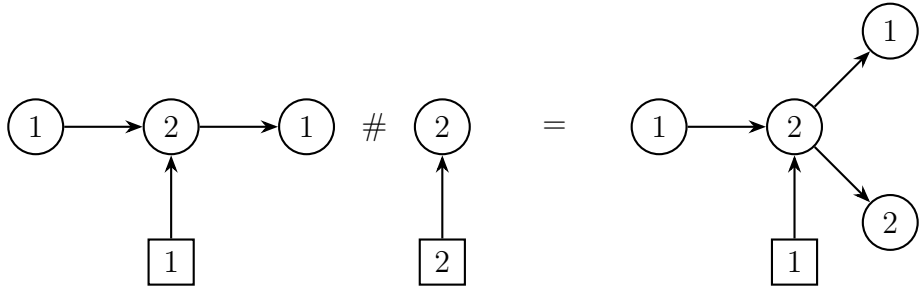
\[
\begin{tikzpicture}[->,>=Stealth,thick,node distance=1.8cm]

% left quiver gauge nodes
  \node[circle,minimum size=0.6cm,draw] (V1) {$1$};
  \node[circle,draw,minimum size=0.6cm, right of= V1] (V2){$2$}; 
    \node[circle,draw,minimum size=0.6cm, right of= V2] (V3){$1$}; 
    % left quiver framing
      \node[rectangle,minimum size=0.6cm, draw,below of=V2] (W1) {$1$};
        % left quiver arrows
  \draw[->] (W1) -- (V2);
   \draw[->] (V1) -- (V2);
    \draw[->] (V2) -- (V3);
    
% slant sum symbol
\node[right =0.2cm of V3] (ss){$\slantsum$}; 

      % Right quiver gauge node
  \node[circle,draw,minimum size=0.6cm,right = 0.2cm of ss] (V4) {$2$};
  % right quiver framing node
    \node[rectangle,minimum size=0.6cm, draw,below of=V4] (W2) {$2$};
  % right quiver arrow
   \draw[->] (W2) -- (V4);

% % Right half
 \node[right=0.8 of V4] (EE) {$=$};
 \node[circle,draw,minimum size=0.6cm, right=0.7 of EE] (V5){$1$};
 \node[circle,draw,minimum size=0.6cm, right of =V5] (V6){$2$};
 \node[circle,draw,minimum size=0.6cm, above right of =V6] (V7){$1$};
 \node[circle,draw,minimum size=0.6cm, below right of =V6] (V8){$2$};
  \node[rectangle,minimum size=0.6cm, draw,below of=V6] (W3) {$1$};

\draw[->] (W3) -- (V6);
   \draw[->] (V5) -- (V6);
    \draw[->] (V6) -- (V7);
      \draw[->] (V6) -- (V8);
 
% \node[circle,draw,minimum size=0.8cm, right of=V3] (V4) {$\dv_{\sspt_{2}}$};
% \node[rectangle,draw,minimum size=0.8cm,below=1 of V3] (W3) {$\dw_{\sspt_{2}}$};

%  \node[above left of=V3] (E){};
%    \node[above right of=V4] (F){};
%     \node[below right of=V4] (G){};
%      \node[below left of=V3] (H){};

    %  \draw[-,dotted, bend right=30] (V3) to (E);
    % \draw[-,dotted,bend left=30] (V3) to (H);
    % \draw[-,dotted,bend left=30] (V4) to (F);
    %   \draw[-,dotted,bend right=30] (V4) to (G);

\end{tikzpicture}
\]
The vertex function of $\qv^{(1)}$ was written in Example \ref{ex: (2,2) partition}. Using localization, one sees that the vertex function of $\qv^{(2)}$ is
    \[
\msver_{\qv^{(2)}}(z)=\sum_{d_1, d_2 \geq 0} \left( \kappa^{2} z_4 \right)^{d_1+d_2} \left(\prod_{i,j=2}^{2}\frac{\left(\hbar \frac{a_{i}}{a_{j}}\right)_{d_{i}}}{\left(q \frac{a_{i}}{a_{j}}\right)_{d_{i}}}\right) \left(\prod_{i,j=1}^{2}\frac{\left(q \frac{a_{i}}{a_{j}}\right)_{d_{i}-d_{j}}}{\left(\hbar \frac{a_{i}}{a_{j}}\right)_{d_{i}-d_{j}}} \right)
    \]
see, for example \cite{dinkms1}. It is not at all clear from the formula that $\ver_{\qv^{(2)}}$ in independent of $a_{1}$ and $a_{2}$. Nevertheless, it was shown in \cite{dinkms1} that 
    \[
\msver_{\qv^{(2)}}(z)=\frac{\Phi\left(\hbar \kappa^{2}  z_4\right) \Phi\left(\hbar \kappa z_4  \right)}{\Phi\left(\kappa^{2} z_4 \right) \Phi\left(\kappa z_4  \right)}
    \]
    in agreement with Conjecture \ref{conj: 0 dimensional vertex}. This is equivalent to the statement that the Macdonald-Ruijsenaars operators of row type act diagonally on the Macdonald polynomial for the empty partition with a certain eigenvalue, see \cite{NoumiSano}.

By Corollary \ref{cor: vertex slant sum factorization} and accounting for the various normalizations, we obtain

\[
\msver_{\qv}(z)= \frac{\Phi\left(\hbar \kappa^2 z_2 z_4 \right) \Phi(\hbar \kappa z_1 z_2 z_4) \Phi(\hbar \kappa^3 z_2 z_3 z_4 ) \Phi(\hbar \kappa^2 z_1 z_2 z_3 z_4 )}{\Phi( \kappa^2 z_2 z_4 ) \Phi( z_1 z_2 z_4 q/\hbar) \Phi( \kappa^3 z_2 z_3 z_4) \Phi(\kappa^2 z_1 z_2 z_3 z_4 )} 
\cdot \frac{\Phi\left(\hbar \kappa^2 z_4  \right) \Phi\left(\hbar \kappa z_4  \right)}{\Phi\left( \kappa^2 z_4 \right) \Phi\left(\kappa z_4 \right)}
\]
Associated to $\qv$ are the coweights $\lambda=\fundwt_{2}^\vee$ and $\mu=\fundwt_{2}^\vee-2\fundwt_{4}^\vee$. Using the explicit construction of the $D_{4}$ root system, it is easy to calculate that
\[
\langle\alpha,\mu\rangle=
\begin{cases}
    2 & \alpha=-\alpha_{4} \\
    1 & \alpha \in \{-\alpha_2-\alpha_4,-\alpha_1-\alpha_2-\alpha_4,-\alpha_2-\alpha_3-\alpha_4,-\alpha_1-\alpha_2-\alpha_3-\alpha_4\} \\
    \leq 0 & \text{otherwise}
\end{cases}
    \]
for any negative root $\alpha$. So we again confirm Conjecture \ref{conj: 0 dimensional vertex}.
\end{example}

We provide one more example of our slant sum constructions where the corresponding Kac-Moody algebra is of indefinite type.

\begin{example}
  Let $\qv^{(1)}$ be the type $D$ quiver variety with $\dv^{(1)}=(1,1,1,2,3,2,2)$ and $\dw^{(1)}=(0,0,\dots,1)$:

      
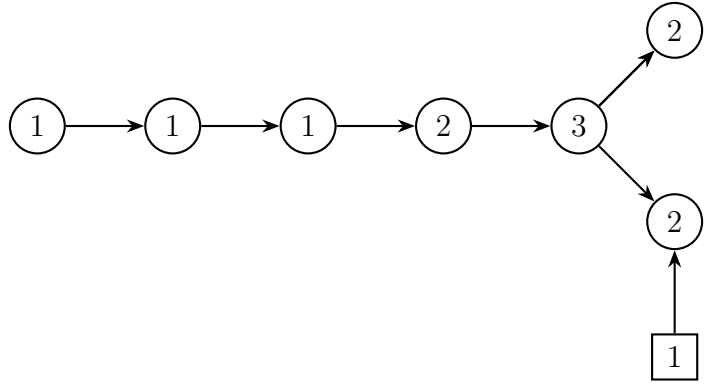
\[
\begin{tikzpicture}[->,>=Stealth,thick,node distance=1.8cm]

 \node[circle,draw,minimum size=0.6cm] (V1){$1$};
 \node[circle,draw,minimum size=0.6cm, right of =V1] (V2){$1$};
  \node[circle,draw,minimum size=0.6cm, right of =V2] (V3){$1$};
   \node[circle,draw,minimum size=0.6cm, right of =V3] (V4){$2$};
    \node[circle,draw,minimum size=0.6cm, right of =V4] (V5){$3$};
 \node[circle,draw,minimum size=0.6cm, above right of =V5] (V6){$2$};
 \node[circle,draw,minimum size=0.6cm, below right of =V5] (V7){$2$};

  \node[rectangle,minimum size=0.6cm, draw,below of=V7] (W1) {$1$};

\draw[->] (W1) -- (V7);
   \draw[->] (V1) -- (V2);
      \draw[->] (V2) -- (V3);
         \draw[->] (V3) -- (V4);
            \draw[->] (V4) -- (V5);
                  \draw[->] (V5) -- (V6);
   \draw[->] (V5) -- (V6);
    \draw[->] (V5) -- (V7);

\end{tikzpicture}
\]

%   \begin{tikzcd}
%                      &                      &                      &                      &                                           & 2                    \\
% 1 \arrow[r, ] & 1 \arrow[r, ] & 1 \arrow[r, ] & 2 \arrow[r, ] & 3 \arrow[ru, ] \arrow[rd, ] &                      \\
%                      &                      &                      &                      &                                           & 2 \arrow[d, ] \\
%                      &                      &                      &                      &                                           & \boxed{1}           
% \end{tikzcd}

This is an example of a minuscule framing. Let $\sspt_{1,1}$ be the unframed spin node and $\sspt_{1,2}$ be the framed spin node. 

Let $\qv^{(2)}$ have a single vertex and no loops, with $\dv^{(2)}=(2),\dw^{(2)}=(2)$, and let $\sspt_2$ be this single vertex. Let $\qv^{(3)}$ be another copy of this same quiver variety and let $\sspt_{3}$ be the vertex.

Consider the quiver variety 

\[
\qv:=\left(\qv^{(1)}{}_{\sspt_{1,1}}\slantsum_{\sspt_2}\qv^{(2)}\right){}_{\sspt_{1,2}}\slantsum_{\sspt_3}\qv^{(3)}
\]
This looks like 

     
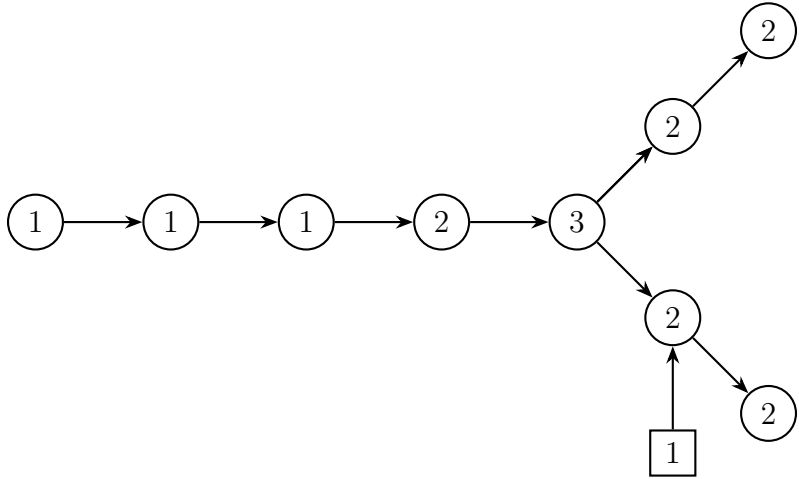
\[
\begin{tikzpicture}[->,>=Stealth,thick,node distance=1.8cm]

 \node[circle,draw,minimum size=0.6cm] (V1){$1$};
 \node[circle,draw,minimum size=0.6cm, right of =V1] (V2){$1$};
  \node[circle,draw,minimum size=0.6cm, right of =V2] (V3){$1$};
   \node[circle,draw,minimum size=0.6cm, right of =V3] (V4){$2$};
    \node[circle,draw,minimum size=0.6cm, right of =V4] (V5){$3$};
 \node[circle,draw,minimum size=0.6cm, above right of =V5] (V6){$2$};
 \node[circle,draw,minimum size=0.6cm, below right of =V5] (V7){$2$};
  \node[circle,draw,minimum size=0.6cm, below right of =V7] (V8){$2$};
 \node[circle,draw,minimum size=0.6cm, above right of =V6] (V9){$2$};
 
  \node[rectangle,minimum size=0.6cm, draw,below of=V7] (W1) {$1$};

\draw[->] (W1) -- (V7);
   \draw[->] (V1) -- (V2);
      \draw[->] (V2) -- (V3);
         \draw[->] (V3) -- (V4);
            \draw[->] (V4) -- (V5);
                  \draw[->] (V5) -- (V6);
   \draw[->] (V5) -- (V6);
    \draw[->] (V5) -- (V7);
     \draw[->] (V6) -- (V9);
      \draw[->] (V7) -- (V8);

\end{tikzpicture}
\]

The corresponding Kac-Moody algebra is of indefinite type and is not hyperbolic, see \cite{Kac}, exercise 4.2. We will calculate the vertex function of $\qv$ using Corollary \ref{cor: vertex slant sum factorization}.

% \[
% \begin{tikzcd}
%             &             &             &             &                         &                      & 2 \\
%             &             &             &             &                         & 2 \arrow[ru]         &   \\
% 1 \arrow[r] & 1 \arrow[r] & 1 \arrow[r] & 2 \arrow[r] & 3 \arrow[ru] \arrow[rd] &                      &   \\
%             &             &             &             &                         & 2 \arrow[rd]         &   \\
%             &             &             &             &                         &                      & 2 \\
%             &             &             &             &                         & \boxed{1} \arrow[uu] &  
% \end{tikzcd}
% \]

Conjecture \ref{conj: 0 dimensional vertex} claims that
\[
\msver_{\qv^{(1)}}=\prod_{m \in S^{(1)}} \frac{\Phi(\hbar m)}{\Phi(m)} 
\]
where 
% =\frac{\Phi(z_1)\Phi(z_6)\Phi(z_1z_2)\Phi(z_5z_6)\Phi(z_4z_5z_6)\Phi(z_5z_6z_7)\Phi(z_4z_5z_6z_7)\Phi(z_4z_5^2z_6z_7)\Phi(z_1z_2z_3z_4z_5z_6)\Phi(z_1z_2z_3z_4z_5z_6z_7)\Phi(z_1z_2z_3z_4z_5^2z_6z_7)\Phi(z_1z_2z_3z_4^2z_5^2z_6z_7)}{\Phi(z_1)\Phi(z_6)\Phi(z_1z_2)\Phi(z_5z_6)\Phi(z_4z_5z_6)\Phi(z_5z_6z_7)\Phi(z_4z_5z_6z_7)\Phi(z_4z_5^2z_6z_7)\Phi(z_1z_2z_3z_4z_5z_6)\Phi(z_1z_2z_3z_4z_5z_6z_7)\Phi(z_1z_2z_3z_4z_5^2z_6z_7)\Phi(z_1z_2z_3z_4^2z_5^2z_6z_7)}
\begin{multline*}
S^{(1)}=\{z_1,\kappa^2 z_6,z_1z_2,\kappa z_5 z_6, z_4 z_5 z_6, \kappa^3 z_5 z_6 z_7, \kappa^2 z_4 z_5 z_6 z_7, \kappa z_4 z_5^2 z_6 z_7, \\ \kappa^{-1} z_1 z_2 z_3 z_4 z_5 z_6, \kappa z_1 z_2 z_3 z_4 z_5 z_6 z_7,  z_1 z_2 z_3 z_4 z_5^2 z_6 z_7, \kappa^{-1} z_1 z_2 z_3 z_4^2 z_5^2 z_6 z_7 \}
\end{multline*}
and $\kappa=q/\hbar$.
This formula was proven in \cite{dinkjang}. Using $z_8$ (resp. $z_9$) as the K\"ahler parameter of $\qv^{(2)}$ (resp. $\qv^{(3)}$), we have
\[
\msver^{(2)}=\frac{\Phi(\hbar \kappa z_8) \Phi(\hbar \kappa^2 z_8)}{\Phi(\kappa z_8) \Phi(\kappa^2 z_8)}, \quad \msver^{(3)}=\frac{\Phi(\hbar \kappa z_9) \Phi(\hbar \kappa^2 z_9)}{\Phi(\kappa z_9) \Phi(\kappa^2 z_9)}
\]

By Corollary \ref{cor: vertex slant sum factorization}, we have 
\[
\msver_{\qv}=\prod_{m \in S} \frac{\Phi(\hbar m)}{\Phi(m)}
\]
where 
\begin{multline*}
    S=\{z_1,\kappa^2 z_6 z_8 ,z_1z_2,\kappa z_5 z_6 z_8, z_4 z_5 z_6 z_8, \kappa^3 z_5 z_6 z_7 z_8 z_9, \kappa^2 z_4 z_5 z_6 z_7 z_8 z_9, \kappa z_4 z_5^2 z_6 z_7 z_8 z_9,   \\ 
    \kappa^{-1} z_1 z_2 z_3 z_4 z_5 z_6 z_8, \kappa z_1 z_2 z_3 z_4 z_5 z_6 z_7 z_8 z_9,  z_1 z_2 z_3 z_4 z_5^2 z_6 z_7 z_8 z_9, \kappa^{-1} z_1 z_2 z_3 z_4^2 z_5^2 z_6 z_7 z_8 z_9, \\
    \kappa z_8, \kappa^2 z_8, \kappa z_9, \kappa^2 z_9
    \}
\end{multline*}

% \[
% \msver_{\qv}=\frac{\Phi(z_1)\Phi(z_8)^2\Phi(z_9)^2\Phi(z_6z_8)\Phi(z_1z_2)\Phi(z_5z_6z_8)\Phi(z_4z_5z_6z_8)\Phi(z_5z_6z_7z_8z_9)\Phi(z_4z_5z_6z_7z_8z_9)\Phi(z_4z_5^2z_6z_7z_8z_9)\Phi(z_1z_2z_3z_4z_5z_6z_8)\Phi(z_1z_2z_3z_4z_5z_6z_7z_8z_9)\Phi(z_1z_2z_3z_4z_5^2z_6z_7z_8z_9)\Phi(z_1z_2z_3z_4^2z_5^2z_6z_7z_8z_9)}{\Phi(z_1)\Phi(z_8)^2\Phi(z_9)^2\Phi(z_6z_8)\Phi(z_1z_2)\Phi(z_5z_6z_8)\Phi(z_4z_5z_6z_8)\Phi(z_5z_6z_7z_8z_9)\Phi(z_4z_5z_6z_7z_8z_9)\Phi(z_4z_5^2z_6z_7z_8z_9)\Phi(z_1z_2z_3z_4z_5z_6z_8)\Phi(z_1z_2z_3z_4z_5z_6z_7z_8z_9)\Phi(z_1z_2z_3z_4z_5^2z_6z_7z_8z_9)\Phi(z_1z_2z_3z_4^2z_5^2z_6z_7z_8z_9)}
% \]

In the Kac-Moody algebra, the associated coweights are $\lambda=\fundwt_7^\vee$ and 
\[
\mu=-\fundwt_1^\vee+\fundwt_3^\vee+\fundwt_6^\vee+2\fundwt_7^\vee-2\fundwt_8^\vee-2\fundwt_9^\vee
\]

We can observe directly that all roots $\alpha$ associated to the monomials of $S$ satisfy $\langle\alpha,\mu\rangle>0$ and with the appropriate magnitude. Using Lemma \ref{lem: conj 1.4 at the level of dimensions}, it is a straightforward exercise to check that these are the only negative, real roots satisfying this property.

\end{example}

\subsection{Branching of the nonstationary Ruijsenaars function}\label{sec: ruijsenaars branching}

For $n\geq 2$, let $X_n$ be the quiver variety $\qv_{Q,\theta}(\dv,\dw)$ with $\dv=(1,2,\dots,n-1)$ and $\dw=(0,\dots,0,n)$, and $\theta=(1,1,\ldots,1)$ for the $A_{n-1}$ quiver:

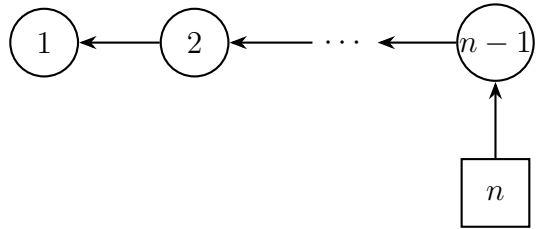
\[
\begin{tikzpicture}[->,>=Stealth,thick,node distance=2cm]

  % Gauge nodes (circles)
  \node[circle,draw,minimum size=0.9cm] (V1) {$1$};
  \node[circle,draw,minimum size=0.9cm,right of=V1] (V2) {$2$};
\node[right of=V2] (V3) {$\cdots$};
 \node[circle,draw,minimum size=0.9cm,inner sep=0,right of=V3] (V4) {$n-1$};

  % Framing nodes (squares)
 
  \node[rectangle,draw,minimum size=0.9cm,below of=V4] (W4) {$n$};

  % Arrows between gauge nodes
  \draw[->] (V2) to (V1);
    \draw[->,] (V3) to (V2);
    \draw[->,] (V4) to (V3);

  % Framing arrows
\draw[->] (W4) -- (V4);
 
\end{tikzpicture}
\]

It is known that $X_n \cong T^*\mathcal F\ell(\C^n)$, the cotangent bundle to the variety parameterizing  quotients $\C^{n} \twoheadrightarrow V_{n-1} \twoheadrightarrow\ldots \twoheadrightarrow V_{1}$ where $\dim V_{i}=i$. The slant sum construction can be used to relate $X_n$ and $X_{n-1}$. 

Let $Y_{n-1}$ be the quiver variety $\qv_{Q',\theta'}(\dv',\dw')$ with $\dv'=(n-1)$ and $\dw'=(n)$ for the quiver with a single vertex and no edges:

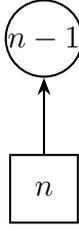
\[
\begin{tikzpicture}[->,>=Stealth,thick,node distance=2cm]

 % Gauge node for grassmannian
\node[circle,draw,minimum size=0.9cm,inner sep=0,left of=V1] (V0) {$n-1$};
 
  % Framing nodes (squares)

   \node[rectangle,draw,minimum size=0.9cm,below of=V0] (W0) {$n$};

  % Framing arrows
\draw[->] (W0) -- (V0);
 
\end{tikzpicture}
\]
with $\theta'=(1)$. It is known that $Y_{n-1}$ is the cotangent bundle to the Grassmannian parameterizing quotients $\C^{n} \twoheadrightarrow V$ where $\dim V=n-1$. 

Let $\sspt_1$ be the single vertex in the quiver corresponding to $Y_{n-1}$, and $\sspt_2$ be the rightmost vertex of the $A_{n-1}$ quiver. By definition, the slant sum over the compatible vertices $\sspt_{1}$ and $\sspt_{2}$ is
\[
X_{n}=Y_{n-1} \slantsum X_{n-1}
\]
Informally, this equation says that the flag variety in $\C^{n}$ can be built by attaching a flag variety in $\C^{n-1}$ to a Grassmannian. This is not literally true, but the point of this paper is that this can be made sense of at the level of torus fixed points.

The torus $\bT^{(1)}=(\Cs)^{n} \times \Cs_{\hbar}$ acts on $Y_{n-1}$. Let $a_1,\ldots,a_{n}$ denote the equivariant parameters of the first factor. Let $\tbw_{n-1}$ and $\tb_{n-1}$ denote the tautological bundles on $Y_{n-1}$ corresponding to the framing and gauge vertices, respectively. Let $p^{(1)}\in Y_{n-1}^{\bT^{(1)}}$ be the torus fixed point such that
\[
\tb_n\big|_{p^{(1)}} = a_1+\dots+a_{n-1}
\]

The torus $\bT^{(2)}=(\Cs)^{n-1} \times \Cs_{\hbar}$ acts on $X_{n-1}$. Let $b_1,\ldots,b_{n-1}$ denote the equivariant parameters for the first factor. Let $\tbw_{n-2}$ and $\tb_i$ for $1 \leq i \leq n-2$ denote the tautological bundles on $X_{n-1}$. Let $p^{(2)}\in X_{n-1}^{\bT^{(2)}}$ be the fixed point such that
\[
\tb_i \big|_{p^{(2)}}= b_1+\dots+b_i
\]
We choose the chamber for $\bT^{(1)}$ which orders the weights of $\tb_{n-1}|_{p^{(1)}}$ as
\[
a_1<\dots<a_{n-1}
\]
Let $p=p^{(1)}\slantsum p^{(2)}$.

Applying Theorem \ref{thm: vertex and slant sum 2} with trivial descendant, we obtain a formula for the vertex function of $X_n$ at $p$. 

Recall from \cite{bottadink} or \cite{dinkms1} that there is a bijection
\[
\left( \qm_{p^{(1)}}^{d}\right)^{\bT^{(1)} \times \Cs_{q}}=\left\{(d_1,\ldots,d_{n-1}) \, \mid \, \sum_{i} d_{i}=d\right\}
\]
and for a tuple $f=(d_1,\ldots,d_{n-1})$,
\[
\frac{\vrs|_{f}}{\extpow( N_{\vir}|_{f}^{\vee})} z_{n-1}^{\deg f}=\left(-\frac{q}{\hbar^{1/2}}\right)^{n |f|}\left(\prod_{i=1}^{n-1} \prod_{j=1}^{n} \frac{\left(\hbar \frac{a_{i}}{a_{j}} \right)_{d_{i}}}{\left(q \frac{a_{i}}{a_{j}} \right)_{d_{i}}}\right) \left(\prod_{i,j=1}^{n-1} \frac{\left(q \frac{a_{i}}{a_{j}} \right)_{d_{i}-d_{j}}}{\left(\hbar \frac{a_{i}}{a_{j}} \right)_{d_{i}-d_{j}}} \right) z_{n-1}^{|f|}
\]

We also calculate
\[
\exppref_{p^{(2)}}=\exp\left(\ln(q)^{-1}\sum_{i=1}^{n-2} \ln(b_{1}\ldots b_{i}) \ln(z_{i})\right)
\]
and
\[
\Phipref_{p^{(2)}}=\prod_{1 \leq i < j \leq n-1}\frac{\Phi\left( q \frac{b_{j}}{b_{i}}\right)}{\Phi\left( \hbar \frac{b_{j}}{b_{i}}\right)}
\]

So
\[
\exppref_{p^{(2)}}|_{b_i= b_i q^{d_i}}= \exppref_{p^{(2)}} \prod_{i=1}^{n-2} z_{i}^{d_1+\ldots d_{i}}
\]
and 
\[
\Phipref_{p^{(2)}}|_{b_i= b_i q^{d_i}}=\Phipref_{p^{(2)}} \prod_{1 \leq i < j \leq n-1} \frac{\left( \hbar \frac{b_{j}}{b_{i}}\right)_{d_{j}-d_{i}}}{\left( q \frac{b_{j}}{b_{i}}\right)_{d_{j}-d_{i}}}
\]

Putting this all together, Theorem \ref{thm: vertex and slant sum 2} gives

\begin{multline*}
\ver_{p}=\sum_{d_{1},\ldots,d_{n-1} \geq 0} \left(\prod_{i=1}^{n-1} \prod_{j=1}^{n} \frac{\left(\hbar \frac{a_{i}}{a_{j}} \right)_{d_{i}}}{\left(q \frac{a_{i}}{a_{j}} \right)_{d_{i}}}\right) \left(\prod_{i,j=1}^{n-1} \frac{\left(q \frac{a_{i}}{a_{j}} \right)_{d_{i}-d_{j}}}{\left(\hbar \frac{a_{i}}{a_{j}} \right)_{d_{i}-d_{j}}} \right) \\
\left( \prod_{1 \leq i < j \leq n-1} \frac{\left( \hbar \frac{b_{j}}{b_{i}}\right)_{d_{j}-d_{i}}}{\left( q \frac{b_{j}}{b_{i}}\right)_{d_{j}-d_{i}}}\right) \left(-\frac{q}{\hbar^{1/2}}\right)^{n |f|} \left( \prod_{i=1}^{n-1} z_{i}^{d_1+\ldots d_{i}}\right) \ver_{p^{(2)}}|_{b_i=a_i q^{d_i}}
\end{multline*}
Up to changes in notation, this is exactly equation (3.11) from \cite{NSmac}. In fact, the similarity of (3.11) of \cite{NSmac} with the vertex function of the cotangent bundle of Grassmannian was part of this inspiration for this work.

\section{Mirror symmetry of Higgs and Coulomb branches}\label{subsec: Higgs and Coulomb branches} 
In this section we recall the basic facts about Higgs and Coulomb branches that we will use.

\

Recall that $\qv_{Q,\theta}(\dv,\dw)$ denotes a quiver variety corresponding to some quiver $Q=(Q_0,Q_1)$, $\dv,\dw \in \mathbb{N}^{Q_{0}}$ and a generic stability parameter $\theta \in \mathbb{Z}^{Q_0}$. As we discussed in Section \ref{sec: quiver varieties}, there is a natural (projective) morphism:
\begin{equation*}
\pi\colon \qv_{Q,\theta}(\dv,\dw) \to \qv_{Q,0}(\dv,\dw), \qquad
\end{equation*}
where $\qv_{Q,0}(\dv,\dw) = \mu^{-1}(0)/\!\!/G_{\dv}$
is the Hamiltonian reduction of ${\bf{M}}=T^*{\bf N}$, where ${\bf N}=\rep_{Q}(\dv,\dw)$, by the action of the group $G_{\dv}$.

\subsection{Higgs branches}  
The variety $\qv_{Q,0}(\dv,\dw)$ is the {\emph{Higgs branch}} of the \emph{quiver gauge theory} corresponding to the pair $(G_{\dv},{\bf{N}})$. More generally, an arbitrary pair $(G,{\bf{N}})$ of a reductive group $G$ and a finite dimensional representation ${\bf{N}}$ should define a certain gauge theory such that its Higgs branch $\mathcal{M}_H(G,{\bf{N}})$ is the Hamiltonian reduction $T^*{\bf{N}}/\!\!/\!\!/ G$.

One important feature of quiver gauge theories (i.e. those $(G,{\bf{N}})$ that arise from a choice of a quiver $Q$ as above) is that we can consider the GIT quotient $\qv_{Q,\theta}(\dv,\dw)$ depending on a stability parameter $\theta$. For generic $\theta$ (see Proposition \ref{prop: generic theta}) the variety $\qv_{Q,\theta}(\dv,\dw)$ is smooth, symplectic, and is a resolution of singularities of the image of $\pi\colon \qv_{Q,\theta}(\dv,\dw) \rightarrow \mathcal{M}_H(G,{\bf{N}})$. So, whenever the map $\pi$ is surjective (see \cite{Crawley-Boevey} for the criteria), we can extract the Higgs branch $\mathcal{M}_H(G,{\bf{N}})$ from the resolution $\qv_{Q,\theta}(\dv,\dw)$. However, in many cases (for example, if ${\dw}=0$), $\qv_{Q,\theta}(\dv,\dw)$ is empty but $\mathcal{M}_H(G,{\bf{N}})$ is highly nontrivial.   We will denote  $\mathcal{M}_H(G,{\bf{N}})$ by $\mathcal{M}_H$ whenever $(G,{\bf{N}})$ are clear.

\subsection{Coulomb branches} In \cite{BFNII}, Braverman, Finkelberg, and Nakajima proposed a mathematical definition of the {\emph{Coulomb branch}} $\mathcal{M}_C(G,{\bf{N}})$ associated to a gauge theory corresponding to a  pair $(G,{\bf{N}})$ as above. The variety $\mathcal{M}_C(G,{\bf{N}})$ is an affine Poisson variety. Even for nice quiver gauge theories (e.g., when $Q$ is ADE and $\dv, \dw$ are such that $\qv_{Q,\theta}(\dv,\dw)$ is nonempty), in contrast to $\mathcal{M}_H(G,{\bf{N}})$,  the varieties $\mathcal{M}_C(G,{\bf{N}})$ {\emph{do not}} admit symplectic resolutions. 

On the other hand, it is known that $\mathcal{M}_C(G,{\bf{N}})$ always has symplectic singularities, see \cite{Bel}.
%In particular, $\mathcal{M}_C(G,{\bf{N}})$ always admits a {\emph{$\mathbb{Q}$-factorial terminalization}} $\widetilde{\mathcal{M}}_C(G,{\bf{N}})$, which is a partial resolution of $\mathcal{M}_C(G,{\bf{N}})$ with nice properties (see \cite[Section 4]{Los_Mas_Mat} for the general discussion). 
We will denote  $\mathcal{M}_C(G,{\bf{N}})$  by $\mathcal{M}_C$ whenever $(G,{\bf{N}})$ are clear. 

\

Let us briefly recall the definition of $\mathcal{M}_C$.  Set $\mathcal{K}:=\mathbb{C}((z))$, $\mathcal{O}:=\mathbb{C}[[z]]$ and let $G_{\mathcal{K}}$, $G_\mathcal{O}$, ${\bf{N}}_{\mathcal{O}}$ be the corresponding spaces whose $\mathbb{C}$-points are $G((z))$, $G[[z]]$, and ${\bf{N}}[[z]]$ respectively. Set 
\begin{equation*}
\mathcal{T} := G_{\mathcal{K}} \times^{G_\mathcal{O}} {\bf{N}}_{\mathcal{O}},~\mathcal{R}=\{[g,n] \in \mathcal{T}\,|\, gn \in \mathbf{N}_{\mathcal{O}}\}. 
\end{equation*}
The group $G_{\mathcal{O}}$ acts naturally on the space $\mathcal{R}$. In \cite[2(ii)]{BFNII}, the authors define the equivariant Borel-Moore homology  $H_*^{G_{\mathcal{O}}}(\mathcal{R})$ of $\mathcal{R}$. Moreover, in \cite[Section 3]{BFNII}, they endow $H_*^{G_{\mathcal{O}}}(\mathcal{R})$ with the algebra structure given by convolution $*$. Finally, they prove that $*$ is commutative and define: 
\begin{equation*}
\mathcal{M}_C := \operatorname{Spec}(H_*^{G_{\mathcal{O}}}(\mathcal{R}),*).
\end{equation*}
One can also define a quantized Coulomb branch $\mathcal{A}_{C,\hbar}$ by adding an additional $\mathbb{C}^\times$-equivariance with respect to the ``loop rotation'' action:
\begin{equation*}
\mathcal{A}_{C,\hbar} := H_*^{G_{\mathcal{O}} \rtimes \mathbb{C}^\times}(\mathcal{R}),~ \mathcal{A}_{C}:=\mathcal{A}_{C,\hbar}|_{\hbar=1}.
\end{equation*}
Convolution equips $\mathcal{A}_{C,\hbar}$ with an associative algebra structure.  

%By $\widetilde{\mathcal{M}}_H, \widetilde{\mathcal{M}}_C$ we will de

\subsection{Torus action and fixed points on Coulomb branches}\label{sec: torus action on Coulomb}
The space $\mathcal{R}$ maps naturally to the affine Grassmannian $\operatorname{Gr}_G=G_{\mathcal{K}}/G_{\mathcal{O}}$. Connected components of $\operatorname{Gr}_G$ are  labeled by the lattice $\pi_1(G)$. So, $\mathbb{C}[\mathcal{M}_C]$ admits a grading by $\pi_1(G)$, and hence, the variety $\mathcal{M}_C$ admits a (Hamiltonian) action of the torus $\pi_1(G)^{\wedge}$, where by $\bullet^{\wedge}$ we mean the Pontryagin dual (see \cite[Section 3(v)]{BFNII} for the details). The cohomological grading on $\mathbb{C}[\mathcal{M}_C]=H_*^{G_\mathcal{O}}(\mathcal{R})$ defines the action of $\mathbb{C}^\times$ on $\mathcal{M}_C$.

A choice of the stability parameter $\theta\colon G \rightarrow \mathbb{C}^\times$ defines a cocharacter $\mathbb{C}^\times \rightarrow \pi_1(G)^{\wedge}$. Abusing notations, we will denote this cocharacter by $\theta$.
In \cite[Conjecture 3.25(1)]{BFN_slices} the authors conjecture that the set $\mathcal{M}_C^{\theta(\mathbb{C}^\times)}(\mathbb{C})$ of $\mathbb{C}$-points of the {\emph{schematic}} fixed points of $\mathcal{M}_C$ is either empty or consists of one point. From now on, we denote $\mathcal{M}_C^{\theta(\mathbb{C}^\times)}$ simply by $\mathcal{M}_C^{\theta}$.

\subsection{Mirror symmetry and some conjectures}\label{subsec: conjectures} 

%\hunter{I propose we move section 6 to be a subsection in section 5, since both sections discuss quiver gauge theories very generally.}

%\hunter{Vasily: Think about generic stability (see Nakajima instantons on ALE spaces Theorem 2.8) in terms of the dual} 

 In this section we formulate various conjectures relating Higgs and Coulomb branches for quiver gauge theories.

 %{\vasya{I should make this section more readable}}

 \

\subsubsection{Coulomb branches with non-singular torus fixed points} 3d mirror symmetry is a nontrivial relation between Higgs and Coulomb branches for $(G,{\bf{N}})$. From now on, we assume that our theory is a quiver gauge theory for a quiver $Q$ without edge loops, and we denote by $\mathfrak{g}=\mathfrak{g}_Q$ the corresponding Kac-Moody Lie algebra. We will denote the smooth symplectic variety $\qv_{Q,\theta}(\dv,\dw)$ by $\widetilde{\mathcal{M}}_H$ (recall that in many cases $\widetilde{\mathcal{M}}_H$ is a symplectic resolution of the Higgs branch $\mathcal{M}_H$). 

For $i \in Q_0$ we will denote by $\omega_i^\vee, \alpha_i^\vee$ the corresponding fundamental coweight and simple coroot for $\mathfrak{g}$.
Given $\dv,\dw \in \mathbb{N}^{Q_{0}}$, set:
\begin{equation}\label{eq: def lambda mu}
\lambda=\sum_i \dw_i\omega_i^\vee,~\mu=\lambda-\sum_{i}\dv_i\alpha_i^\vee
.
\end{equation}

One mirror symmetry statement is the Hikita conjecture \cite{hikita} predicting an identification of (graded) algebras:
\begin{equation}\label{conj: Hikita}
H^*(\widetilde{\mathcal{M}}_H) \simeq \mathbb{C}[\mathcal{M}_C^{\theta}].
\end{equation}

%So, it is natural to expect that $\mathcal{M}_C^{\theta}$ is nonempty if and only if $\widetilde{\mathcal{M}}_H$ is nonempty. 

\begin{comment}

We start with the following standard lemma.
\begin{lemma}\label{trivial cohomology implies affine space}
Let $\pi\colon X \rightarrow Y$ be a conical symplectic resolution. Assume that $H^*(X) = \mathbb{C}$. Then $X$ is isomorphic to $\mathbb{A}^{2k}$, where $2k=\operatorname{dim}X$.     
\end{lemma}
\begin{proof}
Let $p \in Y$ be the unique $\mathbb{C}^\times$-fixed point of $Y$ (recall that $Y$ contracts to $p$ via the action of $\mathbb{C}^\times$).
The $\mathbb{C}^\times$-action contracts  $X$  to $X^{\mathbb{C}^\times}=\pi^{-1}(p)^{\mathbb{C}^\times}$, so $H^*(X)=H^*(\pi^{-1}(p)^{\mathbb{C}^\times})$. Note now that $\pi^{-1}(p)^{\mathbb{C}^\times}$ is smooth and proper, so $H^{\mathrm{top}}(\pi^{-1}(p)^{\mathbb{C}^\times})= \mathbb{C}$.  It follows that $\pi^{-1}(p)$ consists of one point. We conclude that $X$ contracts to one point under the action of $\mathbb{C}^\times$, hence, by the Bialynicki-Birula theorem (see \cite{BB}), $X$ must be isomorphic to the affine space.  
\end{proof}

\end{comment}

%We make the following conjecture:
%\begin{conjecture}
%If the variety $\widetilde{\mathcal{M}}_H$ is a point then $p^! \in \mathcal{M}_{C}$ is nonsingular. 
% the $\theta$-fixed point $p^! \in \mathcal{M}_C$ is nonsingular.
%\hunter{I rephrased this to make it clear that part of the conjecture is that there is a fixed point. Is it clear?}\vasya{yes, great!}
%\end{conjecture}

 If $\widetilde{\mathcal{M}}_H$ is a point then, assuming (\ref{conj: Hikita}), we deduce that $\mathbb{C}[\mathcal{M}_C^{\theta}]=\mathbb{C}$ that suggests that $p^! \in \mathcal{M}_C$ should be nonsingular. It follows from Remark \ref{rem: quiver point mu extremal} below that $\widetilde{\mathcal{M}}_H$ is a point iff $\mu \in W\lambda$.  So, we arrive to the following Conjecture that can be formulated without referring to $\widetilde{\mathcal{M}}_H$ and so has a chance to be valid for arbitrary symmetrizable quivers $Q$ (where $\mathcal{M}_C$ is defined as in \cite{NW}).

\begin{conjecture}\label{conj: nonsingular Coulomb point Higgs}
We have $\mathcal{M}_C^\theta=\{p^!\}$ and  $p^! \in \mathcal{M}_C$ is nonsingular if $\mu \in W\lambda$.      
\end{conjecture}   

\begin{remark}
It's natural to expect that the implication of Conjecture \ref{conj: nonsingular Coulomb point Higgs} also holds in the opposite direction. Namely if $p^! \in \mathcal{M}_C$ is nonsingular then $\mu \in W\lambda$. Note that, assuming (\ref{conj: Hikita}) one would see that $H^*(\widetilde{\mathcal{M}}_H)=\mathbb{C}$. It's not hard to show that $H^*(\widetilde{\mathcal{M}}_H)=\mathbb{C}$ forces $\widetilde{\mathcal{M}}_H$ to be isomorphic to the affine space. 
\end{remark}

%If $p^! \in \mathcal{M}_C$ is nonsingular, then $\mathbb{C}[\mathcal{M}_C^{\theta}]$ must be equal to $\mathbb{C}$.  \vasya{example when opposite does not hold?} Assuming \eqref{conj: Hikita}, we deduce that $H^*(\widetilde{\mathcal{M}}_H)=\mathbb{C}$. Then Lemma \ref{trivial cohomology implies affine space} implies that $\widetilde{\mathcal{M}}_H$ is isomorphic to $\mathbb{A}^{2k}$. We expect that for quivers $Q$ without edge loops, the dimension of $\widetilde{\mathcal{M}}_H$ is actually equal to zero whenever $p^! \in \mathcal{M}_C$ is nonsingular.

%\vasya{can it happen that $(V(\lambda_1) \otimes \ldots \otimes V(\lambda_n))_\mu = \mathbb{C}$, but $\mu \notin W\lambda$?}

% In words, if $p^! \in \mathcal{M}_C$ is nonsingular, then (assuming the Hikita conjecture), $\widetilde{\mathcal{M}}_H$ must be isomorphic to an affine space.

%the dual condition to  being a point (or maybe, more generally, an affine space) should be that the (unique)  $\theta$-fixed point of the corresponding Coulomb branch is nonsingular. 

%The starting point for all explict vertex function computations of this paper is a quiver variety $\widetilde{\mathcal{M}}_H$ together with an {\emph{isolated}} fixed point $p \in \widetilde{\mathcal{M}}_H^{\bT}$. As we will recall below, fixed point $p$ is itself a Nakajima quiver variety, so Conjecture \ref{conj: nonsingular Coulomb point Higgs} can be applied to $p$. Our 

\subsubsection{Tangent spaces to nonsingular fixed points on partial resolutions} 
Recall that Conjecture \ref{conj: dual tangent space} gives a formula for the tangent space $T_{p^!}\mathcal{M}_C$ when $\mu \in W\lambda$. Actually we expect more to be true: this formula should allow one to compute the character of the tangent space for an arbitrary {\emph{nonsingular}} fixed point of a partially resolved Coulomb branch. Let us explain the details.

Let $\bA$ be a torus acting on ${\bf{N}}$ commuting with $G$ and let $\nu\colon \mathbb{C}^\times \rightarrow \bA$ be a cocharacter. As in \cite[Section 3(ix)]{BFNII}, it defines a {\emph{partial}} resolution 
\begin{equation}\label{eq: part resol coulomb}
\widetilde{\mathcal{M}}_{C,\nu} \rightarrow \mathcal{M}_C.
\end{equation}
The action of $\pi_1(G)^{\vee}$ extends to $\widetilde{\mathcal{M}}_{C,\nu}$ and the morphism (\ref{eq: part resol coulomb}) is  $\pi_1(G)^{\vee}$-equivariant.  
On the Higgs side, cocharacter $\nu$ defines an action $\mathbb{C}^\times \curvearrowright \widetilde{\mathcal{M}}_H$. 

One can show that:
\begin{equation*}
\widetilde{\mathcal{M}}_{H}(G,{\bf{N}})^{\nu} = \bigsqcup_{\gamma\colon  \mathbb{C}^\times \rightarrow G} \widetilde{\mathcal{M}}_H(Z_{G}(\gamma),{\bf{N}}^{(\gamma,\nu)}),
\end{equation*}
where the disjoint union is taken over all conjugacy classes of cocharacters $\gamma\colon \mathbb{C}^\times \rightarrow G$.
Moreover, every $(Z_{G}(\gamma),{\bf{N}}^{(\gamma,\nu)})$ corresponds to some quiver, hence, each $\widetilde{\mathcal{M}}_H(Z_{G}(\gamma),{\bf{N}}^{(\gamma,\nu)})$ itself is a Nakajima quiver variety, see, for example \cite[Proposition 2.13]{Dum_Kr}.

%in particular, is connected (by \cite[Section 1]{Crawley-Boevey}). 

Assume now that $\widetilde{\mathcal{M}}_{H}^{\nu}$ contains an isolated fixed point $p$. By the above, this point itself is equal to the Nakajima quiver variety corresponding to some choice of $\gamma$. Assuming Conjecture \ref{conj: nonsingular Coulomb point Higgs} holds, we see that the Coulomb branch $\mathcal{M}_C(Z_G(\gamma),{\bf{N}}^{(\gamma,\nu)})$ has a unique nonsingular $\theta$-fixed point $p^!$. 
%\hunter{Can we say that it is a nonsingular fixed point? Also can we clarify that it is torus fixed by the torus $\theta(\Cs)$?}

The following conjecture is due to Justin Hilburn (see \cite[Conjecture 1.6]{KPW}).
\begin{conjecture}\label{conj: open embedding}
%(1)Point $p^! \in \mathcal{M}_C(Z_G(\gamma),N^{(\gamma,\nu)})$ is nonsingular iff .     

There is a $\pi_1(G)^{\wedge}$-equivariant open embedding:
\begin{equation}\label{eq:Hilburn_embedding}
j\colon \mathcal{M}_C(Z_G(\gamma),{\bf{N}}^{(\gamma,\nu)}) \hookrightarrow \widetilde{\mathcal{M}}_{C,\nu}. 
\end{equation}
for each $\gamma$.
\end{conjecture}

Assuming both Conjectures \ref{conj: nonsingular Coulomb point Higgs} and \ref{conj: open embedding} hold, one obtains the following corollary.

\begin{corollary}\label{cor: iso of tangent spaces}
The point $j(p^!) \in \widetilde{\mathcal{M}}_{C,\nu}$ is nonsingular and $j$ induces an isomorphism of tangent spaces at $p^!$ and $j(p^!)$.    
\end{corollary}

Recall now that Conjecture \ref{conj: dual tangent space} describes the tangent space $T_{p^!}\mathcal{M}_C(Z_G(\gamma),{\bf{N}}^{(\gamma,\nu)})$. So, combining it with Conjecture \ref{conj: open embedding}, we should obtain a conjectural description of the tangent space to a  nonsingular $\theta$-fixed point of $\widetilde{\mathcal{M}}_{C,\nu}$. We expect that there should be a bijection between nonsingular $\theta$-fixed points of $\widetilde{\mathcal{M}}_{C,\nu}$ and isolated $\nu$-fixed points on $\widetilde{\mathcal{M}}_H$. If that is the case, then a tangent space to arbitrary non-singular $\theta$-fixed point of $\widetilde{\mathcal{M}}_{C,\nu}$ should have a description as above.

\begin{remark}
Note that we are not assuming that the torus $\bA$ arises from the action of $\prod_{i}GL(W_i)$ on ${\bf{N}}$ (i.e., that it is a {\emph{framing}} torus). If it does, then    $\mathcal{M}_C(Z_G(\gamma),{\bf{N}}^{(\gamma,\nu)})$ is a product of Coulomb branches for the {\emph{same}} quiver $Q$. 
%We do {\emph{not}} have a strong evidence that Conjecture \ref{eq:Hilburn_embedding} holds for non-framing tori. 
Note also, that for our purposes
%, however, 
it suffices to have a much weaker version of Conjecture \ref{eq:Hilburn_embedding}, namely that the completions of (\ref{eq:Hilburn_embedding}) at the corresponding fixed points are $\pi_1(G)^{\wedge}$-isomorphic. One possible way to prove this is by showing that this holds after deforming $\widetilde{\mathcal{M}}_{C,\nu}$.
\end{remark}

\section{Irreducible modules over some Coulomb branches}\label{section: character of simple} 
Assume that $Q,\mathsf{v},\mathsf{w}$ are such that $\mathcal{M}_C^\theta = \{p\}$ and $p \in \mathcal{M}_C$ is {\emph{nonsingular}}.  Recall that conjecturally this holds iff $\mu \in W\lambda$, where $\lambda, \mu$ are constructed from ${\mathsf{v}}, \mathsf{w}$ as in (\ref{eq: lambda and mu general}).

\

Fix $x \in \mathfrak{a}_Q$ and let 
\begin{equation*}
\mathcal{A}_x := H_*^{(G_{\mathcal{O}} \rtimes \mathbb{C}^\times) \times \mathsf{A}_Q}(\mathcal{R})_{(x,1)}
\end{equation*}
be the corresponding quantized Coulomb branch algebra. To simplify the notation, we denote $\mathcal{A}_x$ simply by $\mathcal{A}$. 
%{\vasya{say that $\mathcal{A}$ quantizes $\mathbb{C}[\mathcal{M}_C]$}}
Cocharacter $\theta$ determines the category $\mathcal{O}_\theta(\mathcal{A})$ as in \cite[Section 3.4]{catO_coulomb}.

\

It's not hard to see that the category $\mathcal{O}_\theta(\mathcal{A})$ contains one simple object (to be denoted by $L$, see Section \ref{sec:cat_O} below for the details). The goal of this section is to show that the normalized $\mathsf{a}_Q$-character of $L$ coincides with the character of $T_p^-\mathcal{M}_C$, in particular, it (and it's $\hbar^!$-deformed version) can be read off from the equation (\ref{eq: intro tangent formula}) whenever Conjecture \ref{conj: dual tangent space} holds.

\subsection{Category $\mathcal{O}$}\label{sec:cat_O}
Recall that the cocharacter $\theta$ defines the category $\mathcal{O}_\theta(\mathcal{A})$.
By, \cite[Lemma 3.27]{catO_coulomb}, irreducible objects of this category are in bijection with irreducible modules over the (finite-dimensional) algebra $\mathsf{C}_\theta(\mathcal{A})$ called {\emph{Cartan subquotient}} or \emph{$B$-algebra} and defined as follows:
\begin{equation*}
\mathsf{C}_\theta(\mathcal{A}) := \mathcal{A}_{0}/\sum_{i<0}\mathcal{A}_{i}\mathcal{A}_{-i},
\end{equation*}
where index $i$ corresponds to the $\mathbb{Z}$-grading on $\mathcal{A}$ induced by the cocharacter $\theta$. It follows from Lemma \ref{lemma_descr_grtr} below that $\mathsf{C}_\theta(\mathcal{A})=\mathbb{C}$, so $\mathcal{O}_\theta(\mathcal{A})$ has the unique irreducible object to be denoted $L$.

\begin{remark}
Alternatively, assuming $\mu \in W\lambda$ one can use   \cite[Theorem 9.14]{categorical_yang_modules} to see that  $\mathcal{O}_\theta(\mathcal{A})$ has a unique irreducible object. We are to grateful to Joel Kamnitzer for pointing this out.  
\end{remark}

We have an embedding $\mathsf{a}_Q \subset \mathcal{A}$ inducing the action $\mathsf{a}_Q \curvearrowright L$. By {\emph{normalized}} $\mathsf{a}_Q$-character of $L$ we  mean the $\mathsf{a}_Q$-character of $L$ normalized in such a way that it starts with $1$.

\subsection{Completions and a formal version of $L$}
The algebra $\mathcal{A}$ is $\mathbb{Z}$-filtered and the filtration is complete and separated. For $k \in \mathbb{Z}$ let us denote the $k$'th filtration term by $F^k\mathcal{A}$. Algebra $\mathcal{A}$ is also $\mathbb{Z}$-graded, where the grading is induced by the $\theta(\mathbb{C}^\times)$-action. For $i \in \mathbb{Z}$ we denote by $\mathcal{A}_i$ the corresponding graded component. 
Set: 
\begin{equation*}
\widehat{\mathcal{A}}= \underset{\underset{k \rightarrow -\infty}{\longleftarrow}}{\operatorname{lim}}\, \mathcal{A}/F^k\mathcal{A}, 
\end{equation*} 
algebra $\widehat{\mathcal{A}}$ is the completion of $\mathcal{A}$ w.r.t. the filtration $F^\bullet\mathcal{A}$. Filtration $F^\bullet$ on $\mathcal{A}$ induces the filtration on $\widehat{\mathcal{A}}$ to be denoted by the same symbol. We have $\operatorname{gr}\widehat{\mathcal{A}}=\operatorname{gr}\mathcal{A}=\mathbb{C}[\mathcal{M}_C]$.

For a $\mathbb{Z}$-filtered vector space $F^\bullet V$, we set $\operatorname{R}_\hbar(V):=\bigoplus_{k \in \mathbb{Z}}(F^kV) \hbar^k \subset V[\hbar^{\pm 1}]$ and set $\widehat{\operatorname{R}}_\hbar(V) := \underset{\underset{k \rightarrow +\infty}{\longleftarrow}}{\operatorname{lim}}\, \operatorname{R}_\hbar(V)/(\hbar^k)$.

Set: 
\begin{equation*}
\widehat{\mathcal{A}}_\hbar = \widehat{R}_\hbar(\widehat{\mathcal{A}}).
\end{equation*}

We have $\widehat{\mathcal{A}}_\hbar/(\hbar)=\operatorname{gr}\widehat{\mathcal{A}}=\mathbb{C}[\mathcal{M}_C]$. Let $\mathfrak{m}_p \subset \mathbb{C}[\mathcal{M}_C]$ be the maximal ideal of $p \in \mathcal{M}_C$ and let $\mathfrak{m}_{p,\hbar} \subset \widehat{\mathcal{A}}_\hbar$ be its preimage. 
Set 
\begin{equation*}
\widehat{\mathcal{A}}_{\hbar,\widehat{p}} := \underset{\underset{n \rightarrow +\infty}{\longleftarrow}}{\operatorname{lim}}\, \widehat{\mathcal{A}}_{\hbar}/\mathfrak{m}_{p,\hbar}^n.
\end{equation*}

%\vasya{mention that $\widehat{\mathcal{A}}_\hbar$ can be identified with the completion of $\mathcal{A}$ no need to deal with $\widehat{\mathcal{A}}$ first?}

Let's describe the algebra $\widehat{\mathcal{A}}_{\hbar,\widehat{p}}$. First of all recall that $\mathcal{A}$ admits an action of $\mathsf{A}_{Q}$, it induces the action on $\widehat{\mathcal{A}}_{\hbar,\widehat{p}}$. 
%Consider the $\theta(\mathbb{C}^\times) \times \mathbb{C}^\times_\hbar$-graded module $T_p^*\mathcal{M}_C$, and let $(\chi_1,q_1),\ldots,(\chi_{d},q_d)$ be its  $\theta(\mathbb{C}^\times) \times \mathbb{C}^\times_\hbar$-weights (with multiplicities).

For a vector space $V$ equipped with a skew-symmetric pairing $(\,,\,)$
let $\widehat{\mathbb{W}}_{\hbar}(V)$ be the formal Weyl algebra over $\mathbb{C}[[\hbar]]$ defined as follows: 
\begin{equation*}
\widehat{\mathbb{W}}_{\hbar}(V) = \mathbb{C}[[V,\hbar]]
\end{equation*}
subject to relations 
\begin{equation}\label{eq: formal Weyl}
[\hbar,V]=0,[a,b]=\hbar(a,b),~a,b \in V.
\end{equation}

%\begin{remark}
%As a vector space,  $\widehat{\mathbb{W}}_{\hbar}(V)$ can be identified with  $\mathbb{C}[[V,\hbar]]$, multiplication can be extracted from the relations above.    
%\end{remark}

We will be interested in the vector space $V=T^*_p\mathcal{M}_C$ with a pairing $(\,,\,)$ induced by the sympletic form on $\mathcal{M}_C^{\mathrm{reg}}$. Note that the torus $\mathsf{A}_Q$ acts on $T^*_p\mathcal{M}_C$ inducing the action on $\widehat{\mathbb{W}}_\hbar(T^*_p\mathcal{M}_C)$ (by algebra automorphisms).

\begin{lemma}\label{lem: completion is W}
There exists a $\mathsf{A}_Q$-equivariant isomorphism of algebras over $\mathbb{C}[[\hbar]]$:
\begin{equation*}
\widehat{\mathbb{W}}_\hbar(T^*_p\mathcal{M}_C) \simeq \widehat{\mathcal{A}}_{\hbar,\widehat{p}}.
\end{equation*}
\end{lemma}
\begin{proof}
Algebra $\widehat{\mathcal{A}}_{\hbar,\widehat{p}}$ is a formal quantization of $\mathbb{C}[\widehat{p}]=\underset{\underset{n \rightarrow +\infty}{\longleftarrow}}{\operatorname{lim}}\, \mathbb{C}[\mathcal{M}_C]/\mathfrak{m}_{p}^n$. Note that $\mathbb{C}[\widehat{p}]$ as a $\mathsf{A}_Q$-graded Poisson algebra is isomorphic to $\mathbb{C}[[T_p\mathcal{M}_C]]$. Any formal quantization of $\mathbb{C}[[T_p\mathcal{M}_C]]$ is isomorphic to $\widehat{\mathbb{W}}_\hbar(T^*_p\mathcal{M}_C)$ and this isomorphism can be made ${\mathsf{A}}_Q$-equivariant (see \cite[Lemma 1.5]{BK} and \cite[Lemma 5.2]{BLPWII}).
%Let us first of all show that that $\widehat{\mathcal{A}}_{\hbar,\widehat{p}}$ is a quantization of the Poisson algebra $\mathbb{C}[\widehat{p}]=\underset{\underset{n \rightarrow +\infty}{\longleftarrow}}{\operatorname{lim}}\, \mathbb{C}[\mathcal{M}_C]/\mathfrak{m}_{p}^n$. We have a natural map $\widehat{\mathcal{A}}_{\hbar,\widehat{p}}/(\hbar) \rightarrow \mathbb{C}[\widehat{p}]$... 
\end{proof}

\begin{comment}
\begin{lemma}
Assume that $V^{\mathbb{T}}=\{0\}$, then $\mathbb{C}[[V]]^{\mathrm{fin}}=\mathbb{C}[V]$.
\end{lemma}
\begin{proof}
Clearly, $\mathbb{C}[V] \subset \mathbb{C}[[V]]^{\mathrm{fin}}$. To check that $\mathbb{C}[V] = \mathbb{C}[[V]]^{\mathrm{fin}}$ it's enough to show that every $\mathbb{T}$-weight space of $\mathbb{C}[V]$ is {\emph{finite-dimensional}}. Pick some $\mathbb{T}$-weight $\eta$. 
\end{proof}
\end{comment}

\begin{comment}
\begin{lemma}
The subalgebra of $\theta(\mathbb{C}^\times) \times \mathbb{C}^\times_\hbar$-finite vectors in $\widehat{\mathbb{W}}_\hbar(T_p\mathcal{M}_C)$ identifies with $\mathbb{W}_\hbar(T_p\mathcal{M}_C)$.
\end{lemma}
\begin{proof}
As $\theta(\mathbb{C}^\times) \times \mathbb{C}^\times_\hbar$-module, $\widehat{\mathbb{W}}_\hbar$ is isomorphic to $\mathbb{C}[[T_p\mathcal{M}_C,T_p\mathcal{M}_C^*,\hbar]]$. The claim follows from Lemma \ref{}. 
\end{proof}
\end{comment}

Recall now that whenever a torus $\mathbb{T}$ acts on some vector space $R$, we can consider a vector subspace $R^{\mathrm{fin}} \subset R$ generated by the eigenvectors of $\mathbb{T} \curvearrowright R$. We will call $R^{\mathrm{fin}} \subset R$ the subspace of $\mathbb{T}$-finite vectors.

\

We define:  
\begin{equation*}
\mathbb{W}_\hbar(T^*_p\mathcal{M}_C):=\widehat{\mathbb{W}}_\hbar(T^*_p\mathcal{M}_C)^{\mathrm{fin}}, 
%\mathbb{W}(T^*_p\mathcal{M}_C):=\mathbb{W}_\hbar(T^*_p\mathcal{M}_C)/(\hbar-1),
\end{equation*}
where $\bullet^{\mathrm{fin}}$ is taken with respect to the torus $\mathsf{A}_Q$.
%$\mathbb{C}\langle V,V^*,\hbar\rangle$ subject the same relations as in (\ref{eq: formal Weyl}). As a vector space, it can be identified with $\mathbb{C}[V,V^*,\hbar]$. 

\begin{remark}
Note that  $\mathbb{W}_\hbar(T^*_p\mathcal{M}_C)$ contains the ``ordinary'' Weyl algebra as a subalgebra. By  ``ordinary'' we mean the one given by  $\mathbb{C}[T^*_p\mathcal{M}_C,\hbar]$ subject the relations (\ref{eq: formal Weyl}). Algebra $\mathbb{W}_\hbar(T^*_p\mathcal{M}_C)$ is {\emph{larger}} than the ordinary Weyl algebra. %That's the main difference between our situation and more standard ``conical'' situation (when the action of $\mathbb{C}^\times_\hbar$ on $\mathcal{M}_C$ is conical).
\end{remark}

\begin{comment}
We define $\widehat{\mathcal{A}}_{1,\widehat{p}}^{\mathrm{fin}}:=\widehat{\mathcal{A}}_{\hbar,\widehat{p}}^{\mathrm{fin}}/(\hbar-1)$.
\begin{corollary}
There exists a $\mathsf{A}_{Q} \times \mathbb{C}^\times_\hbar$-equivariant isomorphism of $\mathbb{C}[\hbar]$ (resp. filtered) algebras:
\begin{equation*}
\mathbb{W}_\hbar(T^*_p\mathcal{M}_C) \simeq \widehat{\mathcal{A}}_{\hbar,\widehat{p}}^{\mathrm{fin}},~\mathbb{W}(T^*_p\mathcal{M}_C) \simeq \widehat{\mathcal{A}}_{1,\widehat{p}}^{\mathrm{fin}}.
\end{equation*}
\end{corollary}
\begin{proof}
Follows from Lemma \ref{lem: completion is W} by taking $\mathsf{A}_Q \times \mathbb{C}^\times_\hbar$-finite vectors.
\end{proof}
\end{comment}

%Let's now define an action of $\mathbb{W}_\hbar(T_p\mathcal{M}_C)$  on  $\mathbb{C}[T_p^-\mathcal{M}_C,\hbar]$, specializing $\hbar=1$, we would obtain 
The action $\theta(\mathbb{C}^\times) \curvearrowright T^*_p\mathcal{M}_C$ induces the decomposition
\begin{equation*}
T_p^*\mathcal{M}_C = T^{*,+}_p\mathcal{M}_C \oplus T^{*,-}_p\mathcal{M}_C
\end{equation*}
into positive and negative parts. Symplectic form on $\mathcal{M}_C^{\mathrm{reg}}$ induces the nondegenerate pairing:
\begin{equation*}
(\,,\,)\colon T^{*,+}_p\mathcal{M}_C \times T^{*,-}_p\mathcal{M}_C \rightarrow \mathbb{C}.
\end{equation*}
For $y \in T^{*,+}_p\mathcal{M}_C$ let $v_y:=(y,-)$ be the corresponding element of $T^{-}_p\mathcal{M}_C$.
Let's now define an action of  $\mathbb{W}_\hbar(T^*_p\mathcal{M}_C)$ on $\mathbb{C}[T_p^-\mathcal{M}_C][[\hbar]]$. We define the action of $x \in T_p^{*,-}\mathcal{M}_C$ and $y \in T_p^{*,+}\mathcal{M}_C$ on $f \in \mathbb{C}[T^-_p\mathcal{M}_C][[\hbar]]=S^{\bullet}(T^{*,-}_p\mathcal{M}_C)[[\hbar]]$ by the following formulas:
\begin{equation}\label{eq: action functions reprellent definition}
x \cdot f = xf ,~y \cdot f = \hbar\partial_{v_y}(f).
\end{equation}

\begin{lemma}
The action (\ref{eq: action functions reprellent definition}) extends to the action of $\mathbb{W}_\hbar(T_p^*\mathcal{M}_C)$ on $\mathbb{C}[T_p^-\mathcal{M}_C][[\hbar]]$.
\end{lemma}
\begin{proof}
It is enough to check that for every integer $i \in \mathbb{Z}$ and $a \in \mathbb{W}_\hbar(T^*_p\mathcal{M}_C)_i$, the action of $a$ is well-defined. 
Let's pick a basis $x_1,\ldots,x_d$ of $T_p^{*,-}\mathcal{M}_C$ consisting of $\theta(\mathbb{C}^\times)$-eigenvectors and let $y_1,\ldots,y_d \in T^{*,+}_p\mathcal{M}_C$ be the dual basis. For $j=1,\ldots,d$ let $k_j$ be the $\theta$-weight of $x_j$. Note that $k_j<0$ for any $j$ as above and the $\theta$-weight of $y_j$ is equal to $-k_j$.

We can write $a$ as a (possibly infinite) sum of elements 
\begin{equation*}
x_1^{l_1}\ldots x_d^{l_d}y_1^{m_1}\ldots y_d^{m_d},~l_j, m_j  \in \mathbb{Z}_{\geqslant 0},~\sum_{j=1}^dl_jk_j-\sum_{j=1}^dm_jk_j=i
\end{equation*}
with some coefficients from $\mathbb{C}[[\hbar]]$. In particular, for any fixed $r \in \mathbb{Z}_{\geqslant 0}$, we have only finitely many possible $l_j, m_j$ as above such that $\sum_{j=1}^dm_j(-k_j) \leqslant r$.
It then follows that for any fixed vector $f \in \mathbb{C}[T_p^-\mathcal{M}_C]$ only finite number of elements in the sum  representing element $a$  would act on $f$ by nonzero operators. So, the action of $a$ on $f$ is well-defined. 
\end{proof}

\begin{lemma}
After the identification $\mathbb{W}_\hbar(T^*_p\mathcal{M}_C) \simeq \widehat{\mathcal{A}}_{\hbar,\widehat{p}}^{\mathrm{fin}}$ module $\mathbb{C}[T^-_p\mathcal{M}_C][[\hbar]]$ becomes isomorphic to the induced module from $\mathsf{C}_\theta(\widehat{\mathcal{A}}_{\hbar,\widehat{p}}^{\mathrm{fin}}) = \mathbb{C}[[\hbar]]$. We will denote this induced module by $M_\hbar$.
\end{lemma}
\begin{proof}
Follows from the definitions. 
\end{proof}

\subsection{Graded traces and their formal versions}

We will now need a notion of the $D$-module of graded traces. This notion was introduced in \cite{KMP} and further studied in \cite{DKK}.
Let us recall the construction. Let $\Sigma \subset \mathsf{a}_Q^*$ be the subset of $\mathsf{A}_Q$-weights that appear in $T^*_p\mathcal{M}_C$. Set 
\begin{equation*}
\Sigma_- := \{\alpha \in \Sigma\,|\, \langle \alpha,\theta\rangle < 0\}.
\end{equation*}
If $V$ is a vector space then  we will denote by $V[[z]]$ the space of formal linear combinations $\sum_{d \in \operatorname{Span}_{\mathbb{Z}_{\geqslant 0}}(\Sigma_-)}z^dv_d$, $v_d \in V$. The space $\operatorname{GrTr}(\mathcal{A})$ is defined as follows:
\begin{equation*}
\operatorname{GrTr}(\mathcal{A}) = \mathcal{A}_0[[z]]/\mathbb{C}[[z]] \cdot \{ab-z^\eta ba\},
\end{equation*}
where $a \in \mathcal{A}_\eta$, $b \in \mathcal{A}_{-\eta}$ s.t. $\eta \in \operatorname{Span}_{\mathbb{Z}_{\geqslant 0}}(\Sigma_-)$.

\

The following general properties of $\operatorname{GrTr}(\mathcal{A})$ are known by \cite{DKK}.

\begin{itemize}
    \item Module $\operatorname{GrTr}(\mathcal{A})$   is finitely generated over $\mathbb{C}[[z]]$.
    \item The natural map $H[[z]] \rightarrow \operatorname{GrTr}(\mathcal{A})$ is surjective.
    \item The $z=0$ specialization of $\operatorname{GrTr}(\mathcal{A})$ is equal to $\mathsf{C}_\theta(\mathcal{A})$.
\end{itemize}

%{\vasya{only using the $\hbar=1$ statement here}}

\begin{lemma}\label{lemma_tr_functions_ind_hbar}
Map $\operatorname{tr}_{M_\hbar}(-)\colon \operatorname{GrTr}(\widehat{\mathcal{A}}_{\hbar,\widehat{p}}^{\mathrm{fin}})  \rightarrow \mathbb{C}[[z]][[\hbar]]$  induces the   map 
$
\operatorname{GrTr}(\mathcal{A}_\hbar) \rightarrow \mathbb{C}[[z]][\hbar]
$
that, in turn, induces the map
\begin{equation*}
\operatorname{GrTr}(\mathcal{A}) \rightarrow \mathbb{C}[[z]]
\end{equation*}
to be denoted $\operatorname{tr}$ or $\operatorname{tr}(-)$.
\end{lemma}
\begin{proof}
Set $\mathcal{A}_\hbar := \operatorname{R}_\hbar(\mathcal{A})$.
  The map $\mathcal{A}_\hbar \rightarrow \widehat{\mathcal{A}}_{\hbar,\widehat{p}}$ induces the $\mathbb{C}^\times_\hbar$-equivariant map:
  \begin{equation*}
  \operatorname{GrTr}(\mathcal{A}_\hbar) \rightarrow \operatorname{GrTr}(\widehat{\mathcal{A}}_{\hbar,\widehat{p}}^{\mathrm{fin}})=\operatorname{GrTr}(\mathbb{W}_\hbar(T^*_p\mathcal{M}_C)) =\mathbb{C}[[z]][[\hbar]].
  \end{equation*}
  Passing to $\mathbb{C}^\times_\hbar$-finite vectors we obtain the $\mathbb{C}[[z]][\hbar]$-equivariant map 
  \begin{equation}\label{maps_between_grtr}
  \operatorname{GrTr}(\mathcal{A}_\hbar) \rightarrow \mathbb{C}[[z]][\hbar].
  \end{equation}
 Note now that the functional $\operatorname{tr}_{M_\hbar}(-) \colon \operatorname{GrTr}(\widehat{\mathcal{A}}_{\hbar,\widehat{p}}^{\mathrm{fin}}) = \mathbb{C}[[z]][[\hbar]] \rightarrow \mathbb{C}[[z]][[\hbar]]$ induces the map $\mathbb{C}[[z]][\hbar] \rightarrow \mathbb{C}[[z]][\hbar]$ (use that $\operatorname{tr}_{M_\hbar}(1)$ lives in $\mathbb{C}[[z]]$) to be denoted by the same symbol. Composing it with (\ref{maps_between_grtr}) we obtain the desired $\mathbb{C}[[z]][\hbar]$-linear map  $\operatorname{GrTr}(\mathcal{A}_\hbar) \rightarrow \mathbb{C}[[z]][\hbar]$ that restricts to the map $\operatorname{tr}\colon \operatorname{GrTr}(\mathcal{A}) \rightarrow \mathbb{C}[[z]]$ sending $1$ to $\operatorname{tr}_{M_\hbar}(1)$.

  %Clearly, the map (\ref{maps_between_grtr}) is surjective so by Nakayama lemma it's enough to check that it becomes  injective after passing to $z=\hbar=0$. This follows from the fact that $\mathbb{C}[\mathcal{M}_C^{\theta}]=\mathbb{C}$ is one-dimensional. 

\end{proof}

 \begin{lemma}\label{lemma_descr_grtr}
 The natural map $\mathbb{C}[[z]] \rightarrow \operatorname{GrTr}(\mathcal{A})$ is an isomorphism. In particular, $\mathsf{C}_\theta(\mathcal{A})=\mathbb{C}$ (as algebras) and $\mathcal{O}_\theta(\mathcal{A})$ contains uinque irreducible object $L$. 
 %and $\operatorname{GrTr}(\mathcal{A}_\hbar) \subset \operatorname{GrTr}(\widehat{\mathcal{A}}_{\hbar,\widehat{p}}^{\mathrm{fin}})$.  
 \end{lemma}
 \begin{proof}
 First of all note that the map $\mathbb{C}[[z]] \rightarrow \operatorname{GrTr}(\mathcal{A})$ is surjective, this follows from Nakayama lemma combined with the fact that  $\operatorname{dim}\mathsf{C}_\theta(\mathcal{A}) \leqslant \operatorname{dim}\mathbb{C}[\mathcal{M}_C^\theta]=1$. To see that it is injective it's enough to note that the composition $\mathbb{C}[[z]] \rightarrow \operatorname{GrTr}(\mathcal{A}) \xrightarrow{\operatorname{tr}} \mathbb{C}[[z]]$ is given by the multiplication by an invertible function hence is an isomorphism. 
 \end{proof}

 %{\vasya{apply Nakayama lemma here, need fin. gen. so surjectivity from $H_\hbar[[z]]$? previous approach via constructing a homomorphism to $\hbar=1$-version of Weyl -- make into a comment?}}

 \subsection{Normalized character of $L$} 

Set $D:=S^\bullet(\mathsf{a})[[z]]$ and define the algebra structure on $D$ by the following formula (see \cite[Section 3.2]{KMP}):
\begin{equation}\label{eq:mult_D}
z^{\eta}x \cdot z^\xi y = z^{\eta+\xi}xy+\langle\xi,x\rangle z^{\eta+\xi} y = z^{\eta+\xi} (x+\langle \xi,x\rangle) y
\end{equation}
for all $x,y \in \mathsf{a}$, $\eta,\xi \in \operatorname{Span}_{\mathbb{Z}_{\geqslant 0}}(\Sigma_-)$. This is a module over a (local) ring $\mathbb{C}[[z]]$.

Let $Z$ be the following vector space. It consists of formal linear combinations $\sum_{\eta \in \mathsf{a}^*}z^\eta a_\eta$ such that:
\begin{itemize}
    \item there exists a finite set $\{\kappa_1,\ldots,\kappa_r\} \subset \mathsf{a}^*$ such that for any $\eta \in \mathsf{a}^*$ such that $a_\eta \neq 0$ we have either $\eta=\kappa_i$ or $\langle \kappa_i-\eta,\theta\rangle \in \mathbb{Z}_{>0}$ for some $i \in \{1,\ldots,r\}$,
    \item for any $i \in \{1,\ldots,r\}$ and $m \in \mathbb{Z}_{>0}$ the set of $\eta \in \mathsf{a}^*$ such that $\langle \kappa_i-\eta,\theta\rangle = m$ is finite.
\end{itemize}

%\begin{itemize}
%\item there exists $N \in \mathbb{Z}$ such that $\langle \eta,\theta\rangle<N$ for any $\eta$ s.t. $a_\eta \neq 0$,
%\item for any fixed $m \in \mathbb{Z}$ the set of $\eta$ such that %$\langle \eta,
%\theta\rangle \geqslant m$, $a_\eta \neq 0$ is finite.
%\end{itemize}

Algebra $D$ acts on $Z$ via the following formulas:
\begin{equation*}
  x \cdot z^\eta = \langle \eta,x\rangle z^\eta,~z^\xi \cdot z^\eta = z^{\xi+\eta}.  
\end{equation*}

%Space $\mathbb{C}[[z]]$ admits an action of $D$ by the formulas:
%\begin{equation}\label{eq:action_D_on_C[[z]]_zeta}
%x \cdot z^\eta = \langle \eta,x\rangle z^\eta,~z^\xi \cdot z^\eta = z^{\xi+\eta}. 
%\end{equation}

%Because of the $\kappa$-shift in the definition of the $D$-module, we will denote $\mathbb{C}[[z]]$ with this $D$-module structure by $\mathbb{C}[[z]]_{\kappa}$. 

%Recall that $M_\hbar$ is the (unique) simple module in $\mathcal{O}_\theta(\mathcal{A})$.
%Torus $\mathsf{a}_Q \subset \mathcal{A}$ acts on $L$, 

Let $\kappa \in \mathsf{a}_Q^*$ be the composition $\mathsf{a}_Q \subset \mathcal{A}_0 \twoheadrightarrow \mathsf{C}_\theta(\mathcal{A})=\mathbb{C}$. Space $\mathbb{C}[[z]]$ admits an action of $D$ by the formulas (note the $\kappa$-shift):
\begin{equation}\label{eq:action_D_on_C[[z]]_zeta}
x \cdot z^\eta = \langle \eta+\kappa,x\rangle z^\eta,~z^\xi \cdot z^\eta = z^{\xi+\eta}. 
\end{equation}
Because of the $\kappa$-shift in the definition of the $D$-module, we will denote $\mathbb{C}[[z]]$ with this $D$-module structure by $\mathbb{C}[[z]]_{\kappa}$. 
\begin{lemma}\label{norm_tr_is_iso_D_mod}
The map $\operatorname{tr}$ induces an isomorphism of $D$-modules $\operatorname{GrTr}(\mathcal{A}) \iso \mathbb{C}[[z]]_\kappa$. 
\end{lemma}
\begin{proof}
It follows from Lemma \ref{lemma_descr_grtr} that $\operatorname{tr}$ is an isomorphism. The fact that $\operatorname{tr}$ is a homomophism of $D$-modules follows from the definitions. 
\end{proof}

\begin{lemma}\label{lem_sol_grtr_one_dim}
The space of $D$-module homomorphisms $\operatorname{GrTr}(\mathcal{A}) \rightarrow Z$ is equal to $\mathbb{C}$.   
\end{lemma}
\begin{proof}
By Lemma \ref{norm_tr_is_iso_D_mod}, we can identify  $\operatorname{GrTr}(\mathcal{A}) \simeq \mathbb{C}[[z]]$ so we are dealing with $D$-equivariant maps $\mathbb{C}[[z]] \rightarrow \mathbb{C}[[z]]$. Any such map is uniquely determined by the image of $1$ (use that it is $\mathbb{C}[[z]]$-linear). Let $f(z) \in \mathbb{C}[[z]]$ be the image of $1$. It must be annihilated by the action of any $x \in \mathsf{a}_Q$ implying that $f(z)$ is a constant as desired.
\end{proof}

\begin{lemma}
The normalized $\mathsf{a}_Q$-character of $L$ is equal to the $\mathsf{A}_Q$-character of $\mathbb{C}[T_p^-\mathcal{M}_C]$.
\end{lemma}
\begin{proof}
Let $\operatorname{tr}_{L}(-)\colon \operatorname{GrTr}(\mathcal{A}) \rightarrow Z$ be the graded trace of $L$. Consider also $\operatorname{tr}(-)$ constructed in Lemma \ref{lemma_tr_functions_ind_hbar}. Both of them define homomorphisms of $D$
-modules $\operatorname{GrTr}(\mathcal{A}) \rightarrow Z$ so
by Lemma \ref{lem_sol_grtr_one_dim}  $\operatorname{tr}_{L}(-)=\operatorname{tr}(-)$.

Now, $\operatorname{tr}(1)=\operatorname{tr}_{M_\hbar}(1)$ is equal to $\operatorname{tr}_{\mathbb{C}[T^-_p\mathcal{M}_C][[\hbar]]}(1)$. 
The claim then follows from the fact that the commutator action of $\mathsf{a}_Q$ on $\mathcal{A}$ integrates to the standard action  $\mathsf{A}_Q \curvearrowright \mathcal{A}$.
\end{proof}

Combining all of the above we obtain the following proposition.
Let $\widetilde{\operatorname{ch}}(L)$ be the $\mathsf{a}_Q$-character of $L$ normalized in such a way that it starts with $1$. 

\begin{proposition}\label{prop: char irreducible}
Assuming Conjecture \ref{conj: dual tangent space} holds and for $\theta>0$ we have:
\begin{equation}\label{eq: char generalized prefundamental}
\widetilde{\operatorname{ch}}(L) = \prod_{\alpha \in \Phi^{-,\mathrm{re}}_{\mu}}\frac{1}{(1-e^\alpha)^{\langle \alpha,\mu\rangle}},
\end{equation}    
where $L$ is the unique irreducible object in the category $\mathcal{O}_\theta(\mathcal{A})$.
\end{proposition}

 \begin{remark}
  By taking into account the  $\mathbb{C}^\times_\hbar$-action on $T_p^{-}\mathcal{M}_C$ we obtain a natural deformation of the normalized $\mathsf{a}_Q$-character of $L$ given by $\prod_{\alpha \in \Phi^{-,\mathrm{re}}_{\mu}}\prod_{s=1}^{\langle \alpha,\mu\rangle}\frac{1}{1-e^{s\hbar}e^\alpha}$. This is compatible  with the additional  ``$v$''-parameter in \cite[Theorem 1.3]{Negut_char}.
 \end{remark}

\begin{remark}
%The representations $L$ should be considered as generalizations of so-called {\emph{prefundamental modules}} over shifted Yangians. 
For $Q$ being a finite type Dynkin quiver, we prove Conjecture \ref{conj: dual tangent space} in Section \ref{eq: section proofs ADE} (by combining the argument similar to the one in \cite[Lemma 4.4]{KamnitzerTingleyWebsterWeeksYacobi} with \cite[Theorem 3.1(1)]{Kr}). In particular, we obtain a character formula (\ref{eq: char generalized prefundamental}) for the irreducible module over the truncated shifted Yangian $Y^\lambda_{w\lambda}$ for arbitrary dominant $\lambda$ and $w \in W$. Another proof of this result will appear in the upcoming paper \cite{min_chamber_mod} by Artem Kalmykov, Joel Kamnitzer, Alexis Leroux-Lapierre, and Th\'eo Pinet. A closely related result
was also obtained by Negut in \cite[Theorem 1.3]{Negut_char} by completely different methods.
\end{remark}

\begin{remark}\label{eq: prefund building blocks}
Let's note that the quantized counterpart of Conjecture \ref{conj: open embedding} should claim that there exists a homomorphism between the corresponding quantized Coulomb branches. Then, taking modules $L$ as above, one would obtain a collection of modules over the quantization of $\mathcal{M}_C$ labeled by nonsingular fixed points of $\widetilde{\mathcal{M}}_{C,\nu}$. Let's decompose $\lambda=\lambda_1+\ldots+\lambda_N$ into the sum of fundamental weights. Then, we obtain the decomposition $\mu=\mu_1+\ldots+\mu_N$ ($\mu_i=w(\lambda_i)$). The module $L$ should be isomorphic to a tensor product of modules $L_i$ corresponding to $\lambda_i,\mu_i$. For $Q$ of finite type the corresponding modules $L_i$ are so-called {\emph{chamber}} modules over $Y_{\mu_i}$ (we follow the terminology of the upcoming paper \cite{min_chamber_mod} by  Kamnitzer, Leroux-Lapierre, Pinet, and  Weekes). For $\mu =\lambda,w_0\lambda$, the corresponding modules are so-called positive and negative prefundamental modules over $Y_{\mu_i}$ (as in \cite[Section 3]{Hern_Zhang}).
\end{remark}

\section{Coulomb branches for finite type Dynkin quivers}\label{subsec: proof for ADE}
The goal of this section is to prove Conjectures  \ref{conj: dual tangent space}, \ref{conj: nonsingular Coulomb point Higgs}, and  \ref{conj: open embedding}  for finite type Dynkin quivers. Conjecture \ref{conj: nonsingular Coulomb point Higgs} follows by combining the results of \cite{MW}, \cite{BFN_slices} and Conjecture \ref{conj: open embedding} follows from the results of \cite{KP}. Our argument for the proof of Conjecture \ref{conj: dual tangent space} is obtained by combining the argument similar to the one in \cite[Lemma 4.4]{KamnitzerTingleyWebsterWeeksYacobi} with \cite[Theorem 3.1(1)]{Kr}. From now on, we assume that $Q$ is of finite type.

\subsection{Realization via slices in affine Grassmannians}
In this section, we recall all of the objects we need to prove conjectures above. Recall that $\mathfrak{g}=\mathfrak{g}_Q$ is the simple Lie algebra corresponding to $Q$. Let $G=G_Q$ be the adjoint Lie group with Lie algebra $\mathfrak{g}$ and choose a Borel subalgebra $B$ with opposite Borel $B_{-}$. Let $\mathsf{A}_Q$ be the maximal torus of $G$ (i.e. the intersection of $B$ and $B_-$) and let $\mathfrak{a}_Q$ be its Lie algebra. Recall that for $i \in Q_0$ we will denote by $\omega_i^\vee, \alpha_i^\vee$ the corresponding fundamental coweight and simple coroot for $\mathfrak{g}$.
Given $\dv,\dw \in \mathbb{N}^{Q_{0}}$, we define $\lambda, \mu$ as in (\ref{eq: def lambda mu}). Recall that $W$ is the Weyl group of $G$ and let $w_0 \in W$ be its longest element.

Following~\cite{BFN_slices}, we define the generalized transversal slice in the affine Grassmannian $\overline{\mathcal W}_\mu^\lambda$. 
    It is the moduli space of the data
    $(\mathcal{P},\sigma,\phi)$, where: 
    \begin{itemize}
        \item 
    $\mathcal{P}$ is a $G$-bundle  on $\mathbb{P}^{1};$
    \item
     $\sigma\colon \mathcal{P}^{\mathrm{triv}}|_{\mathbb{P}^{1} \setminus \{0\}} \iso \mathcal{P}|_{\mathbb{P}^{1} \setminus \{0\}}$ --
    a trivialization, having a pole of degree $\leqslant \lambda$ at the point $0$. This means that the point $(\mathcal{P}, \sigma) \in \operatorname{Gr}_G$ lies in $\overline{\operatorname{Gr}}{}^{\lambda}_G:=\overline{G_{\mathcal{O}} \cdot z^{\lambda}}$;
    \item
    $\phi$ is a $B$-structure on $\mathcal{P}$ (i.e. a $B$-subbundle of  $\mathcal{P}$) of degree $w_{0}(\mu)$, having fiber $B_{-}$ at $\infty$ (with respect to $\sigma$).
     \end{itemize}

The open subvariety $\mathcal{W}^\lambda_\mu \subset \overline{\mathcal{W}}^\lambda_\mu$ consists of triples $(\mathcal{P},\sigma,\phi)$ above such that the degree of a pole of $\sigma$ at zero is {\emph{equal}} to $\lambda$.

It follows from \cite[Section 2(xi)]{BFN_slices}  that $\overline{\mathcal{W}}^\lambda_\mu$ has the following matrix description:
\begin{equation}\label{eq: matrix description of slices}
\overline{\mathcal{W}}^\lambda_\mu \simeq U[[z^{-1}]]_1 \mathsf{A}_{Q}[[z^{-1}]]_1 z^\mu U_-[[z^{-1}]]_1 \cap \overline{(G[z]z^\lambda G[z])}.
\end{equation}
where the notation is as in \cite{BFN_slices}. The open subvariety $\mathcal{W}^\lambda_\mu \subset \overline{\mathcal{W}}^\lambda_\mu$ consists of points $g \in \overline{\mathcal{W}}^\lambda_\mu$ that lie in $G[z]z^\lambda G[z]$. 

%\hunter{Should the subscript $U_{1}$ really be there?}

%\hunter{Notation inconsistency: $T$ vs $\mathsf{A}_{Q}$, and for Lie algebras.}

It follows from \cite[Theorem 3.10]{BFN_slices} for ADE types and \cite[Theorem 4.1]{NW} for BCFG types that we have an isomorphism: 
\begin{equation}\label{eq: coulomb vs slice}
\mathcal{M}_C(\dv,\dw) \simeq 
\overline{\mathcal W}_{\mu}^{\lambda}. 
\end{equation}
%To simplify notations, we will only use $\lambda^*$, $\mu^*$ when we make comparison between slices and Coulomb/Higgs branches. In all other situations, we will just deal with $\lambda$ and $\mu$ instead.

\begin{remark}
For a coweight $\eta \in \mathfrak{a}_Q$ of $\mathfrak{g}$ we will denote by $\eta^*$ the coweight $-w_0(\eta)$ (here $w_0$ is the longest element of the Weyl group $W$ acting naturally on $\mathfrak{a}_Q$). Strictly speaking, in \cite[Theorem 3.10]{BFN_slices}, authors construct an isomorphism $\mathcal{M}_C(\dv,\dw) \simeq 
\overline{\mathcal W}_{\mu^*}^{\lambda^*}$.
Note now that there exists an identification: 
\begin{equation}\label{eq: slice vs slice for star}
\overline{\mathcal{W}}^{\lambda}_\mu \simeq \overline{\mathcal{W}}^{\lambda^*}_{\mu^*}.
\end{equation}
 To see that, recall the matrix description of slices \eqref{eq: matrix description of slices}. The element $w_0 \in W=N_G({\mathsf{A}}_Q)/{\mathsf{A}}_Q$ lifts to some $\dot{w}_0 \in N_G(\mathsf{A}_Q) \subset G$. We have an automorphism of $G$ given by $g \mapsto \dot{w}_0g^{-1}\dot{w}_0^{-1}$. It induces the isomorphism $G((z^{-1})) \iso G((z^{-1}))$, which, in turn, induces the desired identification (\ref{eq: slice vs slice for star}). Combining the identification \eqref{eq: slice vs slice for star} with \cite[Theorem 3.10]{BFN_slices} we obtain the identification \eqref{eq: coulomb vs slice}.
\end{remark}

The torus $\pi_1(G)^{\wedge}$ identifies with the maximal torus $\mathsf{A}_Q \subset G_Q$ acting naturally on $\overline{\mathcal W}_\mu^\lambda$ (changing the trivialization $\sigma$ at $\infty$). The action of $\mathbb{C}^\times_\hbar$ corresponds to the loop rotation action on $\overline{\mathcal W}_\mu^\lambda$ via the automorphisms of the curve $\mathbb{P}^1$.

%The character $\theta$ identifies with $2\rho$. {\vasya{normalization}}

Let ${\mathsf{A}} \subset \prod_{i}GL_{\dw_i}$ be the flavor torus (i.e. the framing torus on the quiver variety side). Pick a cocharacter $\nu\colon \mathbb{C}^\times \rightarrow {\mathsf{A}}$. Let $N \in \mathbb{Z}_{\geqslant 1}$ be the number of eigenvalues of $\nu$.
The choice of $\nu$ corresponds to the decomposition:
\begin{equation*}
\dw_i=\dw_i^{(1)}+\ldots+\dw_{i}^{(N)},~ i \in Q_0.
\end{equation*}
Recall that we can associate to $\nu$ the partial resolution $\widetilde{\mathcal{M}}_{C,\nu}$ of $\mathcal{M}_{C}$. Let us recall the description of $\widetilde{\mathcal{M}}_{C,\nu}$ in the language of affine Grassmannian slices.

Consider the $N$-tuple of dominant coweights $\underline{\lambda}:=(\lambda_1,\ldots,\lambda_N)$ of $G_Q$, where
\begin{equation*}
\lambda_j = \sum_{i \in Q_0}\dw_i^{(j)}\omega_i^\vee.
\end{equation*}
Clearly, $\lambda=\lambda_1+\ldots+\lambda_N$.
Set $\mathsf{T}_Q := {\mathsf{A}}_Q \times \mathbb{C}^\times_\hbar$. It follows from \cite[Section 5]{BFN_line} that we have a $\mathsf{T}_Q$-equivariant isomorphism of varieties:
\begin{equation}\label{eq: convolution is BFN res}
\widetilde{\mathcal{W}}^{\underline{\lambda}}_{\mu} \simeq \widetilde{\mathcal{M}}_{C,\nu},
\end{equation}
where $\widetilde{\mathcal{W}}^{\underline{\lambda}}_{\mu}$ is the convolution diagram over $\overline{\mathcal{W}}^{\lambda}_{\mu}$ defined as follows.

The variety $\widetilde{\mathcal{W}}^{\underline{\lambda}}_{\mu}$ is the moduli space of the data 
\begin{equation*}
(\mathcal{P}^{\mathrm{triv}}=\mathcal{P}_0,\mathcal{P}_1,\ldots,\mathcal{P}_N,\sigma_1,\ldots,\sigma_N,\phi),
\end{equation*} 
where
\begin{itemize}	
	\item $\mathcal{P}_i$ is a $G$-bundle on $\mathbb{P}^1$;
	
	\item  $\sigma_i\colon \mathcal{P}_{i-1}|_{\mathbb{P}^1 \setminus \{0\}} \iso \mathcal{P}_i|_{\mathbb{P}^1 \setminus \{0\}}$ is an isomorphism having a pole of degree $\leqslant \lambda_i$ at zero;
	
	\item $\phi$ is a $B$-structure on $\mathcal{P}_N$ of degree $w_0\mu$, having fiber $B_-$ at $\infty$ with respect to $\sigma_N \circ \sigma_{N-1} \circ \ldots \circ \sigma_1$.
\end{itemize}

We have a natural (proper and birational) morphism $\widetilde{\mathcal{W}}^{\underline{\lambda}}_{\mu} \rightarrow \overline{\mathcal{W}}^{\lambda}_{\mu}$ given by 
\begin{equation*}
(\mathcal{P}_0,\mathcal{P}_1,\ldots,\mathcal{P}_N,\sigma_1,\ldots,\sigma_N,\phi) \mapsto (\mathcal{P}_N, \sigma_N \circ \sigma_{N-1} \circ \ldots \circ \sigma_1,\phi).
\end{equation*}

Choose generic $\theta$. We have the following lemma.
\begin{lemma}
The set of $\theta$-fixed points $(\overline{\mathcal{W}}^\lambda_\mu)^\theta$ consists of one point if $\mu$ is a weight of $V(\lambda)$, the irreducible representation of the Langlands dual of $\mathfrak{g}$ of highest weight $\lambda$, and is empty otherwise. We will denote this fixed point by $z^\mu$.    
\end{lemma}
\begin{proof}
See \cite[Lemma 2.8]{Kr}.
\end{proof}

More generally, it follows from \cite[Proposition 5.9]{KP} that the $\theta$-fixed points on $\widetilde{\mathcal{W}}^{\underline{\lambda}}_\mu$ are isolated and are in bijection with $N$-tuples of coweights $\underline{\mu}=(\mu_1,\ldots,\mu_N)$ such that  $\mu=\mu_1+\ldots+\mu_N$ and $\mu_i$ is a weight of $V(\lambda_i)$. We will denote the point corresponding to $\underline{\mu}$ by $z^{\underline{\mu}}$. In \cite[2(vi)]{BFN_slices}, the authors defined the so-called {\emph{multiplication morphism}}:
$
{\bf{m}}^{\underline{\lambda}}_{\underline{\mu}}\colon \overline{\mathcal{W}}^{\lambda_1}_{\mu_1} \times \ldots \times \overline{\mathcal{W}}^{\lambda_N}_{\mu_N} \rightarrow \overline{\mathcal{W}}^{\lambda}_\mu 
$
that lifts to a morphism:
\begin{equation}\label{eq:mult morphism slices}
\tilde{{\bf{m}}}^{\underline{\lambda}}_{\underline{\mu}}\colon \overline{\mathcal{W}}^{\lambda_1}_{\mu_1} \times \ldots \times \overline{\mathcal{W}}^{\lambda_N}_{\mu_N} \rightarrow \widetilde{\mathcal{W}}^{\underline{\lambda}}_\mu. 
\end{equation}
The morphism $\tilde{\bf{m}}^{\underline{\lambda}}_{\underline{\mu}}$ is known to be an {\emph{open embedding}} (see \cite[Proposition 5.7]{KP}). It is also known to be ${\mathsf{A}}_Q$-equivariant, and moreover, it becomes $\mathbb{C}^\times_{\hbar}$-equivariant after appropriately twisting the $\mathbb{C}^\times_{\hbar}$-action on the source of (\ref{eq:mult morphism slices}) (see \cite[Section 5.4]{KP} for the details). It sends the unique  $\mathsf{A}_Q$-fixed point of the domain to the point $z^{\underline{\mu}} \in \widetilde{\mathcal{W}}^{\underline{\lambda}}_\mu$.

We also have a natural morphism $p\colon \overline{\mathcal{W}}^\lambda_\mu \rightarrow \overline{\operatorname{Gr}}^\lambda $ given by forgetting the $B$-structure $\phi$.

%\hunter{We already used some affine Grassmannian notation earlier. I think it's fine to assume the reader knows it}

\begin{comment}
The final object we will need is the affine Grassmannian $\operatorname{Gr}_{G_Q}$ for the group $G_Q$. Recall that
\begin{equation*}
\operatorname{Gr}_{G_Q} = G_{\mathcal{K}}/G_{\mathcal{O}},~\mathcal{K}=\mathbb{C}((z)), \mathcal{O}=\mathbb{C}[[z]].
\end{equation*} It can be also described as the moduli space of pairs $(\mathcal{P},\sigma)$ as in the definition of $\overline{\mathcal{W}}^{\lambda}_\mu$ but with no restrictions on the pole of $\sigma$ at $0$. Requiring the pole of $\sigma$ at $0$ to be $\leqslant \lambda$, we obtain a closed subvariety of $\operatorname{Gr}_{G_Q}$ to be denoted $\overline{\operatorname{Gr}}^\lambda_{G_Q}$ or simply $\overline{\operatorname{Gr}}^\lambda$. This variety is nothing else but the closure of $\operatorname{Gr}^\lambda := G_{\mathcal{O}} \cdot z^\lambda$ 
inside $\operatorname{Gr}_{G_Q}$.
\end{comment}

\subsection{Proofs for finite type Dynkin quivers}\label{eq: section proofs ADE}
% It follows from \cite{BFN_slices} that the Coulomb branch $\mathcal{M}_C$ is isomorphic to the generalized slice in affine Grassmannian $\overline{\mathcal{W}}^\lambda_\mu$. 

From now on, we assume that $\overline{\mathcal{W}}^\lambda_\mu$ contains a $\theta$-fixed point, i.e., that $\mu$ is a weight of $V(\lambda)$. Note that this precisely corresponds to $\widetilde{\mathcal{M}}_H:=\widetilde{\mathcal{M}}_H({\mathsf{v}},{\mathsf{w}})$ being nonempty (compare with the proof of Lemma \ref{lem: nakajima quiver point ADE} below).  

\begin{lemma}\label{lem: smooth fixed on slices}
The unique  $\theta$-fixed point $z^\mu \in \overline{\mathcal{W}}^\lambda_\mu$ is nonsingular iff $\mu \in W\lambda$.    
\end{lemma}
\begin{proof}
It follows from \cite[Theorem 1.2]{MW} combined with \cite[Remark 3.19]{BFN_slices} that the regular locus of  $\overline{\mathcal{W}}^\lambda_\mu$ is equal to $\mathcal{W}^\lambda_\mu$. It remains to note that $z^\mu \in \mathcal{W}^\lambda_\mu$ iff $\mu \in W\lambda$.   
\end{proof}

%The following lemma is standard and follows from the fact that $\operatorname{dim}\widetilde{\mathcal{M}}_H = (\lambda,\lambda)-(\mu,\mu)$.
\begin{lemma}\label{lem: nakajima quiver point ADE}
The variety $\widetilde{\mathcal{M}}_H$ is a point iff $\mu \in W\lambda$. 
\end{lemma}

\begin{proof} Assume $\mu \in W\lambda$. For $\theta>0$, $\widetilde{\mathcal{M}}_{H}$ corresponds to the $\mu$ weight space of the irreducible representation with highest weight $\lambda$. Hence it is nonempty. Since $\dim \widetilde{\mathcal{M}}_H = (\lambda,\lambda)-(\mu,\mu)=0$, $\widetilde{\mathcal{M}}_{H}$ is a point. It is known, see for example Section 2.1 of \cite{etingofconjecture}, that the affinization of $\widetilde{\mathcal{M}}_{H}$ is independent of generic $\theta$. Hence $\widetilde{\mathcal{M}}_{H}$ is a point for any $\theta$. 

Now assume that $\widetilde{\mathcal{M}}_{H}$ is a point. Then $(\lambda,\lambda)=(\mu,\mu)$. Furthermore, $\mu$ must be a weight of the irreducible representation $V(\lambda)$ 

It follows that $\mu^+ \leqslant \lambda$, where $\mu^+$ is the dominant representative in $W\mu$. Our goal is to check that $\mu^+=\lambda$. Indeed note that 
\begin{equation*}
0=(\lambda,\lambda)-(\mu^+,\mu^+) = (\lambda,\lambda-\mu^+)+(\mu^+,\lambda-\mu^+)=(\lambda-\mu^+,\lambda-\mu^+)+2(\mu^+,\lambda-\mu^+)
\end{equation*}
and $(\lambda-\mu^+,\lambda-\mu^+) \geqslant 0$, $(\mu^+,\lambda-\mu^+) \geqslant 0$. So the only possibility for this sum to be zero is if $(\lambda-\mu^+,\lambda-\mu^+)=0$, hence $\lambda=\mu^+$.
\end{proof}

\begin{remark}\label{rem: quiver point mu extremal}
One can show that for an {\emph{arbitrary}} quiver $Q$ without loops the corresponding quiver variety $\widetilde{\mathcal{M}}_H$ is a point iff $\mu \in W\lambda$. The argument is standard. Let us provide a sketch. The representation $V(\lambda)$ is integrable. Now, to every pair $\eta$, $\alpha_i$ of a weight $\eta$ of $V(\lambda)$ and a simple root $\alpha_i$ we consider the set of weights $S_{\eta,\alpha_i}$ of weights of $V(\lambda)$ of the form $\eta+k\alpha_i$ for $k \in \mathbb{Z}$. It is a standard fact (see \cite[Proposition 3.6(b)]{Kac}) that $\exists p,q, \in \mathbb{Z}_{\geqslant 0}$ such that $S_{\eta,\alpha_i}=\{\eta + k \alpha_{i} \, \mid \, -p \leqslant k \leqslant q\}$ and $s_{\alpha_i}(\eta-p\alpha_i)=\eta+q\alpha_i$. Note now that $(\eta+k\alpha_i, \eta+k\alpha_i)=2k^2+2k(\eta,\alpha_i)+(\eta,\eta)$ is a {\emph{strictly convex}} function of $k$. So it must attain its maxima exactly at $-p$ and at $q$. Using this observation, one can prove that if $\mu$ is a weight of $V(\lambda)$, then $(\mu,\mu) \leqslant (\lambda,\lambda)$ and the equality holds iff $\mu \in W\lambda$ (argue by the induction on $\langle \rho,\mu\rangle$, where $\rho$ is the sum of fundamental weights for $\mathfrak{g}_Q$).
%idea here: if $\mu$ is a weight of $V(\lambda)$ then we can find a sequence of 
\end{remark}

%\vasya{is this true?}
%\begin{lemma}
%We have $H^*(\widetilde{\mathcal{M}}_H)=\mathbb{C}$ iff   $\widetilde{\mathcal{M}}_H$ is a point.   
%\end{lemma}
%\begin{proof}
%Let $\mu^+$ be the dominant representative of $\mu$ and let ${\bf{v}}_i'$ be the corresponding dimension vector (i.e., $\mu^+=\lambda-\sum_{i}v_i' \alpha_i$). It follows from \cite{} that $\operatorname{dim}H^*(\widetilde{\mathcal{M}}_H({\bf{v}},{\bf{w}}))=\operatorname{dim}H^*(\widetilde{\mathcal{M}}_H({\bf{v}}',{\bf{w}}))$... 
%\end{proof}

\begin{proposition}\label{prop: two first conj hold}
Conjectures \ref{conj: nonsingular Coulomb point Higgs} and \ref{conj: open embedding} hold for finite type Dynkin quivers.
\end{proposition}
\begin{proof}
 Conjecture \ref{conj: nonsingular Coulomb point Higgs} follows from Lemmas \ref{lem: smooth fixed on slices}, and \ref{lem: nakajima quiver point ADE}.
To prove Conjecture \ref{conj: open embedding}, recall that \eqref{eq: convolution is BFN res} identifies the partial resolution $\widetilde{\mathcal{M}}_{C,\nu}$ with the convolution diagram $\widetilde{\mathcal{W}}^{\underline{\lambda}}_\mu$. Now, Conjecture \ref{conj: open embedding} follows from \cite[Proposition 5.7]{KP} which states that the multiplication morphism $\tilde{\bf{m}}^{\underline{\lambda}}_{\underline{\mu}}$ is a $\mathsf{A}_Q$-equivariant open embedding.
\end{proof}

Let us finally deal with Conjecture \ref{conj: dual tangent space}.

\begin{proposition}\label{prop: tangent for ade}
Conjecture \ref{conj: dual tangent space} holds for finite type Dynkin quivers.    
\end{proposition}
\begin{proof}
The Coulomb branches we must consider are $\mathcal{M}_{C}=\overline{\mathcal{W}}^{\lambda}_\mu$ for $\mu \in W \lambda$. Recall the point $z^\mu$ lies in  $\mathcal{W}^\lambda_\mu \subset \overline{\mathcal{W}}^{\lambda}_\mu$. Consider the repellent $R^\lambda_\mu$ (see \cite[Definition 1.8.3]{DG}) with respect to the cocharacter $2\rho^\vee\colon \mathbb{C}^\times \rightarrow \mathsf{A}_Q$, given by the sum of the positive coroots, inside the smooth locus $\mathcal{W}^\lambda_\mu$. At the level of $\mathbb{C}$-points we have:
\begin{equation*}
R^{\lambda}_\mu = \{x \in \mathcal{W}^\lambda_\mu\,|\, \underset{t \rightarrow 0}{\operatorname{lim}}\,2\rho^\vee(t) \cdot x = z^\mu\}.
\end{equation*}
Note that $R^\lambda_\mu$ is nonsingular (see \cite[Proposition 1.7.6]{DG}, or, alternatively, use Luna's slice theorem) and \begin{comment}
We claim that $R^\lambda_\mu$ is nonsingular. To check this, it is enough to show that 
\begin{equation}\label{eq: dim tangent R estimate}
\operatorname{dim}T_{z^\mu}R^\lambda_\mu \leqslant \langle \rho,\lambda-\mu\rangle
\end{equation}
(use the contracting action of $\mathbb{C}^\times$ via $2\rho$). Note that $T_{z^\mu}R^\lambda_\mu$ embeds naturally inside $(T_{z^\mu}\mathcal{W}^\lambda_\mu)^-$, where $-$ corresponds to taking the direct sum in $T_{z^\mu}\mathcal{W}^\lambda_\mu$ of negative $2\rho(\mathbb{C}^\times)$-weight subspaces. Note now that $\mathcal{W}^\lambda_\mu$ is symplectic and the form is ${\mathsf{A}}_Q$-invariant. It follows that the subspace $(T_{z^\mu}\mathcal{W}^\lambda_\mu)^- \subset T_{z^\mu}\mathcal{W}^\lambda_\mu$ must have dimension equal to $\frac{\operatorname{dim}\mathcal{W}^{\lambda}_{\mu}}{2}=\langle \rho,\lambda-\mu\rangle$. We conclude that:
\begin{equation*}
\operatorname{dim}T_{z^\mu}R^\lambda_\mu \leqslant (T_{z^\mu}\mathcal{W}^\lambda_\mu)^- = \langle \rho,\lambda-\mu\rangle
\end{equation*}
so the inequality (\ref{eq: dim tangent R estimate}) holds.  

It follows that $R^\lambda_\mu$ is nonsingular, hence, by (very simple version of) the Bialynicki-Birula theorem (see \cite{BB}), $R^\lambda_\mu$ must be isomorphic to the affine space $\mathbb{A}^{\langle\rho,\lambda-\mu \rangle}$. It also follows that 
\end{comment}
$T_{z^\mu}R^\lambda_\mu=(T_{z^\mu}\overline{\mathcal{W}}^\lambda_\mu)^{-}$, where $-$ corresponds to taking the direct sum in $T_{z^\mu}\mathcal{W}^\lambda_\mu$ of negative $2\rho^\vee(\mathbb{C}^\times)$-weight subspaces.  Using the symplectic form on $\mathcal{W}^\lambda_\mu$ we obtain the  isomorphism of representations of ${\mathsf{T}}_Q$:
\begin{equation}\label{eq: tangent via repellent}
T_{z^\mu}\overline{\mathcal{W}}^\lambda_\mu \simeq (T_{z^\mu}R^\lambda_\mu) \oplus \hbar (T_{z^\mu}R^\lambda_\mu)^\vee
\end{equation}
It remains to describe the ${\mathsf{T}}_Q$-character of $T_{z^\mu}R^\lambda_\mu$. It follows from \cite[Theorem 3.1(1)]{Kr} that $R^\lambda_\mu$ is isomorphic to the repellent to $z^\mu \in \operatorname{Gr}^{\lambda}$. 
%in particular, has dimension $\langle \rho,\lambda-\mu\rangle$ (see \cite[Theorem 3.2]{MV}). 
%Recall that we have a ${\mathsf{T}}_Q$-equivariant isomorphism between $R^\lambda_\mu$ and the repellent to $z^\mu \in \operatorname{Gr}^\lambda$. 
So, $T_{z^\mu}R^\lambda_\mu$ is isomorphic to $(T_{z^\mu}\operatorname{Gr}^\lambda)^{-}$. We have 
\begin{equation}\label{eq: tang to gr lambda}
T_{z^\mu}\operatorname{Gr}^\lambda = T_{z^\mu}(G_{\mathcal{O}} \cdot z^\mu) = \\
=T_{1}(G_{\mathcal{O}}/(G_{\mathcal{O}} \cap z^{\mu}G_{\mathcal{O}}z^{-\mu})) \simeq \bigoplus_{\alpha \in \Phi,\, k = 0,1,\ldots,\langle\alpha,\mu\rangle-1}z^k\mathfrak{g}_{\alpha}.
\end{equation}
The ${\mathsf{A}}_Q$-weight of $z^k\mathfrak{g}_{\alpha}$ is $\alpha$ and the $\mathbb{C}^\times$-weight of $z^k\mathfrak{g}_{\alpha}$ is $-k$. 
Note now that passing from $T_{z^\mu}\operatorname{Gr}^\lambda$ to $(T_{z^\mu}\operatorname{Gr}^\lambda)^{-}$ corresponds to restricting to $\alpha \in \Phi^{-}_\mu$ in the sum (\ref{eq: tang to gr lambda}). We conclude that
\begin{equation*}
T_{z^\mu}R^\lambda_\mu = (T_{z^\mu}\operatorname{Gr}^\lambda)^{-} = \sum_{\alpha \in \Phi^-_\mu}\sum_{i=1}^{\langle \alpha,\mu\rangle}\hbar^{1-i}e^{\alpha}
\end{equation*}
and the claim follows from (\ref{eq: tangent via repellent}). 
%It must be isomorphic to the affine space $\mathbb{A}^{\langle \rho,\lambda-\mu\rangle}$ (this is a variation of , will add details). It means that the negative part of the $T$-character of $T_{z^\mu}\mathcal{W}^\lambda_\mu$ is equal to the product $\frac{1}{\prod_{\alpha^\vee \in \Delta^\vee_{\mu,-}}(1-e^{\alpha})^{\langle \mu,\alpha^\vee\rangle}}$ (the desired negative part of the $T$-character is equal to the $T$-character of the repellent that can already be computed via the affine Grassmannian). 
\end{proof}

%\vasya{I think $\theta=-2\rho^\vee$, not so important but will check}

\begin{remark}
Note that for $\lambda$ being minuscule and $\mu \in W\lambda$ Proposition \ref{prop: tangent for ade} already follows from \cite[Proposition 4.19]{KP}, so the main new point of Proposition \ref{prop: tangent for ade} is that the same formula holds for arbitrary dominant $\lambda$ and $\mu \in W\lambda$. 
%It is worth noting that in \cite[Section 4]{KP} a much stronger statement is proven for $\lambda$ minuscule and $\mu \in W\lambda$. Namely, it is proven that $\overline{\mathcal{W}}^{\lambda}_{\mu}=\mathcal{W}^{\lambda}_{\mu} \simeq T^*\mathbb{A}^n$ as symplectic varieties and the action of ${\mathsf{T}}_Q$ is described explicitly on the right hand side of the isomorphism. The same argument works for arbitrary $\lambda$  and $\mu \in W\lambda$ being almost dominant (i.e., $\langle \alpha,\mu\rangle \geqslant -1$ for any positive root $\alpha$).
\end{remark}

Combining Propositions \ref{prop: two first conj hold} and \ref{prop: tangent for ade} we get the formula for the tangent space to an arbitrary nonsingular $\mathsf{T}_Q$-fixed point $z^{\underline{\mu}}$ of a partial resolution $\widetilde{\mathcal{W}}^{\underline{\lambda}}_\mu$ (recall that these points are in bijection with all possible decompositions $\mu=\mu_1+\ldots+\mu_N$ such that $\mu_i \in W\lambda_i$). This is a generalization of \cite[Equation (5.8)]{KP}: 

\begin{equation}\label{eq: tangent resolved slice}
T_{z^{\underline{\mu}}} \widetilde{\mathcal{W}}^{\underline{\lambda}}_\mu = \sum_{k=1}^{N} \sum_{\alpha \in \Phi^-_{\mu_k}}\sum_{i=1}^{\langle \alpha,\mu_k\rangle} \left(\hbar^{1-i-\langle\alpha,\mu_1+\ldots+\mu_{k-1}\rangle} e^{\alpha} + \hbar^{i+\langle\alpha,\mu_1+\ldots+\mu_{k-1}\rangle}e^{-\alpha}\right).
\end{equation}

Note that for $\mu$ being {\emph{dominant}} and without $\mathbb{C}^\times_\hbar$-equivariance, the equality (\ref{eq: tangent resolved slice}) follows from \cite[Lemma 4.4]{KamnitzerTingleyWebsterWeeksYacobi}. We are grateful to Joel Kamnitzer for pointing this out to us.

\subsection{About the restriction to real roots}

Let us finish this section with a few remarks. When $Q$ is of finite type, it is known that for $\mu$ being dominant $\overline{\mathcal{W}}^{\lambda}_{\mu}$ is indeed a slice to the $G_{\mathcal{O}}$-orbit $\operatorname{Gr}^{\mu}$ of $z^\mu$ inside the closure of the $G_{\mathcal{O}}$-orbit of $z^\lambda$. This is not the case for arbitrary $\mu$ (that is the reason why $\overline{\mathcal{W}}^{\lambda}_{\mu}$ are called {\emph{generalized}} slices), although by \cite[Theorem 3.1(1)]{Kr} it is always true that the repellent to $z^\mu \in \overline{\mathcal{W}}^{\lambda}_\mu$ under the $\mathbb{C}^\times$-action coming from $2\rho^\vee$ coincides with the repellent to $z^{\mu} \in \overline{\operatorname{Gr}}^{\lambda}$.

%Assuming that $\mu \in W\lambda$ this tells us that the $\theta$-repellent to $z^\mu \in \mathcal{W}^\lambda_\mu$ is equal to the $\theta$-repellent to $z^\mu \in \operatorname{Gr}^{\lambda}$.

Now assume that we are in the setting of Conjecture \ref{conj: dual tangent space} (i.e., variety $\widetilde{\mathcal{M}}_H$ is a point), and we now allow $Q$ to be an arbitrary quiver without edge loops. This is equivalent to $\mu \in W\lambda$ (see Remark \ref{rem: quiver point mu extremal} above). 
%Let $G_Q$ be the Kac-Moody group corresponding to $Q$ that is defined in \cite{KP2}, \cite{KP1},  see also \cite{Ku}, where it is denoted by  $\mathcal{G}^{\mathrm{min}}${\footnote{\vasya{expand check}}}.
 Assuming that the $2\rho^\vee$-fixed point $p \in \mathcal{M}_C$ is nonsingular and denoting by $R^\lambda_\mu \subset \mathcal{M}_C$ the $2\rho^\vee$-repellent to $p$ we see that Conjecture \ref{conj: dual tangent space} is equivalent to the isomorphism  of ${\mathsf{T}}_{Q}$-representations:
\begin{equation}\label{eq: tangent rep real}
T_{p}R^{\lambda}_\mu = \bigoplus_{\alpha \in \Phi^{-,\mathrm{re}}}\mathfrak{g}_\alpha[[z]]/\Big(\bigoplus_{\alpha \in \Phi^{-,\mathrm{re}}}\mathfrak{g}_\alpha[[z]] \cap z^\mu \bigoplus_{\alpha \in \Phi^{-,\mathrm{re}}}\mathfrak{g}_\alpha[[z]] z^{-\mu}\Big),
\end{equation}
where $\Phi^{-,\text{re}}$ are all {\emph{real}} negative roots for $\mathfrak{g}=\mathfrak{g}_Q$. We do not know if the right hand side of \eqref{eq: tangent rep real} has a ``geometric'' meaning similar to the one in finite dimensional situation (when it is the tangent space to the repellent to $z^\mu \in \operatorname{Gr}_{G}^\lambda$).

\begin{remark}
Recently, the paper \cite{BV} introduced  schemes $\operatorname{Gr}^\lambda \cap T_\mu$ for arbitrary symmetrizable Kac-Moody group (see \cite[Section 4.3]{BV}).
It's natural to expect that $R^\lambda_\mu$ is  isomorphic to $\operatorname{Gr}^\lambda \cap T_\mu$ (for $Q$ of finite type this is known by \cite[Theorem 3.1(1)]{Kr}). One approach would be to extend the isomorphism $\mathcal{M}_C^{\mathrm{reg}} \simeq \mathcal{W}^\lambda_\mu$ of \cite[Theorem 3.10]{BFN_slices} to this more general  setting, where $\mathcal{W}^\lambda_\mu$ is constructed using \cite[Section 4]{BV}. Once this identification is established, the same argument as in \cite[Section 4.10]{Kr} should yield the isomorphism $R^\lambda_\mu \simeq \operatorname{Gr}^\lambda \cap T_\mu$. Moreover, the isomorphism $\mathcal{M}_C^{\mathrm{reg}} \simeq \mathcal{W}^\lambda_\mu$ would imply Conjecture \ref{conj: dual tangent space} in this greater generality (generalizing the finite type argument above) and hence verify the hypotheses used in Proposition \ref{prop: char irreducible}. At present these ideas are completely speculative.
\end{remark}

The following lemma sheds some light on the importance of real roots, confirming that the dimensions of left and right hand sides of \eqref{eq: tangent rep real} are indeed equal. 
Recall that $\rho$ is the sum of fundamental weights for $\mathfrak{g}=\mathfrak{g}_Q$.
\begin{lemma}\label{lem: conj 1.4 at the level of dimensions}
Let $\lambda$ be a dominant coweight of $\mathfrak{g}$ and let $\mu \in W \lambda$. Then
\[
\sum_{\alpha \in \Phi^{-,\mathrm{re}}_{\mu}}\langle\alpha,\mu\rangle=\langle \lambda-\mu,\rho  \rangle.
\]
\end{lemma}
\begin{proof}
Recall that $\mu \in W\lambda$ and pick $w \in W$ of minimal length such that $w(\lambda)=\mu$. It is a standard fact that  $w$
%(in other words, if $v \in W$ is any such that $v\lambda=\mu$, then $w$ is the shortest element in $vW_\lambda$, here $W_\lambda$ is the stabilizer of $\lambda$)
has the following property: if $\beta$ is any positive root such that $\langle\beta,\lambda\rangle=0$, then $w\beta$ is also positive.
\begin{comment}
\hunter{I think we do not have to assume that $\beta$ is real. But the only roots that can move from negative to positive under $W$ are real roots (I think).} \vasya{yes, I agree, will finish writing this}

To prove that $w$ has this property, let $\beta$ be a negative root such that $(\beta,\lambda)=0$ and assume that $w\beta$ is positive. We can decompose $\beta$ into the sum of simple roots as $\beta=\sum_{i \in S} c_i \alpha_i$ where all $c_i <0$ for all $i \in S$. Since $\lambda$ is dominant, $(\alpha_i,\lambda)=0$ for all $i \in S$.

Moreover, since $w \beta$ is positive, $\exists i \in S$ such that $w \alpha_{i}$ is positive. So, we have reduced to the case when $\beta = -\alpha_i$. Then the length of $ws_{\alpha_i}$ must be less than the length of $w$ (use that the length is equal to the number of negative real roots that become positive after applying the  element). That's a contradiction because $ws_{\alpha_i}(\lambda)=\mu$.
\end{comment}
%--------height of $\lambda-\mu$ is equal to 
%$(\rho,\lambda-\mu)$ (use that $(\rho,\alpha_i)=1$ for any simple root $\alpha_i$)

Now, we can write: 
\begin{equation*}
\langle \rho,\lambda-\mu\rangle=\langle \rho,w^{-1}\mu-\mu\rangle=\langle w(\rho),\mu\rangle-\langle\rho,\mu\rangle=\langle w(\rho)-\rho,\mu\rangle.
\end{equation*}

Note that $w(\rho)-\rho$ is equal to the sum of all {\emph{real}} negative roots $\alpha$ such that $w^{-1}(\alpha)$ is positive. Also note that $\langle\alpha,\mu\rangle=\langle w^{-1} \alpha,\lambda\rangle$. So, it remains to check that for a negative {\emph{real}} root $\alpha$, $w^{-1}(\alpha)$ is positive iff $\langle w^{-1}\alpha,\lambda\rangle>0$.

Assume first that $\langle w^{-1}\alpha,\lambda\rangle>0$. The coweight $\lambda$ is dominant, so $w^{-1}(\alpha)$ must be positive.

Now assume $\beta:=w^{-1}(\alpha)$ is positive. Since $\lambda$ is dominant, $\langle\beta,\lambda\rangle \geqslant  0$. If $\langle\beta,\lambda\rangle=0$, then $\beta$ is a positive root such that $w(\beta)=\alpha$ is negative. This contradicts the property of $w$.
\end{proof}

\begin{remark}
From the proof of Lemma \ref{lem: conj 1.4 at the level of dimensions}, we see that the key spot that accounts for the restriction to the real roots is the equality: 
\begin{equation*}
w(\rho)-\rho = \sum_{\substack{\alpha \in \Phi^{-,\mathrm{re}} \\ w^{-1}(\alpha) \in \Phi^+}}\alpha.
\end{equation*}
\end{remark}

\section{Slant sums of Coulomb branches}\label{subsec:slant via Coulomb}

%\hunter{I changed notation here to match the earlier sections. It's unfortunately more work to type, but at least it matches earlier stuff.} \vasya{Thank you, Hunter}

In this section, we begin the study of the slant sum construction from the Coulomb branch perspective. We focus on the case corresponding to Proctor’s original construction, namely $\dw_{\star_2}^{(2)}=1$ and $\dw_i^{(2)}=0$ for all other $i$. It would be very interesting (and much more useful for applications) to study the case of general $\mathsf{w}$ (see Section \ref{sec:general framing} above for some informal discussion).

%\vasya{will change notations slightly $\mathcal{A}_\hbar = H_*^{G_\mathcal{O} \rtimes \mathbb{C}^\times}(\mathcal{R})$ and $\mathcal{A}$ is $\mathcal{A}_{\hbar=1}$}

\subsection{Integrable systems for Coulomb branches}
Coulomb branches come equipped with a certain additional structure called an {\emph{integrable system}}. Recall that 
\begin{equation*}
\mathcal{M}_C = \operatorname{Spec}H_*^{G_{\mathcal{O}}}(\mathcal{R}),~\mathcal{A}_\hbar=H_*^{G_\mathcal{O} \rtimes \mathbb{C}^\times}(\mathcal{R}).
\end{equation*}
We have a natural (left) action of the algebra $H^*_{G_{\mathcal{O}}}(\operatorname{pt})=H^*_{G}(\operatorname{pt})$ on $\mathbb{C}[\mathcal{M}_C]=H_*^{G_{\mathcal{O}}}(\mathcal{R})$ as well as the action of $H^*_{G_{\mathcal{O}} \rtimes \mathbb{C}^\times}(\operatorname{pt})=H^*_{G \times \mathbb{C}^\times}(\operatorname{pt})$ on $\mathcal{A}_\hbar$. Acting on identity, we obtain homomorphisms:
\begin{equation*}
H^*_{G}(\operatorname{pt}) \rightarrow \mathbb{C}[\mathcal{M}_C],~H^*_{G \times \mathbb{C}^\times}(\operatorname{pt}) \rightarrow \mathcal{A}_\hbar
\end{equation*}
defining integrable systems on $\mathcal{M}_C$ and $\mathcal{A}_\hbar$. In this way, we obtain (Poisson) commutative subalgebras of $\mathbb{C}[\mathcal{M}_C]$ and $\mathcal{A}$. The subalgebra $H:=H^*_{G \times \mathbb{C}^\times}(\operatorname{pt})|_{\hbar=1} \subset \mathcal{A}$ is called the {\emph{Cartan subalgebra}} in \cite[]{BFNII}. It is also sometimes called {\emph{Gelfand--Tsetlin subalgebra}} (see, for example, \cite{webster_GZ}). We can naturally identify it with the polynomial ring $\mathbb{C}[c_{i,k}\,|\, i \in Q_0,\,k=1,\ldots,v_i]$, where $c_{i,k}= c_{i}^{G \times \mathbb{C}^\times}(V_k)$ ($V_k$ is considered as a $G \times \mathbb{C}^\times$-equivariant bundle on $\operatorname{pt}$).

One reason why the algebra $H$ is important is because it allows one to define the notion of a Gelfand--Tsetlin character of a module over $\mathcal{A}$. Namely, whenever $M$ is a finitely generated $\mathcal{A}$-module such that the action of $H$ on it decomposes $M$ into the direct sum of finite dimensional generalized eigenspaces (such modules are called Gelfand--Tsetlin-modules in \cite{webster_GZ}), we can define its Gelfand--Tsetlin character as follows. 

For any collection ${\bf{S}}=(S_i)_{i \in Q_0}$ such that $S_i$ is a set of $\dv_i$ unordered complex numbers, we can define 
\begin{equation*}
M_{\bf{S}} = \{x \in M\,|\,\forall i \in Q_0,\, k =1,\ldots,v_i~\exists N>0~\text{s.t}~(c_{i,k}-e_k(S_i))^N=0\}.
\end{equation*}
Then:
\begin{equation*}
\chi_{GT}(M) := \sum_{{\bf{S}}}z^{\bf{S}} \operatorname{dim}M_{\bf{S}},    \end{equation*}
where we treat $z^{\bf{S}}$ as formal parameters. For a finite type Dynkin quiver  $Q$, the Gelfand-Tsetlin character is one way to package  the Frenkel-Reshetikhin $q$-character of a module.
%For short, we will write GT-character whenever we refer to the Gelfand-Tsetlin character.

\subsection{Slant sum of Coulomb branches}

Let $Q^{(1)}$ and $Q^{(2)}$ be quivers and fix dimension and framing vectors $\dv^{(r)},\dw^{(r)}$ for $r \in \{1,2\}$.

\subsubsection{One-dimensional framing}
We now make the following assumption.

\begin{assumption}\label{eq: assumption w}
%The quiver $Q^{(2)}$ has no loops or cycles. 
We have $\dw^{(2)}_{\sspt_{2}}=1$ for some $\sspt_{2} \in Q^{(2)}_{0}$ and $\dw^{(2)}_{i}=0$ otherwise.
\end{assumption}

\begin{example}\label{Ex:minuscule}
This assumption is automatically satisfied for $(Q^{(2)},\dv^{(2)},\dw^{(2)})$ when  $\mu^{(2)} = w\lambda^{(2)}$ with $w$  $\lambda^{(2)}$-minuscule (in the sense of Peterson, see, for example, \cite[Section 2]{stembridge}), provided the Dynkin subdiagram of $Q^{(2)}$ spanned by the vertices 
\begin{equation*}
\{i \in Q_0^{(2)}\,|\, s_i~\text{occurs in a reduced decomposition of}~w\}
\end{equation*}
is connected. In this case one recovers the classical slant-sum construction of Proctor \cite{proctor}.
\end{example}

%\hunter{Now that I am thinking about it again, what's the problem with loops and cycles?}

Choose a vertex $\sspt_1 \in Q_0^{(1)}$ such that $\dv^{(1)}_{\sspt_{1}}=1$. As in Section \ref{sec: slant sum quivers}, we can form the slant sum, which is the quiver gauge theory for the quiver $Q:=Q^{(1)} {}_{\sspt_1}\slantsum_{\sspt_{2}} Q^{(2)}$ with $\dv=\dv^{(1)} {}_{\sspt_1}\slantsum_{\sspt_{2}} \dv^{(2)}$ and $\dw=\dw^{(1)}  {}_{\sspt_1}\slantsum_{\sspt_{2}} \dw^{(2)}$. Recall that $\dv=\dv^{(1)} \sqcup \dv^{(2)}$ and $\dw_i=\dw_i^{(1)}$ for $i \in Q_0^{(1)}$, $\dw_i=\dw^{(2)}_i$ for $i \in Q_0^{(2)} \setminus \{\sspt_{2}\}$, and $\dw_{\sspt_{2}}=0$.

As in \eqref{eq: def of N}, we have ${\bf{N}}^{(1)}$, ${\bf{N}}^{(2)}$, and ${\bf{N}}$ which are acted on by the gauge groups $G_{\dv^{(1)}}$, $G_{\dv^{(2)}}$, and $G_{\dv}$ respectively. These give rise to the corresponding Coulomb branches $\mathcal{M}_{C}^{(1)}$, $\mathcal{M}_{C}^{(2)}$, and  $\mathcal{M}_{C}$ as well as their quantizations $\mathcal{A}^{(1)}$, $\mathcal{A}^{(2)}$, and $\mathcal{A}$.

% Set
% \begin{equation*}
% {\bf{N}}=\bigoplus_{(k \rightarrow l) \in Q_1}\operatorname{Hom}(V_k,V_l) \oplus \bigoplus_{k \in Q_0}\operatorname{Hom}(W_k,V_k),~G_{\dv}:=\prod_{i \in Q_0}\operatorname{GL}_{v_i},
% \end{equation*}
% \begin{equation*}
% {\bf{N}}'=\bigoplus_{(k \rightarrow l) \in Q'_1}\operatorname{Hom}(V'_k,V'_l) \oplus \bigoplus_{k \in Q'_0}\operatorname{Hom}(W'_k,V'_k),~G_{\dv'}:=\prod_{i \in Q_0}\operatorname{GL}_{v_i},
% \end{equation*}
% \begin{equation*}
% {\bf{N}}''=\bigoplus_{(k \rightarrow l) \in Q_1}\operatorname{Hom}(V''_k,V''_l) \oplus \bigoplus_{k \in Q_0}\operatorname{Hom}(W''_k,V''_k),~G_{\dv''}:=\prod_{i \in Q_0}\operatorname{GL}_{v''_i}.
% \end{equation*}

%Consider the Coulomb branches: \begin{equation*}
%\mathcal{M}_C(\dv,\dw),\, \mathcal{M}_C(\dv^{(1)},\dw^{(2)}),\,\mathcal{M}_C(\dv'',\dw'')
%\end{equation*}

Let 
\begin{equation*}
\mathsf{A}_{Q^{(1)}} = \pi_1(G_{\mathsf{v}^{(1)}})^{\wedge},~\mathsf{A}_{Q^{(2)}} = \pi_1(G_{\mathsf{v}^{(2)}})^{\wedge},~\mathsf{A}_{Q} = \pi_1(G_{\mathsf{v}})^{\wedge}
\end{equation*}
be the tori acting naturally on $\mathcal{M}_C^{(1)}$, $\mathcal{M}_C^{(2)}$, and $\mathcal{M}_C$ respectively (see Section \ref{sec: torus action on Coulomb} above).  We will denote by $\mathsf{T}_{Q^{(1)}}$, $\mathsf{T}_{Q^{(2)}}$, and $\mathsf{T}_{Q}$ their products with $\mathbb{C}^\times_\hbar$.

Consider the isomorphism
\begin{equation*}
\varphi\colon G_{\mathsf{v}^{(1)}} \times G_{\mathsf{v}^{(2)}}  \iso G_{\mathsf{v}} 
\end{equation*}
defined as follows.  Recall that $\operatorname{GL}(V_{\sspt_1})=\mathbb{C}^\times$ and for $k \in Q^{(2)}_{0}$ let $\iota_k\colon \operatorname{GL}(V_{\sspt_{1}}) \hookrightarrow \operatorname{GL}(V^{(2)}_{k})$ be the embedding given by $t \mapsto \operatorname{diag}(t,t,\ldots,t)$. 
Then $\varphi(g^{(1)},g^{(2)})_{k}$ is equal to $g_k$ for $k \in Q_0^{(1)}$, and is equal to   $\iota_k(g_{\sspt_{1}})g^{(2)}_{k}$ for $k \in Q_0^{(2)}$. 

Applying $\pi_1(\bullet)^{\wedge}$ to both sides of the isomorphism $\varphi$, we obtain the identification:
\begin{equation*}
\mathsf{A}_{Q} \iso \mathsf{A}_{Q^{(1)}} \times \mathsf{A}_{Q^{(2)}},
\end{equation*}
inducing the identification $\mathsf{T}_{Q} \iso \mathsf{A}_{Q^{(1)}} \times \mathsf{A}_{Q^{(2)}} \times \mathbb{C}^\times_{\hbar}$ and the embedding $\mathsf{T}_{Q} \hookrightarrow 
\mathsf{T}_{Q^{(1)}} \times \mathsf{T}_{Q^{(1)}}$. This embedding defines an action of $\mathsf{T}_{Q}$ on $\mathcal{M}_{C}^{(1)} \times \mathcal{M}_C^{(2)}$ and quantizes to an action on $\mathcal{A}_{C}^{(1)} \otimes \mathcal{A}_{C}^{(2)}$.

The isomorphism $\varphi$ also induces identifications:
\begin{equation*}
H^*_{G_{\mathsf{v}^{(1)}}}(\operatorname{pt}) \otimes H^*_{G_{\mathsf{v}^{(2)}}}(\operatorname{pt}) \simeq H^*_{G_{\mathsf{v}}}(\operatorname{pt})
\end{equation*}
\begin{equation}\label{eq:integrable embedding}
H=H^*_{G_{\mathsf{v}} \times \mathbb{C}^\times}(\operatorname{pt})|_{\hbar=1} \iso H^*_{G_{\mathsf{v}^{(1)}} \times \mathbb{C}^\times}(\operatorname{pt})|_{\hbar=1} \otimes H^*_{G_{\mathsf{v}^{(2)}} \times \mathbb{C}^\times}(\operatorname{pt})|_{\hbar=1} = H^{(1)} \otimes H^{(2)}.
\end{equation}

\begin{proposition}\label{prop_slant_sum_of_coulomb}
There exist $\mathsf{T}_{Q}$-equivariant isomorphisms of Poisson varieties and algebras:
\begin{equation*}
\mathcal{M}_C \simeq \mathcal{M}_C^{(1)} \times \mathcal{M}_C^{(2)},~\mathcal{A} \simeq \mathcal{A}^{(1)} \otimes \mathcal{A}^{(2)}.  
\end{equation*}
This isomorphism is compatible with the identification (\ref{eq:integrable embedding}) of integrable systems.

\end{proposition}
\begin{proof} It follows from the definitions that $\mathcal{M}_C^{(1)} \times \mathcal{M}_C^{(2)}$ is the Coulomb branch for the pair $(G_{\dv^{(1)}} \times G_{\dv^{(2)}},{\bf{N}}^{(1)} \oplus {\bf{N}}^{(2)})$. So, to construct the desired isomorphism it is enough to construct isomorphisms
\begin{equation*}
\varphi\colon G_{\dv^{(1)}} \times G_{\dv^{(2)}} \iso G_{\dv},\quad 
\psi\colon {\bf{N}}^{(1)} \oplus {\bf{N}}^{(2)} \iso {\bf{N}} 
\end{equation*}
such that
\begin{equation*}
\psi((g^{(1)},g^{(2)}) \cdot (n^{(1)},n^{(2)}))=\varphi(g^{(1)},g^{(2)})\cdot \psi(n^{(1)},n^{(2)})~
\end{equation*}
for all $g^{(r)} \in G_{\dv^{(r)}}$ and $n^{(r)} \in {\bf{N}}^{(r)}$, $r \in \{1,2\}$. We have already constructed the isomorphism $\varphi$.

Choose an isomorphism between the one-dimensional spaces $V_{\sspt_{1}}^{(1)} \cong W_{\sspt_{2}}^{(2)}$. Then the identification $\psi$ is the natural identification ${\bf{N}}^{(1)} \oplus {\bf{N}}^{(2)} \iso {\bf{N}}$ (it identifies $\Hom(W_{\sspt_{2}}^{(2)},V_{\sspt_{2}}^{(2)})$ with $\Hom(V_{\sspt_{1}},V_{\sspt_{2}})$ via the identifications $W_{\sspt_{2}}^{(2)}\cong V_{\sspt_{1}}^{(1)}=V_{\sspt_{1}}$ and $V_{\sspt_{2}}^{(2)}=V_{\sspt_2}$).

It follows from the definitions that the maps $\varphi$, $\psi$ satisfy the desired properties. Compatibility with integrable systems and ${\mathsf{T}}_{Q}$-equivariance follows from construction.
\end{proof}

\begin{remark}
Fixing arbitrary pair of characters $\theta^{(1)}\colon G_{\mathsf{v}^{(1)}} \rightarrow \mathbb{C}^\times$, $\theta^{(2)}\colon G_{\mathsf{v}^{(2)}} \rightarrow \mathbb{C}^\times$ and using the identification $\varphi$ we obtain the character $\theta\colon G_{\mathsf{v}} \rightarrow \mathbb{C}^\times$. The same argument as in the proof of Proposition \ref{prop_slant_sum_of_coulomb} above shows that the identification $\mathcal{M}_C \simeq \mathcal{M}_C^{(1)} \times \mathcal{M}_C^{(2)}$ lifts to the $\mathsf{T}_Q$-equivariant identification of partially resolved  Coulomb branches   $\widetilde{\mathcal{M}}_{C,\theta} \simeq \widetilde{\mathcal{M}}_{C,\theta^{(1)}}^{(1)} \times \widetilde{\mathcal{M}}_{C,\theta^{(2)}}^{(2)}$. Similarly, if we have any pair of tori ${\mathsf{A}}^{(1)}, \mathsf{A}^{(2)}$ acting on ${\bf{N}}^{(1)}$, ${\bf{N}}^{(2)}$, then, using the identification $\psi$, we obtain the action of ${\mathsf{A}}^{(1)} \times \mathsf{A}^{(2)}$ on ${\bf{N}}$ and then the isomorphism of Proposition \ref{prop_slant_sum_of_coulomb} can be upgraded to deformations.
\end{remark}

\begin{remark}
The analog of Proposition \ref{prop_slant_sum_of_coulomb} is known to be true on the Higgs side, see \cite[Lemma 8.4]{sing_CM_varieties}. It becomes very natural after applying the Crawley-Boevey trick to our quivers.
\end{remark}

One immediate corollary of Proposition \ref{prop_slant_sum_of_coulomb} is
\begin{corollary}
If Conjecture \ref{conj: nonsingular Coulomb point Higgs} holds for $\mathcal{M}_{C}^{(1)}$, $\mathcal{M}_C^{(2)}$, then it also holds for $\mathcal{M}_C$. If $p^{(i)}$, $i=1,2$ are fixed points of $\mathcal{M}_{C}^{(i)}$, we denote by $p^{(1)}\# p^{(2)}$ the corresponding (unique) fixed point of $\mathcal{M}_C$.
\end{corollary}

\begin{remark}
Proposition \ref{prop_slant_sum_of_coulomb} allows one to extract the character of $T_{p^{(1)}\#p^{(2)}}\mathcal{M}_C$ from the characters of $T_{p^{(i)}}\mathcal{M}_C^{(i)}$, $i=1,2$. 
Proposition \ref{prop_slant_sum_of_coulomb} also allows us to compute the GT-character for an irreducible module $L(p^{(1)}\# p^{(2)})$ over $\mathcal{A}$ corresponding to $p^{(1)}\# p^{(2)}$ (see Section \ref{subsec: conjectures} above) from the GT-characters of irreducible modules for $\mathcal{A}^{(i)}$ corresponding to $p^{(i)}$. %(see \cite[Section 5.3]{Br_Pr_Web} for the discussion of $\Theta$-modules).  %Namely, we have:
%\begin{equation*}
%\chi_{GT}(\Delta(p^{(1)}\# %p^{(2)})) = \chi_{GT}%(\Delta(p^{(1)}))\chi_{GT}(\Delta(p^{(2)})).
%\end{equation*}
\end{remark}

\begin{comment}
\begin{remark}
Note that the assumption \eqref{eq: assumption w} that we put is quite restrictive and the results on the dual side that we prove about vertex functions do not need this assumption. So, it would be very interesting to have some analog of Proposition \ref{prop_slant_sum_of_coulomb} without imposing the assumption ... on the quiver $Q^{(2)}$.     
\end{remark}
\end{comment}

%\vasya{deformations}

%\vasya{example + can we prove Conjecture 1.4 for $Q$ when know it for $Q^{(1)}$ and $Q^{(2)}$?}

Let us now relax Assumption \ref{eq: assumption w}.

\begin{assumption}\label{eq: assumption w relaxed}
%The quiver $Q^{(2)}$ has no loops or cycles. 
We have $\dw^{(2)}_{\sspt_{2}}=1$ for some $\sspt_{2} \in Q^{(2)}_{0}$.
\end{assumption}
As above, we fix the identification $W^{(2)}_{\star_2} \simeq V^{(1)}_{\star_1}\simeq \mathbb{C}$. Identifying $\mathbb{C}^\times \simeq GL(W^{(2)}_{\star_2})$ we obtain the action $\mathbb{C}^\times \curvearrowright {\bf{N}}^{(2)}$. Consider the following embedding:
\begin{equation}\label{eq: embedding w 1}
G_{\mathsf{v}^{(1)}} \times G_{\mathsf{v}^{(2)}} \hookrightarrow G_{\mathsf{v}} \times \mathbb{C}^\times,~(g^{(1)},g^{(2)}) \mapsto (g^{(1)},g^{(2)},g^{(1)}_{\star_1}).
\end{equation}
It follows from the definitions that this embedding intertwines the  actions on ${\bf{N}} = {\bf{N}}^{(1)} \oplus {\bf{N}}^{(2)}$. The cokernel of (\ref{eq: embedding w 1}) identifies with $\mathbb{C}^\times$, i.e., we have the following short exact sequence:
\begin{equation}\label{eq:s.e.s}
1 \rightarrow G_{\mathsf{v}^{(1)}} \times G_{\mathsf{v}^{(2)}} \rightarrow G_{\mathsf{v}} \times \mathbb{C}^\times \rightarrow \mathbb{C}^\times \rightarrow 1
\end{equation}

Applying \cite[Proposition 3.18]{BFNII} we obtain the following standard proposition describing the Coulomb branch $\mathcal{M}_C$ as a Hamiltonian reduction of the product of Coulomb branches.

\begin{proposition}\label{eq: prop coulomb via hamiltonian}
We have isomorphisms:
\begin{equation*}
\mathcal{M}_{C} \simeq (\mathcal{M}_{C}^{(1)} \times \mathcal{M}_C(G_{\dv^{(2)}} \times \mathbb{C}^\times,{\bf{N}}^{(2)}))/\!\!/\!\!/ \mathbb{C}^\times,~\mathcal{A}_C \simeq (\mathcal{A}_C^{(1)} \otimes \mathcal{A}_C(G_{\dv^{(2)}} \times \mathbb{C}^\times,{\bf{N}}^{(2)}))/\!\!/\!\!/ \mathbb{C}^\times.
\end{equation*}
\end{proposition}

\begin{remark}
Note that     $\mathcal{M}_C(G_{\dv^{(2)}} \times \mathbb{C}^\times,{\bf{N}}^{(2)})$ is nothing else but the Coulomb branch for the quiver $'Q^{(2)}$ that is obtained from $Q^{(2)}$ by adding a vertex $\star_1$ to $Q^{(2)}_0$ connected to the vertex $\star_2$ and then putting $\mathsf{v}_{\star_1}=1$, $\mathsf{w}_{\star_1}={\mathsf{w}}_{\star_2}=0$.  
\end{remark}

Let's now compare Propositions \ref{prop_slant_sum_of_coulomb}, \ref{eq: prop coulomb via hamiltonian} above. There is a natural splitting of (\ref{eq:s.e.s}) given by 
\begin{equation}\label{eq:our_splitting}
\mathbb{C}^\times \ni t \mapsto \{1\} \times (\iota_k(t))_{k \in Q_0^{(2)}} \times \{t\} \in G_{\mathsf{v}} \times \mathbb{C}^\times,
\end{equation}
where $\iota_k\colon \mathbb{C}^\times \rightarrow GL(V_k^{(2)})$ is given by $t \mapsto \operatorname{diag}(t,\ldots,t)$. Now, if ${\mathsf{w}}^{(2)}_{i}=0$ for $i \neq \star_2$, then, after the identification $G_{\mathsf{v}} \times \mathbb{C}^\times \simeq G_{\mathsf{v}^{(1)}} \times G_{\mathsf{v}^{(2)}} \times \mathbb{C}^\times$ induced by the splitting (\ref{eq:our_splitting}), the action $G_{\mathsf{v}} \times \mathbb{C}^\times \curvearrowright {\bf{N}}$ factors throught the standard action: 
\begin{equation}
G_{\mathsf{v}} \times \mathbb{C}^{\times} \twoheadrightarrow G_{\mathsf{v}^{(1)}} \times G_{\mathsf{v}^{(2)}}  \curvearrowright {\bf{N}}^{(1)} \oplus {\bf{N}}^{(2)}={\bf{N}}.
\end{equation}
In other words, $\mathbb{C}^\times$ would act {\emph{trivially}} on ${\bf{N}}$.
Indeed, the element  
\begin{equation*}
g := \{1\} \times (\iota_k(t))_{k \in Q_0^{(2)}} \times \{t\} \in G_{\mathsf{v}} \times \mathbb{C}^\times
\end{equation*}
clearly acts trivially on ${\bf{N}}^{(1)}$. Note also that ${\bf{N}}^{(2)}$ is the direct sum of $\operatorname{Hom}(W_{\star_2}^{(2)},V_{\star_2}^{(2)})=\operatorname{Hom}(V_{\star_1}^{(1)},V_{\star_2}^{(2)})$ and the terms of the form $\operatorname{Hom}(V_{k}^{(2)},V_{k'}^{(2)})$ (here we crucially use that $\mathsf{w}_i^{(2)} = 0$  for $i \neq \star_2$). The element $g$ as above acts trivially on both $\operatorname{Hom}(V_{k}^{(2)},V_{k'}^{(2)})$ and $\operatorname{Hom}(W_{\star_2}^{(2)},V_{\star_2}^{(2)})$. 

We conclude that 
\begin{equation*}
\mathcal{M}_C(G_{\mathsf{v}} \times \mathbb{C}^\times,{\bf{N}}) \simeq \mathcal{M}_C(G_{\mathsf{v}^{(1)}} \times G_{\mathsf{v}^{(2)}},{\bf{N}}^{(1)} \oplus {\bf{N}}^{(2)}) \times \mathcal{M}_C(\mathbb{C}^\times,0) = \mathcal{M}_C^{(1)} \times \mathcal{M}_C^{(2)} \times T^*\mathbb{C}^\times
\end{equation*}
so 
\begin{equation*}
\mathcal{M}_C = \mathcal{M}_C(G_{\mathsf{v}} \times \mathbb{C}^\times,{\bf{N}})/\!\!/\!\!/ \mathbb{C}^\times = (\mathcal{M}_C^{(1)} \times \mathcal{M}_C^{(2)} \times T^*\mathbb{C}^\times)/\!\!/\!\!/\mathbb{C}^\times = \mathcal{M}_C^{(1)} \times \mathcal{M}_C^{(2)} 
\end{equation*}
which is exactly Proposition \ref{prop_slant_sum_of_coulomb}.

\begin{remark}
Note that Proposition \ref{eq: prop coulomb via hamiltonian} is a {\emph{very particular}} case of a much more general and useful result \cite[Theorem 4.2]{Nak_S_dual} (this follows from \cite[Example 3.8 (1)]{Nak_S_dual}). We are grateful to Hiraku Nakajima for pointing out this relation as well as for the explanations on the material of this Section. 
%We start with $(Q^{(1)},\mathsf{v}^{(1)},{\mathsf{w}}^{(1)})$, $(Q^{(2)},\mathsf{v}^{(2)},{\mathsf{w}}^{(2)})$ such that $\mathsf{w}_{\star_2}=1$. Let ${\bf{M}}^{(1)}$, ${\bf{M}}^{(2)}$ be the corresponding symplectic representations of $G_{\mathsf{v}^{(1)}}$, $G_{\mathsf{v}^{(2)}}$ respectively. 
%It follows from \cite[Example 3.8 (1)]{Nak_S_dual} that $({\bf{M}}^{(1)})^\vee$, $({\bf{M}}^{(2)})^\vee$ (in the notations of \cite{Nak_S_dual}) are nothing else but $\mathcal{M}^{(1)}$, $\mathcal{M}^{(2)}$ in our notations. Then, \cite[Theorem 4.2]{Nak_S_dual} implies that 
%\begin{equation*}
%=(\mathcal{M}_{Q^{(1)}} \times \mathcal{M}_{Q^{(2)}})
%\end{equation*}
%... We should emphasize again that  \cite[Theorem 4.2]{Nak_S_dual} is much more general statement than a very simple Proposition \ref{prop_slant_sum_of_coulomb} above dealing with ${\mathsf{w}}_{\star_2}^{(2)}=1$. 
As the results of this paper suggest, things become much more interesting when we slant quivers with ${\mathsf{w}}_{\star_2}^{(2)}>1$. In this case our slant sum construction and gluing in \cite[Theorem 4.2]{Nak_S_dual} seem to be different procedures. 
\end{remark}

\subsubsection{General framing}\label{sec:general framing}

For $\dw^{(2)}_{\star_2}>1$, we do not know if $\mathcal{M}_C$, $\mathcal{M}_C^{(1)}$, $\mathcal{M}_C^{(2)}$ are related. 
Nevertheless, the $\hbar=q$ specialized vertex functions restricted to fixed points are expected to equal graded traces of Verma modules over quantized Coulomb branches. So, our Higgs side computation suggests a relation between modules over some twisted versions of the (quantized) Coulomb branches for $(Q,\mathsf{v},\mathsf{w}), (Q^{(1)},\mathsf{v}^{(1)},\mathsf{w}^{(1)})$, and $(Q^{(2)},\mathsf{v}^{(2)},\mathsf{w}^{(2)})$, rather than a relation between the Coulomb branches themselves.

Namely, one can consider the category of so-called Gelfand-Tsetlin modules over the quantized Coulomb branch $\mathcal{A}_Q$, see \cite{webster_GZ}. Let ${\mathsf{A}}_{\mathsf{v}} \subset G_{\mathsf{v}}$ be a maximal torus. It follows from \cite{webster_GZ} (see also \cite[Appendix B]{NW} and \cite{VV} for the geometric treatment) that ``blocks'' of categories of GT-modules over $\mathcal{A}_Q$ are equivalent to categories of (topologically nilpotent) modules over the convolution algebra $\widehat{H}_*^{Z_{G}(\sigma_{\mathsf{v}}) \times \mathsf{A}}(\mathcal{T}_Q^{(\sigma(t),t)} \times_{{\bf{N}}_{\mathcal{K}}^{(\sigma(t),t)}} \mathcal{T}_Q^{(\sigma(t),t)})$. Such a block depends on a choice of a cocharacter $\sigma\colon \mathbb{C}^\times \rightarrow \mathsf{A}_{\mathsf{v}} \times \mathsf{A}$. Here $\mathcal{T}_Q=G_{\mathcal{K}} \times^{G_{\mathcal{O}}} {\bf{N}}_{\mathcal{O}}$. Fixed points $\mathcal{T}_Q^{(\sigma(t),t)}$ can be explicitly computed and turn out to be equal to the disjoint unions of products of the corresponding fixed points on $\mathcal{T}_{Q^{(1)}}$ as well as fixed points on some twisted version of $\mathcal{T}_{Q^{(2)}}$, this is completely parallel to Theorem \ref{thm: qm and slant sums}, where we provide a similar description of torus fixed based quasimaps to the resolved Higgs branch for $Q$. As a corollary of the above, one should be able to recover a branching formula for a graded trace of a GT-module over $\mathcal{A}_Q$ reflecting our Theorem \ref{thm: vertex and slant sums intro} on the Coulomb side. It would be very interesting to work out the details.

\printbibliography

\pagebreak

\noindent
Hunter Dinkins\\
Massachusetts Institute of Technology\\
Cambridge, MA\\
hunte864@mit.edu\\
\\
\noindent
Vasily Krylov\\
Harvard University \\
Cambridge, MA\\
vkrylov@math.harvard.edu \\
\\
\noindent
Reese Lance\\
University of North Carolina\\
Chapel Hill, NC\\
rlance@unc.edu \\
\end{document}